\documentclass[final]{amsart} 
\usepackage[utf8]{inputenc}
\usepackage[T1]{fontenc}
\usepackage{amssymb,amsthm}
\usepackage{mathtools}\mathtoolsset{showonlyrefs}
\usepackage{stmaryrd}
\usepackage{microtype}
\usepackage[inline]{enumitem}
\usepackage{mathrsfs}
\usepackage[bottom=1in]{geometry}
\usepackage{subcaption} 
\usepackage{tikz-cd}
\usepackage{appendix}
\usepackage{dynkin-diagrams}
\pgfkeys{/Dynkin diagram, edge length=1cm}

\usepackage{csquotes} 
\usepackage{hyperref}
\newcommand{\Z}{\mathbb Z}
\newcommand{\Q}{\mathbb Q}
\newcommand{\R}{\mathbb R}
\newcommand{\C}{\mathbb C}
\newcommand{\F}{\mathbb F}
\renewcommand{\O}{\mathcal O}
\newcommand{\OO}{\mathscr O}
\newcommand{\field}{\mathfrak F}

\newcommand{\val}{\operatorname{val}}
\newcommand{\eg}{\textit{e.g.}, }
\newcommand{\ie}{\textit{i.e.}, }
\newcommand{\red}{\mathrm{red}}
\renewcommand{\setminus}{\smallsetminus}
\renewcommand{\emptyset}{\varnothing}

\newcommand{\frakS}{\mathfrak{S}} 
\newcommand{\Satake}{\mathcal{S}} 

\newcommand{\apartment}{\mathscr A}
\newcommand{\building}{\mathscr B}

\newcommand{\hyperplanes}{\mathfrak H}

\newcommand{\X}{\mathrm X}
\newcommand{\aff}{\mathrm{aff}}

\newcommand{\pairing}{\langle\,\cdot\,{,}\,\cdot\,\rangle}
\newcommand{\scalar}{(\,\cdot\,{,}\,\cdot\,)}

\newcommand{\Hecke}{\mathcal H}

\newcommand{\triang}[3]{\begin{pmatrix}
#1 & #2\\ 0 & #3\end{pmatrix}}

\newcommand{\alg}[1]{\mathbf{#1}}
\newcommand{\set}[2]{\left\{#1\,\middle|\,#2\right\}}
\newcommand{\wt}[1]{\widetilde{#1}}

\newcommand{\longtwoheadrightarrow}%
{\longrightarrow\hspace{-1.2em}\rightarrow\hspace{.2em}}
\newcommand{\longhookrightarrow}%
{\lhook\joinrel\relbar\joinrel\rightarrow}

\DeclareMathOperator{\Aff}{Aff}

\DeclareMathOperator{\End}{End}
\DeclareMathOperator{\GL}{GL}
\DeclareMathOperator{\Sp}{Sp}
\DeclareMathOperator{\SL}{SL}
\DeclareMathOperator{\SO}{SO}
\DeclareMathOperator{\SU}{SU}

\DeclareMathOperator{\Hom}{Hom}
\DeclareMathOperator{\id}{id}

\DeclareMathOperator{\Ker}{Ker}
\DeclareMathOperator{\pr}{pr}

\newtheorem{prop}{Proposition}[section]

\newtheorem{thm}[prop]{Theorem}
\newtheorem*{thmintro}{Theorem}
\newtheorem{lem}[prop]{Lemma}
\newtheorem{cor}[prop]{Corollary}
\newtheorem{hyp}[prop]{Hypothesis}
\newtheorem{cartan}[prop]{Cartan decomposition}
\newtheorem{iwasawa}[prop]{Iwasawa decomposition}

\theoremstyle{definition}
\newtheorem*{defn}{Definition}
\newtheorem*{claim*}{Claim}
\newtheorem{claim}{Claim}
\newtheorem{notation}[prop]{Notation}
\newtheorem{algo}[prop]{Algorithm}

\theoremstyle{remark}
\newtheorem*{rmk*}{Remark}
\newtheorem*{ex*}{Example}
\newtheorem{ex}[prop]{Example}
\newtheorem{rmk}[prop]{Remark}
\newtheorem*{notation*}{Notation}

\title[Decomposition of Hecke Polynomials]{On the Decomposition of Hecke
Polynomials over Parabolic Hecke Algebras}
\author{Claudius Heyer}
\address{Mathematisches Institut, Universit\"at M\"unster, Einsteinstra\ss{}e 62, D-48149 M\"unster, Germany}
\email{cheyer@uni-muenster.de}

\begin{document}
\begin{abstract} 
We generalize a classical result of Andrianov on the decomposition of Hecke
polynomials. Let $\mathfrak{F}$ be a non-archimedean local field. For every
connected reductive group $\mathbf{G}$, we give a criterion for when a
polynomial with coefficients in the spherical parahoric
Hecke algebra of $\mathbf{G}(\mathfrak{F})$ decomposes over a parabolic Hecke
algebra associated with a \emph{non-obtuse} parabolic subgroup of $\mathbf{G}$.
We classify the non-obtuse parabolics. This then shows that our decomposition
theorem covers all the classical cases due to Andrianov and Gritsenko. We also
obtain new cases when the relative root system of $\mathbf{G}$ contains factors
of types $E_6$ or $E_7$.
\end{abstract} 
\maketitle
\tableofcontents

\section{Introduction} 
\subsection{Motivation} 
The problem to decompose Hecke polynomials emer\-ged in the theory of Hecke
operators acting on spaces of Siegel modular forms, see, \eg 
\cite{Andrianov.1977,Andrianov.1979,Gritsenko.1984}. One of the principal tasks
is to find and study relations between Fourier coefficients of eigenforms
of Hecke operators and the corresponding eigenvalues. It is instructive to work
through an example to see how decomposing Hecke polynomials helps to find such
relations. Consider the modular group $\Gamma = \SL_2(\Z)$. Recall that a
holomorphic function $f\colon \mathbb{H}\to \C$ on the upper half-plane
$\mathbb{H} = \set{z\in \C}{\Im(z)>0}$ is called a \emph{modular form of weight
$k$} if for all $\gamma = \begin{pmatrix}a&b\\c&d\end{pmatrix} \in \Gamma$ and
all $z\in \mathbb{H}$ it satisfies
\begin{equation}\label{eq:modularform}
f(z) = (f|\gamma)(z) \coloneqq (cz+d)^{-k}
f\left(\frac{az+b}{cz+d}\right),
\end{equation}
and if it admits a Fourier expansion of the form $f(z) = \sum_{j=0}^\infty
\alpha_f(j)\cdot e^{2\pi ijz}$. Denote $\mathfrak{M}_k$ the $\C$-vector space of
modular forms of weight $k$. Let $S$ be the set of $2\times 2$-matrices with
integral entries and positive determinant. Then the algebra of Hecke operators
$\Hecke\coloneqq \Hecke_\C(\Gamma, S)$ naturally acts on $\mathfrak{M}_k$: A
double coset $(\Gamma g\Gamma)\in \Hecke$ acts on $f$ via $f|(\Gamma g\Gamma)
\coloneqq \sum_{\Gamma \gamma \in \Gamma\backslash\Gamma g\Gamma} f|\gamma$. 
If $f$ is a Hecke eigenform, we write $\lambda_f\colon \Hecke\to
\C$ for the corresponding eigenvalue. Then $f$ is a Hecke eigenform if and only
if $\alpha_{f|T}(j) = \lambda_f(T)\cdot \alpha_f(j)$, for all $T\in \Hecke$,
$j\in \Z_{\ge0}$.

Fix a prime number $p$ and consider the Hecke polynomial
\[
Q_p(t) = 1 -T_1t + pT_2t^2,\quad \text{where $T_1 = (\Gamma\triang p01 \Gamma)$,
$T_2 = (pE_2\Gamma) \in \Hecke$,}
\]
where $E_2$ is the $2\times 2$ identity matrix. 
There is a natural embedding of $\Hecke$ into the parabolic Hecke algebra
$\Hecke^0 \coloneqq \Hecke_{\C}(\Gamma_0,S_0)$, where $\Gamma_0$ (resp. $S_0$)
is the subgroup of upper triangular matrices in $\Gamma$ (resp. $S$). For
example, one has $T_1 =
T^+_1 + T^-_1$ in $\Hecke_\C(\Gamma_0,S_0)$, where $T^+_1 = (\Gamma_0 \triang
p01\Gamma_0)$ and $T^-_1 = (\Gamma_0\triang10p\Gamma_0)$. 
Over the ring $\Hecke_{\C}(\Gamma_0,S_0)$ the polynomial $Q_p(t)$ decomposes as
\begin{equation}\label{eq:Qp-dec}
Q_p(t) = (1 -T^+_1t) \cdot (1 - T^-_1t).
\end{equation}
Right multiplication with the inverse power series of $1 - T^-_1t$ yields
\begin{equation}\label{eq:Qp}
Q_p(t)\cdot\sum_{l=0}^\infty (T^-_1)^lt^l =
1-T^+_1t\qquad \text{in $\Hecke^0\llbracket t\rrbracket$.}
\end{equation}
Note that $\Hecke^0$ acts naturally on the space $\mathfrak{M}^0_k$ of
holomorphic functions $f\colon \mathbb{H}\to \C$ satisfying
\eqref{eq:modularform} for all $\gamma\in \Gamma_0$ and admitting a Fourier
expansion as above. Then $\mathfrak{M}_k \subseteq \mathfrak{M}^0_k$ and
$\Hecke^0\llbracket t\rrbracket$ acts naturally on
$\mathfrak{M}^0_k\llbracket t\rrbracket$. Given $f\in \mathfrak{M}^0_k$, one
computes $\alpha_{f|T^+_1}(j) = \alpha_f(j/p)$ and $\alpha_{f|T^-_1}(j) =
p^{1-k}\alpha_f(pj)$, for all $j\ge0$. (Here, we define $\alpha_f(j/p) = 0$ if
$p\nmid j$.)

Let now $f\in \mathfrak{M}_k$ be a Hecke eigenform and consider the complex
polynomial
\[
Q_{p,f}(t) = 1 -
\lambda_f(T_1)\cdot t + p\cdot \lambda_f(T_2)\cdot t^2 \in \C[t].
\]
Letting~\eqref{eq:Qp} act on $f$, we obtain on the level of Fourier coefficients
the relations
\[
Q_{p,f}(t)\cdot \sum_{l=0}^{\infty} p^{l(1-k)}\alpha_f(p^lj)t^l = \alpha_f(j) -
\alpha_f(j/p)t\qquad \text{in $\C\llbracket t\rrbracket$, for each $j\ge0$.}
\]

This method of decomposing a Hecke polynomial over a parabolic Hecke algebra
proved to be very fruitful in the more general context of Siegel modular forms.
Andrianov proved a general decomposition theorem of type~\eqref{eq:Qp-dec} in
the context of Siegel modular forms, cf. \cite{Andrianov.1977}. In this case the
modular group $\SL_2(\Z)$ is replaced by $\Sp_{2n}(\Z)$, for some
$n\in\Z_{\ge1}$, and one considers certain holomorphic functions on the Siegel
upper half-space $\mathbb{H}_n$. The subgroup $\Gamma_0$ of upper triangular
matrices is replaced by the ``Siegel parabolic'' in $\Sp_{2n}(\Z)$, that is, the
subgroup of matrices whose lower left quadrant is zero.

It is then natural to ask whether a decomposition of type~\eqref{eq:Qp-dec} also
holds for more general groups. Since the problem is local in nature, one may
replace $\Z$ with the ring of integers of a non-archimedean local field
$\field$. In this context, Gritsenko proved a decomposition theorem for
$\GL_n(\field)$ (and all parabolics) \cite{Gritsenko.1988,Gritsenko.1992} and
for the classical groups $\Sp_{2n}(\field)$, $\SU_n(\field)$, and
$\SO_n(\field)$ (for the parabolics fixing a line in the standard
representation) \cite{Gritsenko.1990}.

The main result in \cite{Gritsenko.1992} found an application in the theory of
$p$-adic $L$-functions, where it was recently used by Januszewski
\cite{Januszewski.2014} to define a projection map in order to obtain
simultaneous eigenforms for certain Hecke operators. It is therefore reasonable
to hope that a decomposition theorem for more general reductive groups will have
applications in the theory of $p$-adic $L$-functions.

The aim of this paper is to generalize the theory developed by Andrianov
\cite{Andrianov.1977} to the group $G$ of $\field$-rational points of a
connected reductive $\field$-group. 
\subsection{Main results} 
Let $\field$ be a non-archimedean local field of residue characteristic $p>0$.
Let $\alg G$ be a connected reductive group over $\field$, let $\alg B$ be a
minimal parabolic $\field$-subgroup of $\alg G$ with Levi decomposition $\alg B
= \alg Z\alg U$. In this article a parabolic subgroup of $\alg G$ is a standard
parabolic $\field$-subgroup with respect to $\alg B$.
Fix a special parahoric subgroup $K$ of $G\coloneqq \alg G(\field)$
corresponding to a special point in the apartment determined by $\alg Z$. Then
$G = KB$, where $B\coloneqq \alg B(\field)$. For any subgroup $X\subseteq G$ we
put $K_X = K\cap X$. Let $\alg P$ be a parabolic subgroup of $\alg G$
and put $P \coloneqq \alg P(\field)$. Let $R$ be a commutative
$\Z[1/p]$-algebra, considered as a ring of coefficients.

The Hecke ring $\Hecke_R(K_Z,Z)$, where $Z\coloneqq \alg Z(\field)$, identifies
with the group algebra $R[Z/K_Z]$, and there are natural embeddings
$\Hecke_R(K,G)\subseteq \Hecke_R(K_P,P)\subseteq \Hecke_R(K_B,B)$. There is a
natural algebra homomorphism
\[
\Theta^B_Z\colon \Hecke_R(K_B,B) \longrightarrow R[Z/K_Z]
\]
induced by the canonical projection map $B\to Z$. The restriction of
$\Theta^B_Z$ to $\Hecke_R(K_G,G)$ is called the (unnormalized) \emph{Satake
homomorphism}. It is well-known that $\Hecke_R(K,G)$ is commutative. Besides
$\Hecke_R(K,G)$, the parabolic Hecke algebra $\Hecke_R(K_P,P)$ contains another
commutative algebra $C^-_P$, which is constructed as the centralizer of a
certain element of $\Hecke_R(K_P,P)$.

Consider the $R$-submodule $\OO^-_P \coloneqq \Hecke_R(K,G).C^-_P$ of
$\Hecke_R(K_P,P)$. In order to develop a reasonable theory one needs to make
the following assumption:
\begin{hyp}\label{hyp-intro} 
The restriction of $\Theta^B_Z$ to $\OO^-_P$ is injective. 
\end{hyp} 

It should be remarked that there is no example known where
Hypothesis~\ref{hyp-intro} fails. There is one maximal parabolic subgroup in
$\Sp_{6}(\Q_p)$ for which we do not know whether Hypothesis~\ref{hyp-intro} is
satisfied.
One can show (see Proposition~\ref{prop:hyp-cap}) that, if
Hypothesis~\ref{hyp-intro} is satisfied for two
parabolics $\alg P$ and $\alg Q$, then it is also satisfied for $\alg P\cap \alg
Q$. Hence, it would suffice to verify Hypothesis~\ref{hyp-intro} for every
maximal parabolic subgroup of $\alg G$. For practical reasons we work with an
equivalent form of Hypothesis~\ref{hyp-intro}, see Hypothesis~\ref{hyp} on
p.~\pageref{hyp}. We prove:

\begin{thmintro}[Theorem~\ref{thm:decomp}] 
Assume that Hypothesis~\ref{hyp-intro} is satisfied. 
Let $d(t) \in \Hecke_R(K,G)[t]$ be a polynomial such that $d^{\Theta^B_Z}(t)$,
the polynomial obtained by applying $\Theta^B_Z$ to the coefficients of $d(t)$,
decomposes in $R[Z/K_Z][t]$ as
\[
d^{\Theta^B_Z}(t) = \wt f(t)\cdot \wt g(t)
\]
such that $\wt g(t)$ has coefficients in $\Theta^B_Z(C^-_P)$ with constant term
$1$. Then there exist polynomials $f(t), g(t)\in \Hecke_R(K_P,P)[t]$ with the
following properties:
\begin{itemize}
\item $\deg f(t) = \deg\wt f(t)$ and $f^{\Theta^B_Z}(t) = \wt f(t)$;
\item $\deg g(t) = \deg\wt g(t)$ and $g^{\Theta^B_Z}(t) = \wt g(t)$;
\item $d(t) = f(t)\cdot g(t)$ in $\Hecke_R(K_P,P)[t]$.
\end{itemize}
\end{thmintro} 

The proof is merely a straightforward extension of the arguments in
\cite{Andrianov.1977}. 
However, it is in general very hard to decide for which $\alg P$
Hypothesis~\ref{hyp-intro} is satisfied. The main contribution of this paper is
to single out a class of maximal parabolic subgroups in $\alg G$ for which this
hypothesis holds. We make the following important definition:

\begin{defn}[See \S\ref{sec:non-obtuse}] 
Let $\alg P$ be a maximal parabolic subgroup of $\alg G$ with unipotent radical
$\alg U_{\alg P}$. Then $\alg P$ is called \emph{non-obtuse} if any two relative
roots that occur in $\alg U_{\alg P}$ span a non-obtuse angle.
\end{defn} 

Our main result is then:
\begin{thm}[Theorem~\ref{thm:hyp}]\label{intro-thm:hyp} 
Assume that $\alg P$ is a non-obtuse parabolic subgroup of $\alg G$. Then
Hypothesis~\ref{hyp-intro} is satisfied for $\alg P$.
\end{thm} 

The obvious question then is whether there exist non-obtuse parabolics and if
they can be classified. In Proposition~\ref{prop:classification} we achieve a
complete classification of non-obtuse parabolic subgroups and formulate several
equivalent conditions. The maximal parabolic subgroups correspond bijectively to
the vertices in the Dynkin diagram of the relative root system of $\alg G$, and
hence it makes sense to say when a vertex of the Dynkin diagram is non-obtuse.
The classification shows that all vertices in type $A_n$, the terminal vertices
in types $B_n$, $C_n$, and $D_n$, two terminal vertices in type $E_6$, and one
terminal vertex in type $E_7$ are non-obtuse. It also shows that there are no
non-obtuse vertices in types $E_8$, $F_4$, and $G_2$, cf.
Figure~\ref{fig:class} on p.~\pageref{fig:class}. 

The proof of Theorem~\ref{intro-thm:hyp} requires us to investigate
intersections of Cartan and Iwasawa double cosets. This problem is well-known
and arises, for example, in the study of the Satake homomorphism, but here it is
of a different flavor. More precisely, let $z,z'\in Z$ be such that $Uz'K
\cap KzK \neq\emptyset$. It is well-understood how $z'$ and $z$ relate. However,
so far almost nothing is known about the $u\in U$ for which $uz'\in KzK$. 

To state our main technical result in this direction, let $\varphi$ be the point
in the (adjoint) Bruhat--Tits building of $G$ corresponding to $K$. By
assumption, $\varphi$ lies in
the apartment $\apartment$ corresponding to a torus $\alg T\subseteq \alg Z$
which is maximal $\field$-split in $\alg G$. By definition, $\varphi$
defines valuations, denoted $\varphi_\alpha$, on the root groups
$U_\alpha$. There is a canonical homomorphism $\nu\colon Z\to V$ into the
underlying $\R$-vector $V$ space of $\apartment$ containing the coroots
with respect to $\alg T$. Fix a strictly positive element $a\in Z$ so that
$\langle \alpha, \nu(a)\rangle < 0$ for all simple roots $\alpha$
(with respect to $\alg B$). Choose the representative $z$ of $KzK$ such that
$z\cdot (K\cap U)\cdot z^{-1} \subseteq K\cap U$, where $U\coloneqq \alg
U(\field)$. Denote $U_P \coloneqq \alg
U_{\alg P}(\field)$ the group of $\field$-points of the unipotent radical of
$\alg P$. We prove the following technical result, which might be of
independent interest:

\begin{thm}[Theorem~\ref{thm:main}]\label{intro-thm:main} 
Assume that $\alg P$ is non-obtuse and that $-\nu(az)$ is a sum of simple
coroots. Let $z'\in Z$ and $u\in U_P$ with $uz'\in KzK$. Then one has $az'\cdot
(K\cap U_P)\cdot (az')^{-1} \subseteq K\cap U_P$ and $aua^{-1}
\in K$.
\end{thm} 

The condition $aua^{-1}\in K$ is the difficult part of the theorem and can be
interpreted as follows: write $u = u_1\dotsm u_r$, for certain $u_i\in
U_{\alpha_i}$. Then $\varphi_{\alpha_i}(u_i) \ge \langle
\alpha_i,\nu(a)\rangle$ for all $i$, that is, we obtain a lower bound for the
valuations of the $u_i$. In order to obtain this bound, we describe an
algorithm, see \S\ref{sec:algo}, which produces a sequence of left cosets
$u_0z_0K,\dotsc, u_rz_rK$ in $KzK$ (with $u_i\in U$, $z_i\in Z$) such that
$u_0z_0 = uz'$ and $u_r=1$ (so that $\nu(z_r)$ lies in the orbit of $\nu(z)$
under the action of the finite Weyl group). As a byproduct, by a careful
analysis, we can estimate the valuations of the root group elements $u_i$.

Finally, it should be mentioned that there is another possible approach to study
intersections of Cartan and Iwasawa double cosets, which we do not follow here.
For $p$-adic Chevalley groups D\k{a}browski describes in \cite{Dabrowski.1994}
the intersections of the form $U\tau I \cap I\sigma I$ in terms of ``good
subexpressions'', where $I$ is an Iwahori subgroup and $\tau,\sigma$ are
elements of the affine Weyl group. By adapting the methods of \cite{Lansky.2001}
it seems plausible that one could in this way explicitly describe the
intersections $Uz'K \cap KzK$. 
\subsection{Structure of the paper} 
In \S\ref{sec:prelim} we fix notations (\S\ref{subsec:notation}) and recall some
notions about reductive groups
(\S\ref{subsec:apartment}--\S\ref{subsec:Iwahori-Weyl}). In
\S\ref{subsec:positive}--\S\ref{subsec:Hecke} we discuss positive elements, the
Cartan and Iwasawa decompositions, and abstract Hecke rings.

In \S\ref{sec:non-obtuse} we study non-obtuse parabolics. In
Proposition~\ref{prop:classification} we classify non-obtuse parabolic subgroups
and provide equivalent characterizations.

In \S\ref{sec:algo} we present the Algorithm~\ref{algo}. Although it will not be
used in the sequel it is worth to mention that it always terminates
(Proposition~\ref{prop:algo}). Our main technical result is
Theorem~\ref{thm:main}.

Finally, in \S\ref{sec:decomp} we develop the theory leading to the
decomposition Theorem~\ref{thm:decomp}. The \S\ref{subsec:parHecke} introduces
parabolic Hecke algebras and defines the unnormalized version of the Satake
homomorphism. In \S\ref{subsec:twisted} we give another presentation of the
twisted action due to Henniart--Vign\'eras \cite{Henniart-Vigneras.2015}. Then
in \S\ref{subsec:Satake-iso} we translate the main theorem of
\cite{Henniart-Vigneras.2015} into our context. In \S\ref{subsec:centralizer} we
recall the commutative algebra $C^+_P$. In \S\ref{subsec:example} we work out an
explicit example of a parabolic Hecke algebra. We provide an explicit
presentation of the parabolic Hecke algebra attached to $\GL_2(\field)$ in terms
of generators and relations. The straightforward proof is given in an appendix.
We also explicitly compute the Satake homomorphism for illustrative purposes. 
Given a strictly positive element $a_P$, we construct in \S\ref{subsec:decomp} a
certain Hecke polynomial $\chi_{a_P}(t)$. With Hypothesis~\ref{hyp} we impose
that $(K_Pa_P)$ is a ``left root'' of $\chi_{a_P}(t)$. This hypothesis is
crucial for proving the decomposition theorem. Theorem~\ref{thm:hyp} shows that
this hypothesis is satisfied provided that the parabolic $\alg P$ is non-obtuse.
Proposition~\ref{prop:hyp-equiv} lists several conditions which are equivalent
to Hypothesis~\ref{hyp}. Proposition~\ref{prop:hyp-cap} shows that, in
principle, it would suffice to verify Hypothesis~\ref{hyp} for \emph{maximal}
parabolics.
\subsection{Acknowledgments} 
This article constitutes a part of my doctoral dissertation \cite{Heyer.2019}
written at the Humboldt University in Berlin. I want to express my deep
gratitude to my advisor Elmar Gro\ss{}e-Kl\"onne for his support. I also thank
Peter Schneider for inviting me to present these results in the Mittagsseminar
at the University of M\"unster. My thanks also goes to the organizers of the
conference ``Representation Theory and $D$-Modules'', held in June 2019 in
Rennes, for giving me the chance to present a poster about this research.
Finally, I thank the anonymous referee for a careful reading of the paper and for
numerous remarks, questions, and suggestions. During
the write-up of this article I was funded by the University of M\"unster and
Germany's Excellence Strategy EXC 2044 390685587, Mathematics M\"unster:
Dynamics--Geometry--Structure.
\section{Preliminaries}\label{sec:prelim} 
\subsection{Notations}\label{subsec:notation} 
We fix a locally compact non-archimedean field $\field$ with residue field
$\F_q$ of characteristic $p$ and normalized valuation $\val_\field\colon
\field\to \Z\cup \{\infty\}$.

If $\alg H$ is an algebraic group defined over $\field$, we denote by the
corresponding lightface letter $H\coloneqq \alg H(\field)$ its group of
$\field$-rational points. The topology on $\field$ makes $H$ into a topological
group.

Let $\alg G$ be a connected reductive group defined over $\field$. We choose a
maximal $\field$-split torus $\alg T$ in $\alg G$ and write $\X^*(T)$ (resp.
$\X_*(T)$) for the group of algebraic $\field$-characters (resp. algebraic
$\field$-cocharacters) of $\alg T$.

We denote by $\alg Z\coloneqq \alg
Z_{\alg G}(\alg T)$ the centralizer and by $\alg N\coloneqq \alg N_{\alg G}(\alg
T)$ the normalizer of $\alg T$ in $\alg G$. We call $W_0\coloneqq N/Z$ the
finite Weyl group of $\alg G$.

The (relative) root system of $(\alg G,\alg T)$ is denoted by $\Phi \subseteq
\X^*(T)$; it need not be reduced if $\alg G$ is non-split. The finite Weyl group
$W_0$ identifies with the Weyl group of the root system $\Phi$. We denote
\[
\Phi_{\red} = \set{\alpha\in\Phi}{\alpha/2 \notin \Phi}
\]
the subroot system of reduced roots. The set of coroots is denoted $\Phi^\vee
\subseteq X_*(T)$.

We consider the
root group $\alg U_\alpha$ attached to $\alpha\in \Phi$. Then $\alg U_{2\alpha}
\subseteq \alg U_\alpha$ whenever $\alpha,2\alpha\in \Phi$.

We fix a minimal parabolic $\field$-subgroup $\alg B$ of $\alg G$ containing
$\alg T$. It then admits a Levi decomposition
\[
\alg B = \alg U\alg Z.
\]
This choice fixes a system of positive roots $\Phi^+$ in $\Phi$ (resp. positive
coroots $(\Phi^\vee)^+$ in $\Phi^\vee$), and the
unipotent radical $\alg U$ of $\alg B$ decomposes as
\[
\alg U = \prod_{\alpha\in \Phi^+_{\red}} \alg U_\alpha,
\]
where $\Phi^+_{\red} = \Phi_\red\cap \Phi^+$ is the set of reduced positive
roots.

All parabolic subgroups are taken to be standard with respect to $\alg B$.
\subsection{The standard apartment}\label{subsec:apartment} 
If $\alg C$ denotes the connected center of $\alg G$, we consider the
finite-dimensional $\R$-vector space
\[
V \coloneqq \R\otimes_{\Z} \bigl(\X_*(T)/\X_*(C)\bigr).
\]
We view the set of coroots $\Phi^\vee$ as a subset of $V$ via the natural map.
Note that $\Phi^\vee$ generates $V$ as an $\R$-vector space. On $V$ there is the
following partial ordering: given $v,w\in V$, we write
\[
v \le w
\]
if $w-v$ is a linear combination of simple coroots with non-negative
coefficients.

The conjugation action of $W_0$ on $\alg T$ induces an action on $V$ such that
the natural pairing
\[
\pairing\colon V^*\times V \longrightarrow \R
\]
is non-degenerate and $W_0$-equivariant. Here, $V^* = \Hom_\R(V,\R)$ denotes
the $\R$-linear dual of $V$. We view $\Phi$ as a generating subset of
$V^*$. We fix a $W_0$-invariant scalar product $\scalar$ on $V$ so that $V$
becomes
a Euclidean vector space. We denote $\lVert\,\cdot\,\rVert$ the norm induced by
$\scalar$. The scalar product $\scalar$ induces a $W_0$-invariant scalar product
on $V^*$ which we again denote $\scalar$.

By \cite[1.1.13]{Bruhat-Tits.1984} the tuple $(Z,(U_\alpha)_{\alpha\in \Phi})$
is a generating root group datum of $G$ in the sense of
\cite[(6.1.1)]{Bruhat-Tits.1972}. In particular, this means that the root groups
$U_\alpha$ satisfy the following condition:
\begin{enumerate}[label=(DR\arabic*),start=2]
\item\label{DR2} For all $\alpha,\beta\in \Phi$, the commutator group
$(U_\alpha,U_\beta)$ is contained in the group generated by the $U_{n\alpha +
m\beta}$, where $n,m\in \Z_{>0}$ are such that $n\alpha+m\beta\in \Phi$.
\end{enumerate}

The standard apartment $\apartment$ in the adjoint building $\building(G)$
of $G$ is an affine space under $V$ consisting of certain valuations
\cite[(6.2.1)]{Bruhat-Tits.1972} of $(Z, (U_\alpha)_{\alpha\in \Phi})$. 
Since valuations will be instrumental
later on, we will recall their definition. For the moment let $L_\alpha$, for
$\alpha\in \Phi$, be the subgroup generated by $U_\alpha$, $Z$, and
$U_{-\alpha}$, and put 
\begin{equation}\label{eq:M_alpha}
M_\alpha \coloneqq \set{x\in L_\alpha}{\text{$xU_\alpha x^{-1} = U_{-\alpha}$
and $x U_{-\alpha} x^{-1} = U_\alpha$}} \subseteq N.
\end{equation}
Then $M_\alpha$ is a left and right coset under $Z$ with image $\{s_\alpha\}$ in
$W_0$. We record the following useful lemma:

\begin{lem}[{\cite[(6.1.2) (2)]{Bruhat-Tits.1972}}]\label{lem:Ua} 
Let $\alpha\in \Phi$ and $u\in U_{\alpha}^*\coloneqq U_{\alpha} \setminus\{1\}$.
Then there exists a unique triple $(u',m(u),u'')\in U_{-\alpha}\times G\times
U_{-\alpha}$ such that $u = u'm(u)u''$, $m(u)U_{-\alpha}m(u)^{-1} = U_\alpha$,
and $m(u)U_\alpha m(u)^{-1} = U_{-\alpha}$. 
Moreover, one has $m(u)\in M_\alpha$ and $u', u''\neq 1$.
\end{lem} 

\begin{defn}[{\cite[(6.2.1)]{Bruhat-Tits.1972}}] 
A \emph{valuation} on $(Z,(U_\alpha)_{\alpha\in\Phi})$ is a tuple $\psi =
(\psi_\alpha)_{\alpha\in \Phi}$ of functions $\psi_\alpha\colon U_\alpha\to
\R\cup \{\infty\}$ satisfying the following conditions:
\begin{enumerate}[label=(V\arabic*),start=0]
\item\label{V0} The image of $\psi_\alpha$ contains at least three elements;
\item\label{V1} For each $\alpha\in\Phi$ and $r\in\R\cup \{\infty\}$ the set
$U_{\alpha,r} = \psi_\alpha^{-1}([r,\infty])$ is a subgroup of $U_\alpha$, and
$U_{\alpha,\infty} = \{1\}$;
\item\label{V2} For each $\alpha\in \Phi$ and each $m\in M_\alpha$ the function
\[
U_{-\alpha}^* = U_{-\alpha}\setminus\{1\} \longrightarrow \R,\qquad
u\longmapsto \psi_{-\alpha}(u) - \psi_\alpha(mum^{-1})
\]
is constant.
\item\label{V3} Given $\alpha,\beta\in\Phi$ with $\beta\notin \R_{<0}\alpha$,
and $r,s\in \R$, the commutator group $(U_{\alpha,r}, U_{\beta,s})$ is contained
in the group generated by the groups $U_{n\alpha + m\beta, nr + ms}$, for
$n,m\in\Z_{>0}$ such that $n\alpha + m\beta\in \Phi$;

\item\label{V4} If $\alpha,2\alpha\in\Phi$, then $\psi_{2\alpha}$ is the
restriction of $2\psi_\alpha$ to $U_{2\alpha} \subseteq U_\alpha$;

\item\label{V5} Given $\alpha\in \Phi$, $u\in U_\alpha$, and $u',u''\in
U_{-\alpha}$ such that $u'uu''\in M_{\alpha}$, one has $\psi_{-\alpha}(u') =
\psi_{-\alpha}(u'') = -\psi_\alpha(u)$.
\end{enumerate}
\end{defn} 

The space of valuations on $(Z,(U_\alpha)_{\alpha\in\Phi})$ admits the following
two actions \cite[(6.2.5)]{Bruhat-Tits.1972}:
\begin{enumerate}[label=--] 
\item Given a valuation $\psi = (\psi_\alpha)_{\alpha\in\Phi}$ and $v\in V$,
the tuple
\[
\psi+v = \bigl(\psi_\alpha + \langle \alpha,v\rangle\bigr)_{\alpha\in\Phi}
\]
is again a valuation.

\item Let $\psi = (\psi_\alpha)_{\alpha\in\Phi}$ be a valuation and $n\in N$.
Denote $w$ the image of $n$ under the canonical projection $N\to N/Z = W_0$.
We obtain a new valuation $n.\psi = ((n.\psi)_\alpha)_{\alpha\in\Phi}$
defined by
\[
(n.\psi)_\alpha(u) = \psi_{w^{-1}(\alpha)}(n^{-1}un),\qquad \text{for all $u\in
U_\alpha$.}
\]
In this way, the group $N$ acts on the space of valuations.
\end{enumerate} 

By \cite[5.1.20 Theor\`eme and 5.1.23 Proposition]{Bruhat-Tits.1984} there
exists a valuation 
\[
\varphi = (\varphi_\alpha\colon U_\alpha\to \R\cup\{\infty\})_{\alpha\in\Phi}
\quad\text{of $(Z,(U_\alpha)_{\alpha\in \Phi})$}
\]
which is discrete \cite[(6.2.21)]{Bruhat-Tits.1972}, special
\cite[(6.2.13)]{Bruhat-Tits.1972} and compatible with the valuation
$\val_\field$ \cite[4.2.8 D\'efinition]{Bruhat-Tits.1984}. 
The standard apartment $\apartment$ is the Euclidean affine space under $V$
given by
\[
\apartment = \set{\varphi + v}{v\in V}.
\]
The action of $N$ restricts to an action on $\apartment$ by Euclidean affine
automorphisms and the subgroup $Z$ acts by translations
\cite[(6.2.10)]{Bruhat-Tits.1972}. More concretely, it follows from \ref{V2}
that there exists a unique group homomorphism $\nu\colon Z \to V$ such that
$z.\varphi = \varphi + \nu(z)$, that is,
\begin{equation}\label{eq:nu} 
\varphi_{\alpha}(z^{-1}uz) = \varphi_\alpha(u) + \langle \alpha, \nu(z)\rangle,
\qquad \text{for all $u\in U_\alpha$, all $\alpha\in\Phi$.}
\end{equation} 
The fact that $\varphi$ is compatible with $\val_\field$ then expresses the
condition $\langle \chi|_\alg T, \nu(z)\rangle =
-\val_\field\bigl(\chi(z)\bigr)$, for all $\chi\in \X^*(Z)$. 

The affine action of $N$ on $\apartment$ induces a linear action of $W_0 = N/Z$
on $V$, obtained by composing the action map $N\to \Aff(\apartment)$ (where
$\Aff(\apartment)$ denotes the group of affine automorphisms of $\apartment$)
with the canonical projection $\Aff(\apartment) \to \GL_{\R}(V)$. This action
coincides with the natural action of $W_0$ on $V$.

Given $\alpha\in V^*$ and $r\in \R$, we consider the hyperplane
\[
H_{\alpha,r}\coloneqq \set{\varphi + v \in \apartment}{\langle \alpha,v\rangle +
r = 0}
\]
and put
\[
\hyperplanes \coloneqq \set{H_{\alpha,r}}{\text{$\alpha\in \Phi_{\red}$ and
$r\in \varphi_\alpha(U^*_\alpha)$}}.
\]
Then $N$ acts on $\hyperplanes$ via
\[
n.H_{\alpha,r} = H_{w(\alpha), r - \langle w(\alpha), n.\varphi -
\varphi\rangle},\quad \text{for $n\in N$ with image $w\in W_0$.}
\]
The groups $U_{\alpha,r} = \varphi_\alpha^{-1}([r,\infty])$, for $r\in \R$, are
a neighborhood basis of $1\in U_\alpha$ consisting of compact open subgroups,
and we have
\[
nU_{\alpha,r}n^{-1} = U_{w(\alpha), r - \langle w(\alpha), n.\varphi-
\varphi\rangle},\quad \text{for $n\in N$ with image $w\in W_0$.}
\]
\subsection{The associated reduced root system}\label{subsec:reducedroot} 
We denote by $S(\hyperplanes)$ the set of orthogonal reflections $s_H$ through
$H\in \hyperplanes$. Conversely, we denote $H_s$ the hyperplane in $\apartment$
fixed by $s\in S(\hyperplanes)$. This exhibits a canonical bijection
$\hyperplanes\cong S(\hyperplanes)$.

The group $W^\aff$ generated by $S(\hyperplanes)$ (inside the group of
affine automorphisms of $\apartment$) is called the \emph{affine Weyl group} of
$G$. The stabilizer of $\varphi$ in $W^\aff$ identifies with $W_0$, since
$\varphi$ is special. This yields a semidirect product decomposition
\[
W^\aff = (W^\aff\cap V) \rtimes W_0,
\]
and $W^\aff\cap V$ is generated by the translations $r\alpha^\vee$, for
$\alpha\in \Phi_\red$ and $r\in \varphi_\alpha(U^*_\alpha)$
\cite[(6.2.19)]{Bruhat-Tits.1972}. In particular, $W^\aff\cap V$ is a lattice of
rank $\dim_\R V$ in $V$. Now, \cite[Ch.\,VI, \S2, no.\,5,
Prop.\,8]{Bourbaki.1981} shows that there exists a unique reduced root system
\[
\Sigma \subseteq V^*
\]
such that $W^\aff$ is the affine Weyl group of $\Sigma$. This means that
$W^\aff$ coincides with the group generated by the reflections $s_{\alpha,k}$,
for $(\alpha,k)\in \Sigma^\aff \coloneqq \Sigma\times \Z$, defined by
\[
s_{\alpha,k}(x) = x - (\langle\alpha, x-\varphi\rangle + k)\cdot
\alpha^\vee,\qquad \text{for $x\in \apartment$.}
\]
We write simply $s_\alpha$ instead of $s_{\alpha,0}$ and view it as an element
of $W_0$.

By \cite[Lemma\,I.2.10]{Schneider-Stuhler.1997},
\[
\varphi_\alpha(U^*_\alpha) = \epsilon_\alpha^{-1}\Z,\qquad\text{for
$\alpha\in\Phi$,}
\]
is a group, where $\epsilon_\alpha\in \Z_{>0}$ is a natural number which is even
whenever $2\alpha\in\Phi$. We obtain a surjective map
\[
\Phi \longtwoheadrightarrow \Sigma,\quad \alpha\longmapsto
\epsilon_\alpha\alpha
\]
which induces a bijection $\Phi_\red\cong\Sigma$.

Under this bijection, $\Phi^+$ corresponds to a system of positive roots in
$\Sigma$, which we denote $\Sigma^+$.

For each $\alpha = \epsilon_\beta \beta$ in $\Sigma$, where $\beta\in
\Phi_\red$, we put $U_{\alpha}\coloneqq
U_{\beta}$ and
\[
U_{(\alpha,k)}\coloneqq U_{\beta, \epsilon_\beta^{-1}
k},\qquad \text{for all $k\in\Z$.}
\]
This defines a $\Z$-indexed descending filtration on $U_\alpha$ by compact open
subgroups which is separated and exhaustive. 
If $n\in N$ with image $w$ in $W_0$, we have 
\begin{equation}\label{eq:N-rootgroup}
n U_{(\alpha,k)}n^{-1} = U_{(w(\alpha), k - \langle w(\alpha),
n.\varphi-\varphi\rangle)}.
\end{equation}
\subsection{The Iwahori--Weyl group}\label{subsec:Iwahori-Weyl} 
Let $K$ be the special parahoric subgroup of $G$ associated with $\varphi$
\cite[5.2.6]{Bruhat-Tits.1984}. If $X\subseteq G$ is a subgroup, we write
\[
K_X \coloneqq K\cap X.
\]
We note the following properties:
\begin{enumerate}[label=--] 
\item the special point $\varphi$ is fixed by $K$ under the natural action of
$G$ on $\building(G)$; 
\item the group $K\cap N$ contains a set of representatives of $W_0$;
\item for all $\alpha\in\Sigma$ we have $K\cap U_\alpha = U_{(\alpha,0)}$
\cite[(51)]{Vigneras.2016};
\item if $\alg P = \alg U_{\alg P}\alg M$ is a parabolic subgroup of $\alg G$
with Levi $\alg M$ and unipotent radical $\alg U_{\alg P}$,
then $K_M$ is a special parahoric subgroup of $M$
\cite[Lemma\,4.1.1]{Haines-Rostami.2010}. In particular, since $Z$ is
anisotropic, $K_Z$ is the unique parahoric subgroup of $Z$;
\end{enumerate} 

Since $K_Z$ is the unique parahoric subgroup of $Z$, it is normalized by $N$. We
call
\[
W \coloneqq N/K_Z
\]
the \emph{Iwahori--Weyl group}. The subgroup 
\[
\Lambda\coloneqq Z/K_Z \subseteq W
\]
is a finitely generated abelian group with finite torsion and the same rank as
$X_*(T)$ \cite[Theorem\,1.0.1]{Haines-Rostami.2010}. We therefore denote it
additively. When we view $\Lambda$ as a subgroup of $W$, we employ an
exponential notation, that is, we write $e^\lambda\in W$ for $\lambda\in
\Lambda$. The natural exact sequence
\[
0\longrightarrow \Lambda \longrightarrow W \longrightarrow W_0\longrightarrow 1
\]
splits, that is, $W$ decomposes as the semidirect product
\[
W\cong \Lambda \rtimes W_0,
\]
and $W_0$ acts on $\Lambda$ by $e^{w(\lambda)}\coloneqq w e^\lambda w^{-1}$. 
We note that the map $\nu\colon Z\to V$ \eqref{eq:nu} factors through $\Lambda$
and induces a $W_0$-equivariant map
\[
\nu\colon \Lambda \longrightarrow V.
\]
\subsection{The positive monoid}\label{subsec:positive} 
We define
\[
\Lambda^+\coloneqq \Lambda^{+,G} \coloneqq
\set{\lambda\in\Lambda}{\text{$\langle\alpha,\nu(\lambda)\rangle \le 0$ for all
$\alpha\in \Sigma^+$}}
\]
and denote $Z^+$ (or $Z^{+,G}$ if we want to emphasize the dependence on $G$)
the preimage of $\Lambda^+$ under the projection
$Z\twoheadrightarrow \Lambda$. We refer to $Z^+$ as the \emph{positive monoid}.
The \emph{negative monoid} is defined as $Z^-\coloneqq (Z^+)^{-1}$. We also
write $\Lambda^- \coloneqq -\Lambda^+$.

An element $\lambda\in \Lambda^+$ is called \emph{strictly positive} if
$\langle \alpha, \nu(\lambda)\rangle < 0$, for all $\alpha\in\Sigma^+$. Note
that if $\lambda$ is strictly positive, the group $\Lambda$ is generated (as a
monoid) by $\Lambda^+$ and $-\lambda$.
\bigskip

More generally, let $\alg P = \alg U_{\alg P}\alg M$ be a parabolic subgroup of
$\alg G$. We denote
\[
M^+ \coloneqq \set{m\in M}{mK_{U_P}m^{-1} \subseteq K_{U_P}}
\]
the \emph{monoid of $M$-positive elements}. Note that $K_M \subseteq M^+ \cap
(M^+)^{-1}$. We define
\begin{equation}\label{eq:mu_UP}
\mu_{U_P}(m) \coloneqq [K_{U_P} : K_{U_P} \cap m^{-1}K_{U_P}m] \in
q^{\Z_{\ge0}}.
\end{equation}
Clearly, $m\in M$ is $M$-positive if and only if $\mu_{U_P}(m) = 1$. The
integers $\mu_{U_P}(m)$ have been studied in \cite[\S3.4]{Heyer.2020}.

An element $\lambda\in \Lambda$ is called \emph{strictly $M$-positive}
if $\langle \Sigma_M, \nu(\lambda)\rangle = 0$ and $\langle \alpha,
\nu(\lambda)\rangle <0$ for all $\alpha\in \Sigma^+\setminus \Sigma_M$. 
Note that by \eqref{eq:N-rootgroup}, the monoid
\begin{equation}\label{eq:LambdaM+}
\Lambda_{M^+} \coloneqq \set{\lambda\in \Lambda}{\text{$\langle\alpha,
\nu(\lambda)\rangle \le 0$ for all $\alpha\in \Sigma^+\setminus\Sigma_M$}}
\end{equation}
coincides with the image of $Z\cap M^+$ in $\Lambda$.

We call $a\in Z$ \emph{strictly $M$-positive} if $a$ lies in the center of $M$
and $aK_Z\in \Lambda$ is strictly $M$-positive.
We remark that by \cite[(6.14)]{Bushnell-Kutzko.1998} there exist strictly
$M$-positive elements.
\subsection{Double coset decompositions}\label{subsec:Cartan} 
We recall here the Cartan and the Iwasawa decomposition of $G$. In
\S\ref{sec:algo} we will study intersections between Cartan and Iwasawa double
cosets.

\begin{cartan}[{\cite[Theorem~1.0.3]{Haines-Rostami.2010}}] 
\label{cartan}
The inclusion $Z\subseteq G$ induces a bijection
\[
\Lambda/W_0 \cong K\backslash G/K.
\]
\end{cartan} 

\begin{rmk*} 
\begin{enumerate}[label=(\alph*)]
\item The monoids $\Lambda^+$ and $\Lambda^-$ are representatives for the
$W_0$-orbits of $\Lambda$ \cite[6.3 Lemma]{Henniart-Vigneras.2015}. Therefore,
the inclusion $Z\subseteq G$ induces bijections $\Lambda^+ \cong K\backslash
G/K$ and $\Lambda^- \cong K\backslash G/K$. These are also referred to as the
Cartan decomposition.

\item The Cartan decomposition implies that if $KzK = Kz'K$, for some $z,z'\in
Z$, then there exists $w\in W_0$ such that $w(zK_Z) = z'K_Z$ (and hence also
$w.\nu(z) = \nu(z')$).
\end{enumerate}
\end{rmk*} 

\begin{iwasawa}\label{iwasawa} 
The inclusion $Z\subseteq G$ induces a bijection
\[
\Lambda \cong U\backslash G/K.
\]
This decomposition is often written as $G = BK = UZK = ZUK$.
\end{iwasawa} 

\begin{rmk}\label{rmk:Iwasawa} 
It is of general interest to study intersections of Cartan and Iwasawa double
cosets. We recall some well-known results. Let $z\in Z^-$ and $z'\in Z$ such that
$Uz'K \cap KzK \neq \emptyset$. Then:
\begin{enumerate}[label=(\alph*)]
\item\label{iwasawa-a} $\nu(z') \le \nu(z)$, see
\cite[Lemma~10.2.1]{Haines-Rostami.2010} or \cite[6.10
Proposition]{Henniart-Vigneras.2015}. 
\item\label{iwasawa-b} If $\nu(z) = \nu(z')$, then $zK_Z = z'K_Z$, see
\cite[6.10 Proposition]{Henniart-Vigneras.2015}.
\item\label{iwasawa-c} $w.\nu(z') \le \nu(z)$, for all $w\in W_0$. This follows
from properties of the Satake homomorphism as in \cite[Lemma~2.1]{Rapoport.2000}
but using \cite[7.13 Theorem]{Henniart-Vigneras.2015}. This argument is also
spelled out in Remark~\ref{rmk:Satake}.

This inequality is equivalent to saying that $\nu(z')$ lies inside the convex
polytope spanned by the $W_0$-orbit of $\nu(z)$, cf.
\cite[(2.6.2)]{Macdonald.2003}.
\end{enumerate}
In \S\ref{sec:algo} we give an algorithm which yields also information about the
$u\in U$ with $uz'\in KzK$.
\end{rmk} 
\subsection{Abstract Hecke rings}\label{subsec:Hecke} 
We briefly discuss abstract Hecke rings. The references below refer to, and
details can be found in, \cite[Ch.\,3, \S1]{Andrianov.1995}. 

Let $G$ be a topological group and $\Gamma\subseteq G$ a compact open
subgroup. Let $\Gamma\subseteq S\subseteq G$ be a submonoid. The pair $(\Gamma,
S)$ is called a \emph{Hecke pair}. Let
\[
\Z[\Gamma\backslash S] = \bigoplus_{\Gamma s\in \Gamma\backslash S} \Z.(\Gamma
s)
\]
be the free $\Z$-module on the set of right cosets $\Gamma\backslash S$. It
admits a natural right $S$-action by $(\Gamma s)\cdot s' = (\Gamma ss')$, for
$s,s'\in S$. Clearly, $\Z[\Gamma\backslash S]$ is a left module under the ring
\[
\Hecke(\Gamma, S) \coloneqq \End_S(\Z[\Gamma\backslash S]).
\]
We usually make the identification
\begin{align*}
\Hecke(\Gamma, S) &\xrightarrow\cong \Z[\Gamma\backslash S]^\Gamma,\\
T &\longmapsto T((\Gamma)).
\end{align*}
The submodule $\Z[\Gamma\backslash S]^\Gamma$ of
$\Gamma$-invariants is a free $\Z$-module on the set of double cosets
$\Gamma\backslash S/\Gamma$. Concretely, it admits $\set{(s)_\Gamma}{\Gamma
s\Gamma \in \Gamma\backslash S/\Gamma}$ as a basis, where 
\[
(s)_\Gamma \coloneqq \sum_{\Gamma s' \subseteq \Gamma s\Gamma} (\Gamma s').
\]
The sum runs through all right cosets contained in $\Gamma s\Gamma$. Note that
the sum is finite, because $\Gamma$ is compact open, so that the set $(\Gamma\cap
s^{-1}\Gamma s)\backslash \Gamma$ is finite, and the map
\begin{align}\label{eq:mu}
(\Gamma\cap s^{-1}\Gamma s)\backslash \Gamma &\xrightarrow{\cong}
\Gamma\backslash
\Gamma s\Gamma,\\
(\Gamma\cap s^{-1}\Gamma s)\gamma &\longmapsto \Gamma s\gamma \notag
\end{align}
is bijective. The multiplication on $\Z[\Gamma\backslash S]^\Gamma$ is
concretely given by
\[
\Bigl(\sum_i n_i\cdot (\Gamma s_i)\Bigr)\cdot \Bigl(\sum_j m_j\cdot (\Gamma
t_j)\Bigr) = \sum_{i,j} n_im_j\cdot (\Gamma s_it_j).
\]
For an explicit description of the multiplication in terms of double cosets,
see \cite[Lemma~1.5]{Andrianov.1995}.

The following two results are frequently useful:

\begin{prop}[{\cite[Prop.~1.9]{Andrianov.1995}}] 
\label{prop:Hecke-embedding}
Let $(\Gamma, S)$ and $(\Gamma_0,S_0)$ be two Hecke pairs satisfying
\begin{equation}\label{eq:Hecke-embedding}
\Gamma_0 \subseteq \Gamma,\qquad S\subseteq \Gamma S_0,\qquad \text{and}\qquad
\Gamma\cap S_0\cdot S_0^{-1} \subseteq \Gamma_0.
\end{equation}
Then the map
\begin{align*}
\varepsilon\colon \Hecke(\Gamma, S) &\longhookrightarrow
\Hecke(\Gamma_0,S_0),\\
\sum_i n_i\cdot (\Gamma s_i) &\longmapsto \sum_i n_i\cdot (\Gamma_0 s_i),
\end{align*}
where the $s_i$ are chosen in $S_0$, is an injective ring homomorphism.
\end{prop} 

\begin{prop}[{\cite[Prop.~1.11]{Andrianov.1995}}] 
\label{prop:Hecke-involution}
Let $(\Gamma, S)$ be a Hecke pair. Then $(\Gamma, S^{-1})$ is also a Hecke pair,
and the map
\begin{align}\label{eq:Hecke-involution}
\zeta_S\colon \Hecke(\Gamma, S) &\longrightarrow \Hecke(\Gamma, S^{-1}),\\
(s)_\Gamma &\longmapsto (s^{-1})_\Gamma \notag
\end{align}
is an anti-isomorphism of rings.
\end{prop} 

\begin{lem}[{\cite[Lem.~1.13]{Andrianov.1995}}] 
\label{lem:emb-anti}
Let $(\Gamma, S)$ and $(\Gamma_0,S_0)$ be two Hecke pairs satisfying
\eqref{eq:Hecke-embedding} such that $(\Gamma, S^{-1})$ and $(\Gamma_0,
S_0^{-1})$ also satisfy \eqref{eq:Hecke-embedding}. Then the following diagram
is commutative:
\[
\begin{tikzcd}
\Hecke(\Gamma, S) \ar[r,"\varepsilon"] \ar[d,"\zeta_S"'] & \Hecke(\Gamma_0,S_0)
\ar[d,"\zeta_{S_0}"]\\
\Hecke(\Gamma, S^{-1}) \ar[r,"\varepsilon"'] & \Hecke (\Gamma_0,S_0^{-1}).
\end{tikzcd}
\]
\end{lem} 

If $R$ is a commutative ring with 1, we put
\[
\Hecke_R(\Gamma, S) \coloneqq R\otimes_\Z \Hecke(\Gamma, S).
\]
It is clear that Propositions~\ref{prop:Hecke-embedding}
and~\ref{prop:Hecke-involution} and Lemma~\ref{lem:emb-anti} remain valid for
Hecke rings over $R$.
\section{Non-obtuse parabolics}\label{sec:non-obtuse} 
We fix a maximal parabolic subgroup $\alg P = \alg U_{\alg P}\alg M$ of
$\alg G$. Recall from \S\ref{subsec:apartment} the Euclidean vector space
$(V^*,\scalar)$, on which the finite Weyl group $W_0$ acts, and the special
point $\varphi\in \apartment$ from \S\ref{subsec:apartment}. We view $\varphi$
as a valuation on the root group datum $(Z,(U_\alpha)_{\alpha\in\Phi})$. Also
recall from \S\ref{subsec:reducedroot} the reduced root system $\Sigma \subseteq
V^*$. The system $\Sigma^+$ of positive roots determines a unique basis $\Delta$
of $\Sigma$. Since $\alg M$ is reductive, all these objects have an analogue for
$\alg M$, and we denote them by adding the subscript `$M$'. For example, we write
$\Sigma_M$, $W_{0,M}$, $\Delta_M$ etc.

\begin{defn} 
The parabolic $\alg P$ is called \emph{non-obtuse} if
\[
\langle\alpha, \beta^\vee\rangle \ge 0,\qquad \text{for all $\alpha,\beta\in
\Sigma^+\setminus \Sigma_M$.}
\]
\end{defn} 
\begin{rmk*} 
For $\alg P$ to be non-obtuse it is equivalent to say that
\[
(\alpha,\beta) \ge0, \qquad \text{for all $\alpha,\beta\in
\Sigma^+\setminus\Sigma_M$.}
\]

This more geometric definition explains the term ``non-obtuse'':  it means that
any two roots in $\Sigma^+\setminus \Sigma_M$ span a non-obtuse angle.
\end{rmk*} 

During the whole section we assume that $\alg P$ is non-obtuse. 
The main result of this section is a complete classification of the
non-obtuse parabolic subgroups of $\alg G$. First, we prove an important
technical result which explains our interest in non-obtuse parabolics.

\begin{lem}\label{lem:no-cone} 
Let $\lambda,\mu\in \Lambda$ such that $\lambda$ is strictly $M$-positive. Assume
that $\nu(w(\mu)) \le \nu(-\lambda)$ for all $w\in W_0$. Then one has 
\[
\langle \alpha, \nu(\lambda+\mu)\rangle \le 0,\qquad \text{for all $\alpha\in
\Sigma^+\setminus \Sigma_M$.}
\]
\end{lem} 

\begin{ex*} 
Before giving the proof, let us look at the two examples in
Figure~\ref{fig:non-obtuse}. Example~\hyperref[fig:non-obtuse-A]
{({\scriptsize A})}
explains why Lemma~\ref{lem:no-cone} should be expected to hold true for
non-obtuse parabolics, while \hyperref[fig:non-obtuse-B]{({\scriptsize B})}
explains why Lemma~\ref{lem:no-cone} fails otherwise.
\begin{figure}[ht] 
\centering
\caption{In both examples we choose $\Sigma_M = \{\pm\alpha\}$, and $\nu(\mu)$
can lie anywhere in the dotted region.}\label{fig:non-obtuse}
\begin{minipage}[t]{0.4\textwidth} 
\centering
\begin{tikzpicture} 
\clip (-2,-2) rectangle (2,2);
\filldraw[fill=gray,opacity=.2] (210:4) -- (0,0) -- (330:4) -- cycle;
\def\radius{1.3}
\def\rlam{1.6}
\foreach \x in {0,60,...,300}
{
\draw[thick] (0,0) -- (\x:\radius);
\fill (\x:\radius) circle (1pt);
}
\foreach \x in {90,210,330}
\draw[dotted] (\x:\rlam) -- (\x+120:\rlam);

\node[right] at (\radius,0) {$\alpha$};
\node[left] at (120:\radius) {$\beta$};
\fill (270:\rlam) circle (.8pt) node[right] {$\nu(\lambda)$};
\fill (90:\rlam) circle (.8pt) node[right] {$-\nu(\lambda)$};
\end{tikzpicture} 
\subcaption{Type $A_2$: The translate of the dotted region by $\nu(\lambda)$
fits into the shaded area. Lemma~\ref{lem:no-cone} holds in this
case.}\label{fig:non-obtuse-A}
\end{minipage} 
\qquad
\begin{minipage}[t]{0.4\textwidth} 
\centering
\begin{tikzpicture} 
\clip (-2,-2) rectangle (2,2);
\filldraw[fill=gray,opacity=.2] (240:4) -- (0,0) -- (300:4) -- cycle;
\def\radius{.7}
\def\rlam{1}
\foreach \x in {30,90,...,330}
{
\draw[thick] (0,0) -- (\x:1.732*\radius);
\fill (\x:1.732*\radius) circle (1pt);
\draw[dotted] (\x:1.732*\rlam) -- (\x+60:1.732*\rlam);
}
\foreach \x in {0,60,...,300}
{
\draw[thick] (0,0) -- (\x:\radius);
\fill (\x:\radius) circle (1pt);
}
\node[right] at (\radius,0) {$\alpha$};
\node[left] at (150:1.732*\radius) {$\beta$};
\fill (270:1.732*\rlam) circle (.8pt) node[right] {$\nu(\lambda)$};
\fill (90:1.732*\rlam) circle (.8pt) node[right] {$-\nu(\lambda)$};
\end{tikzpicture} 
\subcaption{Type $G_2$: The translate of the dotted region by $\nu(\lambda)$
does not fit into the shaded area. Lemma~\ref{lem:no-cone}
fails.}\label{fig:non-obtuse-B} 
\end{minipage}
\end{figure} 
\addtocounter{figure}{-1} 
\end{ex*} 

\begin{proof}[Proof of Lemma~\ref{lem:no-cone}] 
Recall that $\alg P$ is non-obtuse so that $\langle\alpha,\beta^\vee\rangle \ge0$
for all $\alpha,\beta\in \Sigma^+\setminus\Sigma_M$.
We proceed in two steps.

\textit{Step~1:} Assume $\mu = w(-\lambda)$, for some $w\in W_0$. We do an
induction on the length $\ell(w)$ of $w$. If $\ell(w) = 1$, we write $w =
s_\beta$ for some simple root $\beta\in \Delta$. For each $\alpha\in
\Sigma^+\setminus\Sigma_M$ we compute
\begin{align*}
\big\langle \alpha, \nu(\lambda + s_\beta(-\lambda))\big\rangle &= \big\langle
\alpha, \nu(\lambda) - s_\beta(\nu(\lambda))\big\rangle\\
&= \big\langle \alpha, \langle \beta,\nu(\lambda)\rangle\cdot
\beta^\vee\big\rangle\\
&= \langle\beta,\nu(\lambda)\rangle\cdot \langle \alpha,\beta^\vee\rangle\\
&\le 0,
\end{align*}
where in the last step we have used $\langle \beta,\nu(\lambda)\rangle < 0$ and
that $\alg P$ is non-obtuse. Now assume $\ell(w)>1$, and let $\beta\in \Delta$
with $\ell(s_\beta w) <\ell(w)$. We write
\begin{equation}\label{eq:no-cone-1}
\nu\bigl(\lambda + w(-\lambda)\bigr) = \nu\bigl(\lambda +
s_\beta(-\lambda)\bigr) + s_\beta\bigl(\nu(\lambda + s_\beta w(-\lambda))\bigr).
\end{equation}
We distinguish two cases:
\begin{enumerate}[label=--] 
\item If $\beta\in\Delta_M$, then we have $s_\beta(\Sigma^+\setminus \Sigma_M) =
\Sigma^+\setminus\Sigma_M$. For each $\alpha\in\Sigma^+\setminus\Sigma_M$, the
induction hypothesis (applied to $s_\beta w$) yields:
\[
\big\langle\alpha, s_\beta\bigl(\nu(\lambda + s_\beta
w(-\lambda))\bigr)\big\rangle = \big\langle s_\beta(\alpha), \nu\bigl(\lambda +
s_\beta w(-\lambda)\bigr)\big\rangle \le 0.
\]
Together with \eqref{eq:no-cone-1} and the base case the statement follows.

\item If $\beta\in \Delta\setminus\Delta_M$, then we compute for each $\alpha\in
\Sigma^+\setminus\Sigma_M$:
\begin{align*}
\big\langle\alpha,\nu\bigl(\lambda + w(-\lambda)\bigr)\big\rangle &= \big\langle
\alpha, \nu(\lambda + s_\beta(-\lambda))\big\rangle + \big\langle \alpha,
s_\beta\bigl(\nu(\lambda + s_\beta w(-\lambda))\bigr)\big\rangle\\
&= \big\langle \alpha, \langle \beta,\nu(\lambda)\rangle\cdot
\beta^\vee\big\rangle + \big\langle s_\beta(\alpha), \nu(\lambda + s_{\beta}
w(-\lambda))\big\rangle\\
&= \langle \beta,\nu(\lambda)\rangle\cdot \langle \alpha,\beta^\vee\big\rangle +
\big\langle \alpha - \langle\alpha,\beta^\vee\rangle\cdot\beta, \nu(\lambda +
s_\beta w(-\lambda))\big\rangle\\
&= \big\langle \alpha, \nu(\lambda + s_\beta w(-\lambda))\big\rangle -
\langle\alpha,\beta^\vee\rangle\cdot\big\langle \beta, \nu(s_\beta
w(-\lambda))\big\rangle\\
&= \big\langle \alpha, \nu(\lambda + s_\beta w(-\lambda))\big\rangle + \langle
\alpha,\beta^\vee\rangle\cdot \big\langle (s_\beta w)^{-1}(\beta),
\nu(\lambda)\big\rangle\\
&\le 0,
\end{align*}
where the last step uses the induction hypothesis and that $(s_\beta
w)^{-1}(\beta) \in \Sigma^+$, which in turn follows from $\ell((s_\beta
w)^{-1}s_\beta) > \ell((s_\beta w)^{-1})$ (see, \eg
\cite[1.6~Lemma]{Humphreys.1990}).
\end{enumerate} 
This finishes the induction step.\bigskip

\textit{Step~2:} Let $\mu$ be general. Take an arbitrary $\alpha\in
\Sigma^+\setminus\Sigma_M$. Let $\alpha_0 = w(\alpha)$ be a root of maximal
height in the $W_0$-orbit of $\alpha$. (If we write $\alpha_0 =
\sum_{\beta\in\Delta} n_\beta\beta$, the \emph{height} of $\alpha_0$ is
$\sum_{\beta\in\Delta}n_\beta$.) Then we have $\langle \alpha_0,
\beta^\vee\rangle \ge0$, for all $\beta\in\Sigma^+$, since otherwise
$s_\beta(\alpha_0) = \alpha_0 -\langle \alpha_0,\beta^\vee\rangle\cdot \beta$
would have greater height than $\alpha_0$. By the hypothesis, $\nu(-\lambda)
-\nu(w(\mu))$ is a linear combination of simple coroots with non-negative
coefficients. Therefore,
\begin{equation}\label{eq:no-cone-2}
\big\langle\alpha, \nu(w^{-1}(\lambda) + \mu)\big\rangle = \big\langle
w(\alpha), \nu(\lambda + w(\mu))\big\rangle = -\big\langle \alpha_0,
\nu(-\lambda) - \nu(w(\mu))\big\rangle \le 0.
\end{equation}
By Step~1 we have
\begin{equation}\label{eq:no-cone-3}
\big\langle \alpha, \nu(\lambda + w^{-1}(-\lambda))\big\rangle \le 0.
\end{equation}
Finally, \eqref{eq:no-cone-2} and \eqref{eq:no-cone-3} imply
\[
\big\langle \alpha, \nu(\lambda + \mu)\big\rangle = \big\langle \alpha,
\nu(\lambda + w^{-1}(-\lambda))\big\rangle + \big\langle \alpha,
\nu(w^{-1}(\lambda) + \mu)\big\rangle \le 0.\qedhere
\]
\end{proof} 

We now turn to the classification of non-obtuse parabolics. Since $\alg P$ is
assumed to be maximal, the roots in $\Sigma^+\setminus\Sigma_M$ are contained in
an irreducible component of $\Sigma$. Without loss of generality we may
therefore assume that $\Sigma$ is irreducible.

The maximal parabolics of $\alg G$ are in one-to-one correspondence with the
elements of the basis $\Delta$ of $\Sigma$. We write $\Delta =
\{\alpha_1,\dotsc,\alpha_n\}$ and denote $s_r\coloneqq s_{\alpha_r}$ the simple
reflection attached to $\alpha_r$. Let $\alg P_r = \alg U_{\alg P_r}\alg M_r$
be the maximal parabolic subgroup corresponding to $\alpha_r$ so that
$\Delta_{M_r} = \Delta\setminus\{\alpha_r\}$. We put 
\[
\Sigma_{U_{P_r}}\coloneqq \Sigma^+ \setminus \Sigma_{M_r}.
\]
Then $\alpha_r$ is the unique element in $\Sigma_{U_{P_r}} \cap \Delta$. We say
that $\alpha_r$ is \emph{non-obtuse} if $\alg P_r$ is. 

\begin{prop}\label{prop:classification} 
The classification of non-obtuse parabolic subgroups of $\alg G$ is given in
terms of the Dynkin diagram of $\Sigma$ in Figure~\ref{fig:class}.
Moreover, the following conditions are equivalent:
\begin{enumerate}[label=(\roman*)]
\item\label{prop:class-i} $\alpha_r$ is non-obtuse.
\item\label{prop:class-ii} $\langle\alpha_r,\beta^\vee\rangle \ge 0$ for all
$\beta\in \Sigma_{U_{P_r}}$.
\item\label{prop:class-iii} The Weyl group $W_{0,M_r}$ acts transitively on the
roots in $\Sigma_{U_{P_r}}$ of the same length.
\item\label{prop:class-iv} In the notation of \eqref{eq:ci(beta)} below we have
$c_r(\alpha^r_0) = 1$, where $\alpha^r_0$ is the highest root of the same length
as $\alpha_r$.
\end{enumerate}
\begin{figure*}[ht!]
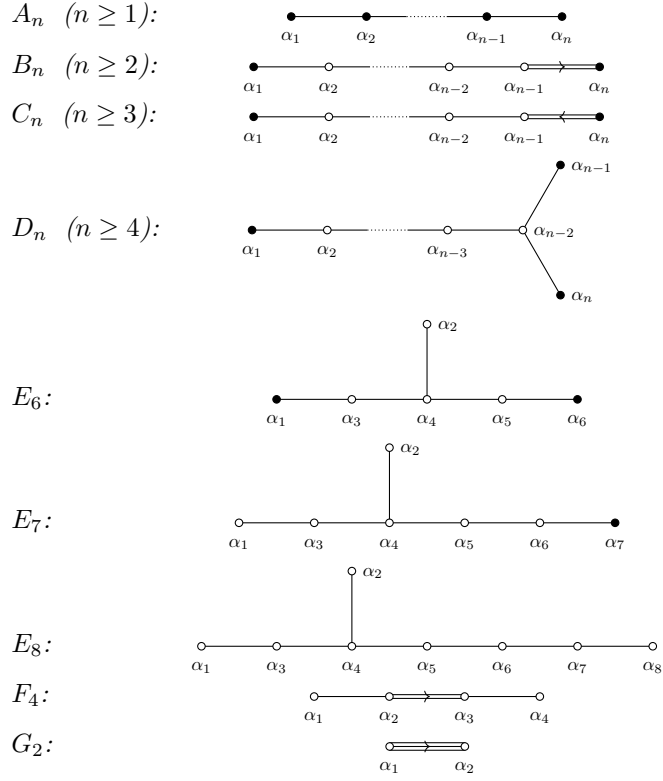
 
\centering
\begin{tabular}{lc} 
$A_n$\; ($n\ge 1$): & \dynkin [%
labels={\alpha_1,\alpha_2,\alpha_{n-1},\alpha_n},
] A{}\\

$B_n$\; ($n\ge2$): & \dynkin [%
labels={\alpha_1,\alpha_2,\alpha_{n-2},\alpha_{n-1},\alpha_n},
] B{*o.oo*}\\

$C_n$\; ($n\ge3$): & \dynkin [%
labels={\alpha_1,\alpha_2,\alpha_{n-2},\alpha_{n-1},\alpha_n},
] C{*o.oo*}\\

$D_n$\; ($n\ge4$): & \dynkin [%
labels={\alpha_1,\alpha_2,\alpha_{n-3},\alpha_{n-2},\alpha_{n-1},\alpha_n},
label directions={,,,right,,},
] D{*o.oo**}\\

$E_6$: & \dynkin [%
labels={\alpha_1,\alpha_2,\alpha_3,\alpha_4,\alpha_5,\alpha_6},
] E{*oooo*}\\

$E_7$: & \dynkin [%
labels={\alpha_1,\alpha_2,\alpha_3,\alpha_4,\alpha_5,\alpha_6,\alpha_7},
] E{oooooo*}\\

$E_8$: & \dynkin [%
labels={\alpha_1,\alpha_2,\alpha_3,\alpha_4,\alpha_5,\alpha_6,\alpha_7,\alpha_8},
] E{oooooooo}\\

$F_4$: & \dynkin [%
labels={\alpha_1,\alpha_2,\alpha_3,\alpha_4},
] F{oooo}\\

$G_2$: & \dynkin [%
labels={\alpha_1,\alpha_2},
] G{oo}
\end{tabular} 

\caption{The black vertices are precisely the non-obtuse simple
roots. Note that in types $E_8$, $F_4$, and $G_2$ there are no non-obtuse
parabolics, while in type $A_n$ all maximal parabolics are
non-obtuse.}\label{fig:class}
\end{figure*} 

\end{prop} 
\begin{proof} 
We consider each type separately. 
The equivalences of \ref{prop:class-i}--\ref{prop:class-iv} are obtained
along the way.
For the concrete description of the roots systems we follow \cite[Planches
I--IX]{Bourbaki.1981}. 

As usual we denote $e_1,\dotsc,e_n$ the standard basis of $\R^n$, endowed with
the canonical scalar product $\scalar$. Given a root $\beta\in \Sigma$, we
write
\begin{equation}\label{eq:ci(beta)}
\beta = \sum_{i=1}^n c_i(\beta)\cdot \alpha_i.
\end{equation}
Observe that $c_i(s_j(\beta)) = c_i(\beta)$, for $i\neq j$. In particular, the
action of $W_{0,M_r}$
does not affect $c_r(\beta)$. Observe that $\Sigma_{U_{P_r}}$ consists of those
$\beta\in \Sigma^+$ with $c_r(\beta)>0$.

\begin{enumerate} 
\item[($A_n$)]\label{class-An} 
Inside $V = \bigl\{(x_1,\dotsc,x_{n+1})\in \R^{n+1}\,\big|\,\sum_{i=1}^{n+1}x_i
= 0\bigr\}$ the root system of type $A_n$ is
\[
\Sigma = \set{\pm (e_i-e_j)}{1\le i< j\le n+1}.
\]
The simple roots are given by $\alpha_i = e_i-e_{i+1}$, for $i=1,\dotsc,n$, and
the roots $e_i-e_j = \sum_{k=i}^{j-1}\alpha_k$, for $1\le i<j\le n+1$, are
positive. The Weyl group $W_0$ is the symmetric group $\frakS_{n+1}$ acting on
$e_1,\dotsc,e_{n+1}$. Fix $1\le r\le n$. Then we have
\[
\Sigma_{U_{P_r}} = \set{e_i-e_j}{1\le i\le r < j\le n+1}.
\]
If $e_i-e_j, e_a-e_b \in \Sigma_{U_{P_r}}$, then we have $a\neq j$ and $i\neq
b$. Hence, $(e_i-e_j,e_a-e_b) = \delta_{ia} + \delta_{jb} \ge 0$ and $\alpha_r$
is non-obtuse. The Weyl group $W_{0,M_r}$ identifies with $\frakS_{r}\times
\frakS_{n+1-r}$ with the first factor acting on $\{e_1,\dotsc,e_r\}$ and the
second on $\{e_{r+1},\dotsc,e_{n+1}\}$. It clearly acts transitively on
$\Sigma_{U_{P_r}}$. 
\item[($B_n$)]\label{class-Bn} 
Inside $V = \R^n$ the root system of type $B_n$ is
\[
\Sigma = \set{\pm e_i}{1\le i\le n} \cup \set{\pm e_i\pm e_j}{1\le i < j\le n}.
\]
A basis is given by $\alpha_i = e_i-e_{i+1}$, for $1\le i < n$, and $\alpha_n =
e_n$. The positive roots are
\[
\begin{cases}
e_i = \sum_{k=i}^n \alpha_k, & \text{for $1\le i\le n$;}\\
e_i - e_j = \sum_{k=i}^{j-1}\alpha_k, & \text{for $1\le i < j\le n$;}\\
e_i + e_j = \sum_{k=i}^{j-1}\alpha_k + 2\sum_{k=j}^n\alpha_k, & \text{for $1\le
i <j\le n$.}
\end{cases}
\]
The Weyl group $W_0$ identifies with $(\Z/2\Z)^n\rtimes \frakS_n$ with
$\frakS_n$ permuting the $e_i$, and $(\Z/2\Z)^n$ acting by changing the signs of
the $e_i$. Consider three cases:
\begin{enumerate}[label=--]
\item Assume $r=1$. Then
\[
\Sigma_{U_{P_1}} = \{e_1\} \cup \set{e_1\pm e_i}{2\le i\le n}.
\]
Given $e_1 + \varepsilon_1e_i, e_1+\varepsilon_2e_j \in \Sigma_{U_{P_1}}$, with
$\varepsilon_1,\varepsilon_2\in \{-1,0,1\}$, we compute
\[
(e_1+\varepsilon_1e_i,e_1+\varepsilon_2e_j) = 1 +
\varepsilon_1\varepsilon_2\delta_{ij} \ge0.
\]
Hence, $\alpha_1$ is non-obtuse. The Weyl group $W_{0,M_1}$ identifies with
$(\Z/2\Z)^{n-1}\rtimes \frakS_{n-1}$ with both groups acting on
$\{\pm e_2,\dotsc,\pm e_n\}$, leaving $e_1$ fixed. It clearly acts transitively
on $\set{e_1\pm e_i}{2\le i\le n}$ (and on $\{e_1\}$).

\item Assume $r=n$. Then
\[
\Sigma_{U_{P_n}} = \set{e_i}{1\le i\le n} \cup \set{e_i+e_j}{1\le i<j\le n}.
\]
It is obvious that $\alpha_n$ is non-obtuse. The Weyl group $W_{0,M_n}$
identifies with $\frakS_n$ acting on $e_1,\dotsc,e_n$. It clearly acts
transitively on both $\{e_1,\dotsc,e_n\}$ and $\set{e_i+e_j}{1\le i <j\le n}$. 

\item Assume $1 < r < n$; in particular, $n\ge3$. Note that both $\alpha_r =
e_r-e_{r+1}$ and $e_{r-1} + e_{r+1}$ lie in $\Sigma_{U_{P_r}}$ and satisfy
$(e_r-e_{r+1}, e_{r-1}+e_{r+1}) = -1$. Hence, $\alpha_r$ is not non-obtuse. The
highest root is $\alpha_0 = e_1+e_2 = \alpha_1 + 2\sum_{k=2}^n\alpha_k$. Notice
that $\alpha_0$ and $\alpha_r$ both lie in $\Sigma_{U_{P_r}}$ and have the same
length. But since $c_r(\alpha_0) = 2\neq 1 = c_r(\alpha_r)$, it follows that
$\alpha_0$ does not lie in the $W_{0,M_r}$-orbit of $\alpha_r$.
\end{enumerate}
\item[($C_n$)]\label{class-Cn} 
Inside $V= \R^n$ the root system of type $C_n$ is
\[
\Sigma = \set{\pm 2e_i}{1\le i\le n} \cup \set{\pm e_i\pm e_j}{1\le i <j\le n}.
\]
A basis is given by $\alpha_i = e_i-e_{i+1}$, for $1\le i < n$, and $\alpha_n =
2e_n$. The positive roots are
\[
\begin{cases}
e_i - e_j = \sum_{k=i}^{j-1}\alpha_k, & \text{for $1\le i<j\le n$;}\\
e_i + e_j = \sum_{k=i}^{j-1}\alpha_k + 2\sum_{k=j}^{n-1}\alpha_k + \alpha_n, &
\text{for $1\le i<j\le n$;}\\
2e_i = 2\sum_{k=i}^{n-1}\alpha_k + \alpha_n, & \text{for $1\le i\le n$.}
\end{cases}
\]
The Weyl group $W_0$ identifies with $(\Z/2\Z)^n\rtimes\frakS_n$ as for type
$B_n$. Consider three cases:
\begin{enumerate}[label=--]
\item Assume $r=1$. Then
\[
\Sigma_{U_{P_1}} = \{2e_1\}\cup \set{e_1\pm e_i}{2\le i\le n}.
\]
Given $e_1+\varepsilon_1e_i, e_1+\varepsilon_2e_j \in \Sigma_{U_{P_1}}$, with
$i,j\neq 1$ and
$\varepsilon_1,\varepsilon_2\in \{\pm1\}$, we compute
\[
(e_1 + \varepsilon_1e_i, e_1+\varepsilon_2e_j) = 1 +
\varepsilon_1\varepsilon_2\delta_{ij}\ge0.
\]
Since also $(2e_1,e_1\pm e_i) =2 \ge 0$ (for $i\neq 1$), the root $\alpha_1$ is
non-obtuse. The
Weyl group $W_{0,M_1}$ identifies with $(\Z/2\Z)^{n-1}\rtimes \frakS_{n-1}$ with
both groups acting on $\{\pm e_2,\dotsc,\pm e_n\}$ leaving $e_1$ fixed. It
clearly acts transitively on $\set{e_1\pm e_i}{2\le i\le n}$ (and on
$\{2e_1\}$). 

\item Assume $r=n$. Then 
\[
\Sigma_{U_{P_n}} = \set{2e_i}{1\le i\le n} \cup \set{e_i+e_j}{1\le i<j\le n}.
\]
It is obvious that $\alpha_n$ is non-obtuse. The Weyl group $W_{0,M_n}$
identifies with $\frakS_n$ acting on $e_1,\dotsc,e_n$. It clearly acts
transitively on $\{2e_1,\dotsc,2e_n\}$ and on $\set{e_i+e_j}{1\le i <j\le n}$.

\item Assume $1<r<n$; in particular, $n\ge3$. Note that both $\alpha_r =
e_r-e_{r+1}$ and $e_{r-1}+e_{r+1}$ lie in $\Sigma_{U_{P_r}}$ and satisfy
$(e_r-e_{r+1},e_{r-1}+e_{r+1}) = -1$. Hence, $\alpha_r$ is not non-obtuse.
Consider the root $\alpha^r_0 = e_1+e_2 = \alpha_1+2\sum_{k=2}^{n-1}\alpha_k +
\alpha_n$. Then $\alpha^r_0$ and $\alpha_r$ both lie in $\Sigma_{U_{P_r}}$ and
have the same length. But since $c_r(\alpha^r_0) = 2 \neq 1 = c_r(\alpha_r)$, it
follows that $\alpha^r_0$ does not lie in the $W_{0,M_r}$-orbit of $\alpha_r$.
\end{enumerate}
\item[($D_n$)]\label{class-Dn} 
Inside $V = \R^n$ the root system of type $D_n$ is 
\[
\Sigma = \set{\pm e_i\pm e_j}{1\le i < j\le n}.
\]
A basis is given by $\alpha_i = e_i-e_{i+1}$, for $1\le i<n$, and $\alpha_n =
e_{n-1} + e_n$. The positive roots are 
\[
\begin{cases}
e_i - e_j = \sum_{k=i}^{j-1}\alpha_k, & \text{for $1\le i<j\le n$;}\\
e_i + e_n = \sum_{k=i}^{n-2}\alpha_k + \alpha_n, & \text{for $1\le i < n$;}\\
e_i+e_j = \sum_{k=i}^{j-1}\alpha_k + 2\sum_{k=j}^{n-2}\alpha_k + \alpha_{n-1} +
\alpha_n, & \text{for $1\le i<j< n$.}
\end{cases}
\]
The Weyl group $W_0$ identifies with $\Gamma\rtimes\frakS_n$, where $\Gamma$ is
the kernel of the map $(\Z/2\Z)^n \to \Z/2\Z$, $(x_i)_i\mapsto
\sum_{i=1}^nx_i$. We distinguish the following cases:
\begin{enumerate}[label=--]
\item Assume $r=1$. Then
\[
\Sigma_{U_{P_1}} = \set{e_1\pm e_i}{2\le i\le n}.
\]
The same computation as in \hyperref[class-Bn]{$(B_n)$} shows that $\alpha_1$ is
non-obtuse.
The Weyl group $W_{0,M_1}$ identifies with $\Gamma_1\rtimes \frakS_{n-1}$, where
$\frakS_{n-1}$ permutes $e_2,\dotsc,e_n$ and $\Gamma_1\subseteq \Gamma$ is the
subgroup of elements $(x_i)_i$ with $x_1 = 0$. It is easy to check that
$W_{0,M_1}$ acts transitively on $\Sigma_{U_{P_1}}$.

\item Assume $r=n-1$ or $r=n$. By the symmetry of the Dynkin diagram it suffices
to consider the case $r=n$. Then
\[
\Sigma_{U_{P_n}} =\set{e_i+e_j}{1\le i<j\le n}.
\]
It is obvious that $\alpha_n$ is non-obtuse. The Weyl group $W_{0,M_n}$
identifies with $\frakS_n$ which acts by permuting the $e_1,\dotsc,e_n$. It
clearly acts transitively on $\Sigma_{U_{P_n}}$.

\item Assume $1<r<n-1$. Both $\alpha_r = e_r-e_{r+1}$ and $e_{r-1}+e_{r+1}$ lie
in $\Sigma_{U_{P_r}}$ and satisfy $(e_r-e_{r+1},e_{r-1}+e_{r+1}) = -1$. Hence,
$\alpha_r$ is not non-obtuse. The highest root is $\alpha_0 = e_1 + e_2 =
\alpha_1 + 2\sum_{k=2}^{n-2}\alpha_k + \alpha_{n-1}+\alpha_n$. Then $\alpha_0$
and $\alpha_r$ both lie in $\Sigma_{U_{P_r}}$ and have the same length. But
since $c_r(\alpha_0)= 2\neq 1 = c_r(\alpha_r)$, it follows that $\alpha_0$ does
not lie in the $W_{0,M_r}$-orbit of $\alpha_r$.
\end{enumerate}
\item[($E_6$)]\label{class-E6} 
Inside $V = \set{(x_1,\dotsc,x_8)\in \R^8}{x_6=x_7=-x_8}$ the root system of
type $E_6$ is
\[
\begin{split}
\Sigma = &\set{\pm e_i\pm e_j}{1\le i<j\le 5}\\
&\cup \set{\pm \frac12 \Bigl(e_8-e_7-e_6 + \sum_{i=1}^5(-1)^{\nu(i)} e_i\Bigr)}
{\nu(i)\in \Z/2\Z,\; \sum_{i=1}^5\nu(i) = 0}.
\end{split}
\]
A basis is given given by $\alpha_1 = \frac12(e_1+e_8) -\frac12 (e_2+e_3+ \dotsb
+ e_7)$, $\alpha_2=e_2+e_1$, and $\alpha_i = e_{i-1}-e_{i-2}$, for $3\le i\le
6$. We distinguish the following cases:
\begin{enumerate}[label=--]
\item Assume $r=1$ or $r=6$. By the symmetry of the Dynkin diagram for $E_6$ it
suffices to consider the case $r=1$. We have
\[
\Sigma_{U_{P_1}} = \set{\frac12 \Bigl(e_8 - e_7-e_6 + \sum_{i=1}^5
(-1)^{\nu(i)}e_i\Bigr)}{\nu(i)\in\Z/2\Z,\; \sum_{i=1}^5\nu(i) = 0}.
\]
To ease the notation we write $\alpha_\nu \coloneqq \frac12(e_8 - e_7-e_6 +
\sum_{i=1}^5(-1)^{\nu(i)}e_i)$, for each $\nu = (\nu(i))_i \in (\Z/2\Z)^5$. Let
$\nu,\mu\in (\Z/2\Z)^5$ such that $\sum_{i=1}^5\nu(i) = \sum_{i=1}^5\mu(i) = 0$.
Since $\sum_{i=1}^5(\nu(i)+\mu(i)) = 0$, we observe that the cardinality of the
set $\set{1\le i\le 5}{\nu(i)\neq \mu(i)}$ is even, hence equals $0$, $2$, or
$4$. But then $\lvert\set{1\le i\le 5}{\nu(i)=\mu(i)}\rvert$ is either $1$, $3$,
or $5$. Thus, we compute
\[
(\alpha_\nu,\alpha_\mu) = \frac14\bigl( 3 + \lvert\set{i}{\nu(i)=\mu(i)}\rvert -
\lvert \set{i}{\nu(i)\neq\mu(i)}\rvert\bigr) \ge 0.
\]
Therefore, $\alpha_1$ is non-obtuse. The Weyl group $W_{0,M_1}$ is the group
$\Gamma\rtimes\frakS_5$ of type $D_5$ described in \hyperref[class-Dn]{$(D_n)$};
it acts on
$e_1,\dotsc,e_5$, leaving $e_6$, $e_7$, and $e_8$ fixed. Given $\nu,\mu\in
(\Z/2\Z)^5$ with $\alpha_\nu,\alpha_\mu\in \Sigma_{U_{P_1}}$, we may view
$\mu-\nu$ as an element of $\Gamma$ which maps $\alpha_\nu$ to $\alpha_\mu$.
Therefore, $W_{0,M_1}$ acts transitively on $\Sigma_{U_{P_1}}$.

\item Assume $1 <r <6$. The roots $\beta_1 = \sum_{k=1}^5 \alpha_k = \frac12
(e_8 -e_7 - e_6 + e_1+e_2-e_3+e_4-e_5)$ and $\beta_2 = \alpha_2 +\alpha_3
+2\alpha_4 +\alpha_5 + \alpha_6 = e_5+e_3$ both lie in $\Sigma_{U_{P_r}}$ and
satisfy $(\beta_1,\beta_2) = -1$. Hence, $\alpha_r$ is not non-obtuse. Notice
that $\beta_1 = w_r(\alpha_r)$, where $w_r\in W_{0,M_r}$ is given by
\begin{align*}
w_2 &= s_5s_1s_3s_4,\\
w_3 &= s_5s_1s_2s_4,\\
w_4 &= s_5s_1s_3s_2,\\
w_5 &= s_1s_2s_3s_4.
\end{align*}
Clearly, $w_r^{-1}(\beta_2) \in \Sigma_{U_{P_r}}$ and $(\alpha_r,
w_r^{-1}(\beta_2)) = (\beta_1,\beta_2) = -1$. 

The highest root is $\alpha_0 = \frac12(e_8-e_7-e_6+e_1+e_2+e_3+e_4+e_5) =
\alpha_1 + 2\alpha_2+2\alpha_3+3\alpha_4+2\alpha_5+\alpha_6$. Then $\alpha_0$
and $\alpha_r$ both lie in $\Sigma_{U_{P_r}}$ and have the same length. But
since $c_r(\alpha_0) > 1 = c_r(\alpha_r)$, it follows that $\alpha_0$ does not
lie in the $W_{0,M_r}$-orbit of $\alpha_r$.
\end{enumerate}
\item[($E_7$)]\label{class-E7} 
Inside $V = \set{(x_1,\dotsc,x_8)\in \R^8}{x_7=-x_8}$ the root system of type
$E_7$ is
\[
\begin{split}
\Sigma =& \set{\pm e_i\pm e_j}{1\le i<j\le 6} \cup \{\pm (e_8-e_7)\}\\
&\cup \set{\pm \frac12 \Bigl(e_8-e_7 +\sum_{i=1}^6(-1)^{\nu(i)}
e_i\Bigr)}{\nu(i)\in \Z/2\Z,\; \sum_{i=1}^6\nu(i)\neq 0}.
\end{split}
\]
A basis is given by $\alpha_1 = \frac12(e_1+e_8) - \frac12(e_2+e_3+\dotsb+e_7)$,
$\alpha_2 = e_2+e_1$, and $\alpha_i = e_{i-1}-e_{i-2}$, for $3\le i\le 7$. We
distinguish the following cases:
\begin{enumerate}[label=--]
\item Assume $r=7$. Then
\[
\begin{split}
\Sigma_{U_{P_7}} = &\set{e_6\pm e_i}{1\le i\le 5} \cup \{e_8-e_7\}\\
&\cup  \set{\frac12\Bigl(e_8-e_7+e_6+ \sum_{i=1}^5(-1)^{\nu(i)} e_i\Bigr)}
{\nu(i)\in \Z/2\Z,\; \sum_{i=1}^5\nu(i)\neq 0}.
\end{split}
\]
(These are all the positive roots not lying in the subroot system of type
$E_6$.) To ease the notation we write $\alpha_\nu\coloneqq \frac12 (e_8-e_7 +
\sum_{i=1}^6(-1)^{\nu(i)}e_i)$, for each $\nu = (\nu(i))_i\in (\Z/2\Z)^6$. For
all $1\le i,j\le 5$ and $\nu\in (\Z/2\Z)^6$ with $\alpha_\nu\in
\Sigma_{U_{P_7}}$ we compute 
\begin{align*}
(e_6\pm e_i, e_6\pm e_j) &= 1\pm \delta_{ij}\ge0,\\
(e_6\pm e_i, e_8-e_7) &= 0,\\
(e_6\pm e_i,\alpha_\nu) &= \frac12 \bigl(1\pm (-1)^{\nu(i)}\bigr) \ge0,\\
(e_8-e_7,\alpha_\nu) &= 1.
\end{align*}
Let now $\nu,\mu\in (\Z/2\Z)^6$ with $\nu(6)= \mu(6) = 0$ and
$\sum_{i=1}^5\nu(i) = \sum_{i=1}^5 \mu(i)\neq 0$. As
$\sum_{i=1}^5(\nu(i)+\mu(i)) = 0$, we observe that the cardinality of the set
$\set{1\le i\le 6}{\nu(i)\neq \mu(i)}$ is even, but not $6$, hence equals $0$,
$2$, or $4$. 
But then $\lvert\set{1\le i\le 6}{\nu(i)=\mu(i)}\rvert$ is either $2$, $4$, or
$6$. Thus, we compute
\[
(\alpha_\nu,\alpha_\mu) = \frac14 \bigl(2 + \lvert\set{i}{\nu(i) = \mu(i)}\rvert
- \lvert \set{i}{\nu(i)\neq \mu(i)}\rvert\bigr) \ge 0.
\]
Therefore, $\alpha_7$ is non-obtuse. The Weyl group $W_{0,M_7}$ is the group
generated by $s_1$ and the group $\Gamma\rtimes\frakS_5$ of type $D_5$ (acting
on $\{\pm e_1,\dotsc,\pm e_5\}$ while leaving $e_6,e_7$, and $e_8$ fixed).
Given $\nu,\mu\in \Sigma_{U_{P_7}}$, we may view $\mu-\nu$ as an element of
$\Gamma$ (by forgetting the last entry) which maps $\alpha_\nu$ to $\alpha_\mu$.
Moreover, $\Gamma\rtimes \frakS_5$ clearly acts transitively on $\set{e_6\pm
e_i}{1\le i\le 5}$. Together with
\begin{align*}
s_1(e_6-e_1) &= e_6-e_1 + \alpha_{(0,1,1,1,1,1)} = \alpha_{(1,1,1,1,1,0)},
\quad\text{and}\\
s_1(\alpha_{(1,0,0,0,0,0)}) &= \alpha_{(1,0,0,0,0,0)} + \alpha_{(0,1,1,1,1,1)} =
e_8-e_7,
\end{align*}
and the fact that $\Gamma$ acts transitively on the set of $\alpha_\nu$ with
$\sum_{i=1}^5\nu(i)\neq 0$ and $\nu(6) = 0$, it follows that $\Sigma_{U_{P_7}}$
is the $W_{0,M_7}$-orbit of $\alpha_7$.
Hence, $W_{0,M_7}$ acts transitively on $\Sigma_{U_{P_7}}$.

\item Assume $1\le r <7$. The roots $\beta_1 = \sum_{k=1}^6\alpha_k =
\alpha_{(0,0,1,1,0,1)}$ and $\beta_2 = \alpha_1+\alpha_2 + 2\alpha_3 + 3\alpha_4
+ 2\alpha_5 +\alpha_6 + \alpha_7 = \alpha_{(1,1,0,0,1,0)}$ both lie in
$\Sigma_{U_{P_r}}$ and satisfy $(\beta_1,\beta_2) = -1$. Hence, $\alpha_r$ is
not non-obtuse. Note that $\beta_1 = w_r(\alpha_r)$, where $w_r\in W_{0,M_r}$ is
given by
\begin{align*}
w_1 &= s_6s_5s_2s_4s_3,\\
w_2 &= s_6s_5s_1s_3s_4,\\
w_3 &= s_6s_5s_2s_4s_1,\\
w_4 &= s_6s_5s_1s_3s_2,\\
w_5 &= s_6s_1s_3s_2s_4,\\
w_6 &= s_1s_2s_3s_4s_5.
\end{align*}
Clearly, $w_r^{-1}(\beta_2)\in \Sigma_{U_{P_r}}$ and $(\alpha_r,
w_r^{-1}(\beta_2)) = (\beta_1,\beta_2) = -1$.

The highest root is $\alpha_0 = e_8-e_7 = 2\alpha_1 + 2\alpha_2 + 3\alpha_3
+4\alpha_4 + 3\alpha_5 + 2\alpha_6 + \alpha_7$. Then $\alpha_0$ and $\alpha_r$
both lie in $\Sigma_{U_{P_r}}$ and have the same length. But since
$c_r(\alpha_0) > 1 = c_r(\alpha_r)$, it follows that $\alpha_0$ does not lie in
the $W_{0,M_r}$-orbit of $\alpha_r$.
\end{enumerate}
\item[($E_8$)]\label{class-E8} 
Inside $V = \R^8$ the root system of type $E_8$ is
\[
\begin{split}
\Sigma = &\set{\pm e_i\pm e_j}{1\le i<j\le 8}\\
&\cup \set{\pm\frac12 \Bigl(e_8 + \sum_{i=1}^7 (-1)^{\nu(i)}e_i\Bigr)}
{\nu(i)\in\Z/2\Z,\; \sum_{i=1}^7\nu(i) = 0}.
\end{split}
\]
A basis is given by $\alpha_1 = \frac12 (e_1+e_8) - \frac12(e_2+e_3+\dotsb +
e_7)$, $\alpha_2 = e_2+e_1$, and $\alpha_i = e_{i-1} - e_{i-2}$, for $3\le i\le
8$. 

Let $1\le r\le 8$. The roots $\beta_1 = \sum_{k=1}^8\alpha_k =
\frac12(e_8+e_7-e_6+e_1+e_2-e_3-e_4-e_5)$ and $\beta_2 = \alpha_1 + 2\alpha_2 +
3\alpha_3 + 5\alpha_4 + 4\alpha_5 + 3\alpha_6 + 2\alpha_7+\alpha_8 = \frac12
(e_8+e_7+e_6-e_1-e_2+e_3+e_4+e_5)$ both lie in $\Sigma_{U_{P_r}}$ and satisfy
$(\beta_1,\beta_2) = -1$. Hence, $\alpha_r$ is not non-obtuse. Note that
$\beta_1 = w_r(\alpha_r)$, where $w_r\in W_{0,M_r}$ is given by
\begin{align*}
w_1 &= s_8s_7s_6s_5s_2s_4s_3,\\
w_2 &= s_8s_7s_6s_5s_1s_3s_4,\\
w_3 &= s_8s_7s_6s_5s_2s_4s_1,\\
w_4 &= s_8s_7s_6s_5s_1s_3s_2,\\
w_5 &= s_8s_7s_6s_2s_1s_3s_4,\\
w_6 &= s_8s_7s_2s_1s_3s_4s_5,\\
w_7 &= s_8s_2s_1s_3s_4s_5s_6,\\
w_8 &= s_2s_1s_3s_4s_5s_6s_7.
\end{align*}
Clearly, $w_r^{-1}(\beta_2)\in \Sigma_{U_{P_r}}$ and $(\alpha_r,
w_r^{-1}(\beta_2)) = (\beta_1,\beta_2) = -1$.

The highest root is $\alpha_0 = e_8+e_7 = 2\alpha_1 + 3\alpha_2 + 4\alpha_3 +
6\alpha_4 + 5\alpha_5 +4\alpha_6 +3\alpha_7 +2\alpha_8$. Then $\alpha_r$ and
$\alpha_0$ both lie in $\Sigma_{U_{P_r}}$ and have the same length. But since
$c_r(\alpha_0) > 1 = c_r(\alpha_r)$, it follows that $\alpha_0$ does not lie in
the $W_{0,M_r}$-orbit of $\alpha_r$.
\item[($F_4$)]\label{class-F4} 
Inside $V = \R^4$ the root system of type $F_4$ is
\[
\Sigma = \set{\pm e_i}{1\le i\le 4} \cup \set{\pm e_i\pm e_j}{1\le i<j\le 4}
\cup \{\tfrac12 (\pm e_1 \pm e_2 \pm e_3 \pm e_4)\}.
\]
A basis is given by $\alpha_1 = e_2-e_3$, $\alpha_2 = e_3-e_4$, $\alpha_3 =
e_4$, and $\alpha_4 = \frac12(e_1-e_2-e_3-e_4)$. We have
\begin{align*}
(\alpha_1,\alpha_1+3\alpha_2+4\alpha_3+2\alpha_4) &= (e_2-e_3,e_1+e_3) = -1,\\
(\alpha_2,\alpha_1+\alpha_2+2\alpha_3+2\alpha_4) &= (e_3-e_4,e_1-e_3) = -1,\\
(\alpha_3, \alpha_1+2\alpha_2+2\alpha_3+2\alpha_4) &= (e_4,e_1-e_4) = -1,\\
(\alpha_4,\alpha_1+2\alpha_2+3\alpha_3+\alpha_4) &= \frac14\cdot
(e_1-e_2-e_3-e_4,e_1+e_2+e_3+e_4) = -\frac12.
\end{align*}
Hence, none of the $\alpha_1,\dotsc,\alpha_4$ is non-obtuse. Consider the
following cases:
\begin{enumerate}[label=--]
\item Assume $r=1$ or $r=2$. The highest root is $\alpha_0 = e_1+e_2 = 2\alpha_1
+ 3\alpha_2 +4\alpha_3+2\alpha_4$. Then both $\alpha_0$ and $\alpha_r$ lie in
$\Sigma_{U_{P_r}}$ and have the same length. But since $c_r(\alpha_0) > 1 =
c_r(\alpha_r)$, it follows that $\alpha_0$ does not lie in the $W_{0,M_r}$-orbit
of $\alpha_r$.

\item Assume $r=3$ or $r=4$. The highest short root is $\alpha_0^r = e_1 =
\alpha_1 + 2\alpha_2+3\alpha_3+2\alpha_4$. Then $\alpha_0^r$ and $\alpha_r$ both
lie in $\Sigma_{U_{P_r}}$ and have the same length. But since $c_r(\alpha_0^r)>1
= c_r(\alpha_r)$, it follows that $\alpha_0^r$ does not lie in the
$W_{0,M_r}$-orbit of $\alpha_r$.
\end{enumerate}
\item[($G_2$)]\label{class-G2} 
Inside $V = \set{(x,y,z)\in \R^3}{x+y+z=0}$ the root system of type $G_2$ is
\[
\Sigma = \pm \{e_1-e_2, e_2-e_3,e_1-e_3, 2e_1-e_2-e_3, 2e_2-e_1-e_3,
2e_3-e_1-e_2\}.
\]
A basis is given by $\alpha_1 = e_1-e_2$ and $\alpha_2 = -2e_1+e_2+e_3$. Since
\begin{alignat}{3}
(\alpha_1,\alpha_1+\alpha_2) &= (e_1-e_2, e_3-e_1) = -1 &\hspace{.5em}
\text{and}\\
(\alpha_2,3\alpha_1+\alpha_2) &= (-2e_1+e_2+e_3,e_1-2e_2+e_3) = -3, 
\end{alignat}
neither $\alpha_1$ nor $\alpha_2$ are non-obtuse.

The highest root is $\alpha_0 = 3\alpha_1+2\alpha_2 = -e_1-e_2+2e_3$. Then both
$\alpha_0$ and $\alpha_2$ lie in $\Sigma_{U_{P_2}}$ and have the same length.
But since $c_2(\alpha_0) = 2\neq 1 = c_2(\alpha_2)$, it follows that $\alpha_0$
does not lie in the $W_{0,M_2}$-orbit of $\alpha_2$.

Similarly, the highest short root is $\alpha_0^1 = 2\alpha_1+\alpha_2 =
e_3-e_2$, and both $\alpha_0^1$ and $\alpha_1$ lie in $\Sigma_{U_{P_1}}$ and
have the same length. But since $c_1(\beta) = 2\neq 1 = c_1(\alpha_1)$, it
follows that $\alpha_0^1$ does not lie in the $W_{0,M_1}$-orbit of $\alpha_1$.
\end{enumerate} 
\end{proof} 

We end this section by applying the previous analysis to prove a result on the
ordering of positive roots that will become useful later. First, we need a
preliminary lemma which also appeared in
\cite[Lemma~(2.1.1)]{Macdonald.1971}. For the convenience of the reader we
supply the simple proof.

We make the following convention: if $\Delta'$ is a basis of $\Sigma$, we denote
by $\Sigma^+_{\Delta'}$ (resp. $\Sigma^-_{\Delta'}$) the system of positive
(resp. negative) roots with respect to $\Delta'$.

Recall, \cite[Ch.\,VI, \S1.6, Cor.\,3 of Prop.\,17]{Bourbaki.1981}, that there
exists a unique longest element $w_0\in W_0$. It satisfies $w_0^2 = 1$ and
$\ell(ww_0) = \ell(w_0) - \ell(w)$, for all $w\in W_0$.

\begin{lem}\label{lem:rootorder} 
Let $w_0 = s_{i_1}\dotsm s_{i_r}$ be a reduced decomposition and put
$\beta_j\coloneqq s_{i_1}\dotsm s_{i_{j-1}}(\alpha_{i_j})$, for $j=1,\dotsc,r$.
Then one has
\[
\Sigma_{s_{i_1}\dotsm s_{i_j}(\Delta)}^+ = \{\beta_{j+1}, \beta_{j+2},\dotsc,
\beta_r, -\beta_1, -\beta_2,\dotsc,-\beta_j\},\qquad \text{for all $0\le j\le
r$.}
\]
\end{lem} 
\begin{proof} 
Write $w_j = s_{i_1}\dotsm s_{i_j}$. Applying \cite[Ch.\,VI, \S1.6, Cor.\,2 of
Prop.\,17]{Bourbaki.1981} to $w_j^{-1} = s_{i_j}\dotsm s_{i_1}$ yields
\[
\Sigma^+_{\Delta} \cap \Sigma^-_{w_j(\Delta)} = \Sigma^+_{\Delta} \cap w_j
\Sigma^-_{\Delta} = \{\beta_1,\beta_2,\dotsc,\beta_j\}.
\]
Note that $\Sigma^+_{\Delta} = \Sigma^+_{\Delta} \cap \Sigma^-_{w_0(\Delta)} =
\{\beta_1,\beta_2,\dotsc,\beta_r\}$. Hence, the assertion follows from
\[
\Sigma^+_{w_j(\Delta)} = (\Sigma^-_{\Delta} \cap \Sigma^+_{w_j(\Delta)}) \sqcup
(\Sigma^+_{\Delta} \cap \Sigma^+_{w_j(\Delta)}) = -(\Sigma^+_\Delta \cap
\Sigma^-_{w_j(\Delta)}) \sqcup (\Sigma^+_{\Delta} \setminus
\Sigma^-_{w_j(\Delta)}).\qedhere
\]
\end{proof} 

\begin{ex*} 
It is instructive to visualize an example. The orderings of the positive
roots in $\Sigma^+$ obtained in Lemma~\ref{lem:rootorder} generalize the
``circular orderings'' one has for root systems of rank 2. Assume $\Sigma$ is of
type $G_2$ with basis $\{\alpha_1,\alpha_2\}$ such that $\alpha_2$ is the long
root. The ordering of $\Sigma^+$ corresponding to the reduced decomposition $w_0
= s_1s_2s_1s_2s_1s_2$ is shown in Figure~\ref{fig:circorder}.
\begin{figure*}[ht] 
\centering
\begin{tikzpicture}
\def\radius{1}
\foreach \x in {0,60,...,300}
{
\draw[thick] (0,0) -- (\x:\radius);
\fill (\x:\radius) circle (1pt);
\draw[thick] (0,0) -- (\x+30:1.732*\radius);
\fill (\x+30:1.732*\radius) circle (1pt);
}
\node[right] at (0:\radius) {$\beta_1 = \alpha_1$};
\node[right] at (30:1.732*\radius) {$\beta_2$};
\node[above] at (60:\radius) {$\beta_3$};
\node[right] at (90:1.732*\radius) {$\beta_4$};
\node[above] at (120:\radius) {$\beta_5$};
\node[left] at (150:1.732*\radius) {$\beta_6 = \alpha_2$};

\end{tikzpicture}
\caption{The circular ordering in type $G_2$}\label{fig:circorder}
\end{figure*}
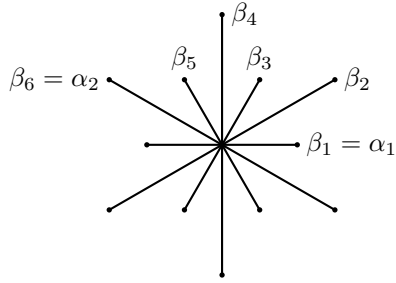 
\end{ex*} 

\begin{cor}\label{cor:rootorder} 
Let $\alpha_i\in \Delta = \{\alpha_1,\dotsc,\alpha_n\}$ be a non-obtuse simple
root, and let $\alpha\in \Sigma_{U_{P_i}}$ such that $\alpha$ and $\alpha_i$
have the same length. There
exists a reduced decomposition $w_0 = s_{i_1}\dotsm s_{i_r}$ such that, if we
put $\beta_j\coloneqq s_{i_1}\dotsm s_{i_{j-1}}(\alpha_{i_j})$, there exists
$0\le l < r$ with $\beta_1,\dotsc,\beta_l\in \Sigma_{M_i}$ and $\beta_{l+1} =
\alpha$. In particular,
\[
\Sigma_{U_{P_i}} \setminus \{\alpha\} 
\subseteq \Sigma^+_{s_{i_1}\dotsm s_{i_{l+1}}(\Delta)} 
= \{\beta_{l+2},\dotsc, \beta_r, -\beta_1,\dotsc,-\beta_l, -\alpha\}.
\]
\end{cor} 

\begin{rmk*} 
Corollary~\ref{cor:rootorder} says geometrically that, for non-obtuse
$\alpha_i$, the roots of length $\lVert \alpha_i\rVert$ are extremal in the cone
generated by $\Sigma_{U_{P_i}}$.

The statement of Corollary~\ref{cor:rootorder} is generally false if $\alpha_i$
is not non-obtuse, see Figure~\ref{fig:circorder}.
\end{rmk*} 

\begin{proof}[Proof of Corollary~\ref{cor:rootorder}] 
Denote $w_{0,M_i}$ the longest element in $W_{0,M_i}$. Since $\alpha_i$ is
non-obtuse, we find by
Proposition~\ref{prop:classification}.\ref{prop:class-iii} an element $w\in
W_{0,M_i}$ with $w(\alpha_i)= \alpha$. Choose reduced decompositions 
\[
w = s_{i_1}\dotsm s_{i_{l}}\qquad \text{and}\qquad 
w_{0,M_i}w_0 = s_{i_{l+1}}\dotsm s_{i_{r'}}.
\]
For each $v\in W_{0,M_i}$ we compute
\begin{equation}\label{eq:rootorder-1}
\ell(vw_{0,M_i}w_0) = \ell(w_0) - \ell(vw_{0,M_i}) = \ell(w_0) - \ell(w_{0,M_i})
+ \ell(v) = \ell(v) + \ell(w_{0,M_i}w_0).
\end{equation}
In particular, $ww_{0,M_i}w_0 = s_{i_1}\dotsm s_{i_{r'}}$ is a reduced
decomposition. We further observe $s_{i_{l+1}} = s_i$, for otherwise we would
have $s_{i_{l+1}} \in W_{0,M_i}$ and $\ell(s_{i_{l+1}}w_{0,M_i}w_0) <
\ell(w_{0,M_i}w_0)$, contradicting \eqref{eq:rootorder-1} for $v = s_{i_{l+1}}$.
Now, if we pick a reduced decomposition $(ww_{0,M_i}w_0)^{-1}w_0 =
s_{i_{r'+1}}\dotsm s_{i_r}$, then it is clear that we obtain a reduced
decomposition $w_0 = s_{i_1}\dotsm s_{i_r}$. From the construction it is clear
that $\beta_1,\dotsc,\beta_l\in \Sigma_{M_i}$ and $\beta_{l+1} = \alpha$. The
last statement is a consequence of Lemma~\ref{lem:rootorder}.
\end{proof} 
\section{The algorithm}\label{sec:algo} 
Recall the special parahoric subgroup $K$ of $G$ associated with $\varphi$.
Given $z\in Z^-$ and $z'\in Z$, it is of general interest to understand the
intersection of the Iwasawa double coset $Uz'K$ and the Cartan double coset
$KzK$. For example, it is well-known that $\nu(z')\le \nu(z)$ provided $Uz'K
\cap KzK
\neq\emptyset$, see \S\ref{subsec:Cartan}.

There is, however, very little known about the $u\in U$ such that $uz' \in KzK$.
One of the main goals of this article is to study the following question: if
$uz'\in KzK$ and if we write $u = u_{\gamma_1}\dotsm u_{\gamma_r}$, with
$u_{\gamma_i}\in U_{\gamma_i}$,
what can be said about the valuations $\varphi_{\gamma_i}(u_{\gamma_i})$? We
will prove that for each strictly positive element $a\in Z$ with $\nu(z) \le
\nu(a^{-1})$, the valuation $\varphi_{\gamma_i}(u_{\gamma_i})$ is bounded below
by $\langle \gamma_i,\nu(a)\rangle$, see Theorem~\ref{thm:main}.

In this section we present an algorithm that gives information about the
$\varphi_{\gamma_i}(u_{\gamma_i})$. First, we need to set up some notation. We
fix a reduced decomposition $w_0 = s_{i_1}\dotsm s_{i_r}$ of the longest
element $w_0$ of $W_0$.

\begin{notation}\label{nota:cycling} 
Recall the reduced root system $\Sigma$ associated with $\Phi$. Let $\Delta =
\{\alpha_1,\dotsc,\alpha_n\}$ be the fixed basis of $\Sigma$.
\begin{enumerate}[label=(\alph*)]
\item\label{nota:cycling-a} Put $\beta_j\coloneqq
s_{i_1}\dotsm s_{i_{j-1}} (\alpha_{i_j})$ and $\Delta^{(j)} = s_{i_1}\dotsm
s_{i_j}(\Delta)$, for $1\le j\le r$. Thanks to Lemma~\ref{lem:rootorder} we
have
\[
\Sigma^{(j)}\coloneqq \Sigma^+_{\Delta^{(j)}} = \{\beta_{j+1},
\beta_{j+2},\dotsc, \beta_r, -\beta_1, - \beta_2,\dotsc, -\beta_j\},\qquad
\text{for $0\le j\le r$.}
\]

Given $k\ge 0$, we let $j(k)\in \{1,\dotsc,r\}$ be the unique integer with
$k\equiv j(k) \pmod r$. Let
\[
\varepsilon_k \coloneqq \begin{cases}
1, & \text{if $k\equiv j(k) \pmod{2r}$,}\\
-1, & \text{otherwise.}
\end{cases}
\]
Put $\beta_k \coloneqq \varepsilon_k\beta_{j(k)}$ and $\Delta^{(k)} \coloneqq
\varepsilon_k \Delta^{(j(k))}$, and then
\[
\Sigma^{(k)} \coloneqq \varepsilon_k \Sigma^{(j(k))} = \Sigma^+_{\Delta^{(k)}} =
\{\beta_{k+1},\beta_{k+2},\dotsc,\beta_{k+r}\}.
\]
Then $\Sigma^{(0)} = \Sigma$ and the sequence $(\Sigma^{(k)})_{k\in\Z_{\ge0}}$
is $2r$-periodic. See also Figure~\ref{fig:Sigma(k)} below.

\item\label{nota:cycling-b} Let $\alpha\in\Sigma$ and write $\alpha =
\epsilon_\beta\beta$ for the
unique $\beta\in\Phi_{\red}$. Define 
\[
\varphi_\alpha \coloneqq \epsilon_\beta
\varphi_\beta \colon U_{\alpha}^* \to \Z.
\]

\item\label{nota:cycling-c} For each $\alpha\in\Sigma$ we fix a lift $n_\alpha
\in N\cap K$ of $s_\alpha\in W_0$.

\item\label{not:cycling-d} Given a basis $\Delta'$ of $\Sigma$, we denote by
$U_{\Delta'}$ the group generated by $\bigcup_{\alpha\in\Sigma^+_{\Delta'}}
U_\alpha$.
\end{enumerate}
\end{notation} 

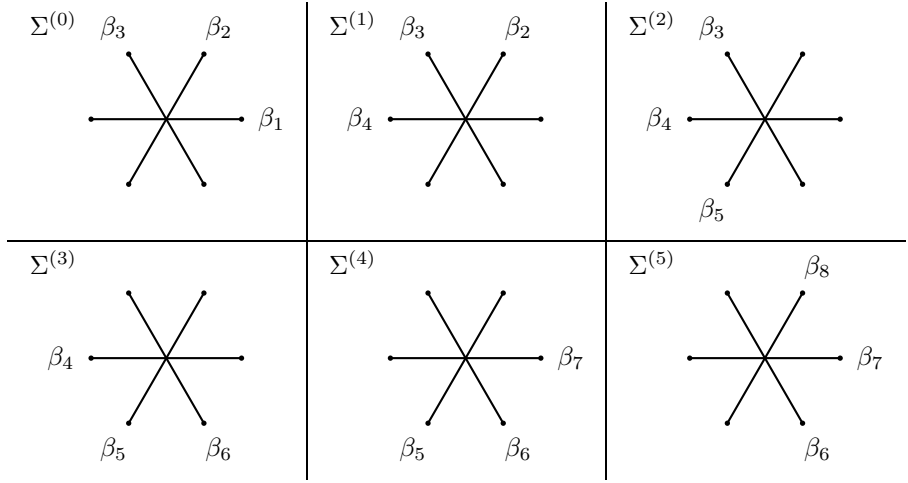
\begin{figure*}[h!] 
\centering
\begin{tabular}{c|c|c}
\begin{tikzpicture} 
\def\radius{1}
\node at (-1.5,1.28) {$\Sigma^{(0)}$};
\foreach \x in {0,60,...,300}
{
\draw[thick] (0,0) -- (\x:\radius);
\fill (\x:\radius) circle (1pt);
}
\foreach \i in {1,2,3}
\node at (\i*60 - 60:1.4*\radius) {$\beta_{\i}$};
\node at (240:1.4*\radius){\phantom{$\beta_5$}};
\end{tikzpicture} 
&
\begin{tikzpicture} 
\def\radius{1}
\node at (-1.5,1.28) {$\Sigma^{(1)}$};
\foreach \x in {0,60,...,300}
{
\draw[thick] (0,0) -- (\x:\radius);
\fill (\x:\radius) circle (1pt);
}
\foreach \i in {2,3,4}
\node at (\i*60 - 60:1.4*\radius) {$\beta_{\i}$};
\node at (240:1.4*\radius){\phantom{$\beta_5$}};
\node at (0:1.4*\radius){\phantom{$\beta_5$}};
\end{tikzpicture} 
&
\begin{tikzpicture} 
\def\radius{1}
\node at (-1.5,1.28) {$\Sigma^{(2)}$};
\foreach \x in {0,60,...,300}
{
\draw[thick] (0,0) -- (\x:\radius);
\fill (\x:\radius) circle (1pt);
}
\foreach \i in {3,4,5}
\node at (\i*60 - 60:1.4*\radius) {$\beta_{\i}$};
\node at (0:1.4*\radius){\phantom{$\beta_5$}};
\end{tikzpicture} 
\\\hline
\begin{tikzpicture} 
\def\radius{1}
\node at (-1.5,1.28) {$\Sigma^{(3)}$};
\foreach \x in {0,60,...,300}
{
\draw[thick] (0,0) -- (\x:\radius);
\fill (\x:\radius) circle (1pt);
}
\foreach \i in {4,5,6}
\node at (\i*60 - 60:1.4*\radius) {$\beta_{\i}$};
\node at (0:1.4*\radius){\phantom{$\beta_5$}};
\end{tikzpicture} 
&
\begin{tikzpicture} 
\def\radius{1}
\node at (-1.5,1.28) {$\Sigma^{(4)}$};
\foreach \x in {0,60,...,300}
{
\draw[thick] (0,0) -- (\x:\radius);
\fill (\x:\radius) circle (1pt);
}
\foreach \i in {5,6,7}
\node at (\i*60 - 60:1.4*\radius) {$\beta_{\i}$};
\node at (180:1.4*\radius) {\phantom{$\beta_5$}};
\end{tikzpicture} 
&
\begin{tikzpicture} 
\def\radius{1}
\node at (-1.5,1.28) {$\Sigma^{(5)}$};
\foreach \x in {0,60,...,300}
{
\draw[thick] (0,0) -- (\x:\radius);
\fill (\x:\radius) circle (1pt);
}
\foreach \i in {6,7,8}
\node at (\i*60 - 60:1.4*\radius) {$\beta_{\i}$};
\node at (180:1.4*\radius) {\phantom{$\beta_5$}};
\end{tikzpicture} 
\end{tabular}
\caption{A visual aid for the sets $\Sigma^{(k)}$ relative to the reduced
decomposition $w_0 = s_1s_2s_1$, where $\alpha_1 = \beta_1$ and $\alpha_2 =
\beta_3$.}\label{fig:Sigma(k)}
\end{figure*} 

Recall from \S\ref{subsec:Iwahori-Weyl} that $K\cap U_\alpha = U_{(\alpha,0)}$,
for all $\alpha\in\Sigma$.

\begin{algo}\label{algo} 
Let $z,z'\in Z$ and $u\in U$ such that $uz' \in KzK$. We define sequences
$(u^{(k)})_{k\ge 0}$ and $(z^{(k)})_{k\ge0}$ with the following properties:
\begin{enumerate}[label=--]
\item $u^{(0)} = u$ and $z^{(0)} = z'$;
\item $u^{(k)} \in U_{\Delta^{(k)}} \cap U_{\Delta^{(k-1)}}$ and $z^{(k)}\in Z$,
for all $k\ge1$;
\item $u^{(k)}z^{(k)} \in KzK$, for all $k\ge0$.
\end{enumerate}
Suppose we have constructed $u^{(k)}$ and $z^{(k)}$ for some $k\ge0$. Write
\[
u^{(k)} = u_{\beta_{k+r}}^{(k)}\cdot u_{\beta_{k+r-1}}^{(k)} \dotsm
u_{\beta_{k+1}}^{(k)} \in \begin{cases}
U_{\Delta}, & \text{if $k=0$,}\\
U_{\Delta^{(k)}} \cap U_{\Delta^{(k-1)}}, & \text{if $k\ge1$,}
\end{cases}
\]
for uniquely determined $u_{\beta_i}^{(k)} \in U_{\beta_i}$.\footnote{Note that
$u_{\beta_{k+r}}^{(k)} = 1$ unless $k=0$.} Depending on
$\varphi_{\beta_{k+1}}\bigl(u_{\beta_{k+1}}^{(k)}\bigr)$ we distinguish three
cases:
\begin{enumerate}[label=(\textbf{Alg-\arabic*})]
\item\label{algo-1} Case $\varphi_{\beta_{k+1}}\bigl(u^{(k)}_{\beta_{k+1}}\bigr)
+ \big\langle \beta_{k+1}, \nu\bigl(z^{(k)}\bigr)\big\rangle \ge0$. By
\eqref{eq:nu} this is equivalent to
\[
x\coloneqq \bigl(z^{(k)}\bigr)^{-1}\cdot u^{(k)}_{\beta_{k+1}}\cdot z^{(k)} \in
U_{(\beta_{k+1},0)} = U_{\beta_{k+1}}\cap K.
\]
We then define 
\begin{align*}
z^{(k+1)} &\coloneqq z^{(k)},\\
u^{(k+1)}&\coloneqq u^{(k)}\cdot \bigl(u^{(k)}_{\beta_{k+1}}\bigr)^{-1} \in
U_{\Delta^{(k+1)}} \cap U_{\Delta^{(k)}}.
\end{align*}
Then $u^{(k+1)}z^{(k+1)} = u^{(k)}z^{(k)} x^{-1} \in KzK$.

\item\label{algo-2} Case $\varphi_{\beta_{k+1}}\bigl(u^{(k)}_{\beta_{k+1}}\bigr)
\ge0$ and not \ref{algo-1}. Then $u^{(k)}_{\beta_{k+1}} \in K$ and we define
\begin{align*}
z^{(k+1)} &\coloneqq z^{(k)},\\
u^{(k+1)} &\coloneqq \bigl(u^{(k)}_{\beta_{k+1}}\bigr)^{-1}\cdot u^{(k)} \in
U_{\Delta^{(k+1)}} \cap U_{\Delta^{(k)}}.
\end{align*}
The fact that $u^{(k+1)} \in U_{\Delta^{(k+1)}}$ follows from \ref{DR2} and
$\Sigma^{(k)} \setminus\{\beta_{k+1}\} = \Sigma^{(k+1)}\cap \Sigma^{(k)}$.

\item\label{algo-3} Case $f_k\coloneqq
\varphi_{\beta_{k+1}}\bigl(u^{(k)}_{\beta_{k+1}}\bigr) <\min\bigl\{0,
-\big\langle \beta_{k+1}, \nu\bigl(z^{(k)}\bigr)\big\rangle\bigr\}$. Note that
$u^{(k)}_{\beta_{k+1}} \neq 1$. By Lemma~\ref{lem:Ua} there exist unique
$u',u''\in U_{-\beta_{k+1}}$ such that
\[
m^{(k)}\coloneqq u'u^{(k)}_{\beta_{k+1}} u'' \in N.
\]
By \cite[Proposition (6.2.10) (ii)]{Bruhat-Tits.1972} the element $m^{(k)}$ acts
on $\apartment$ as the orthogonal reflection $s_{\beta_{k+1},f_k}$ in the
hyperplane $H_{\beta_{k+1},f_k}$. Observe that the element
\[
z^{(k+1)}\coloneqq m^{(k)} z^{(k)}
n_{\beta_{k+1}}
\]
lies in $Z$, because its image in $W_0= N/Z$ is trivial. Considering how
$z^{(k+1)}$ acts on $\varphi\in \apartment$, we deduce
\begin{equation}\label{eq:algo1}
\nu\bigl(z^{(k+1)}\bigr) = s_{\beta_{k+1},f_k}(\nu\bigl(z^{(k)}\bigr)) =
\nu\bigl(z^{(k)}\bigr) - \bigl(\big\langle \beta_{k+1}, \nu\bigl(z^{(k)}\bigr)
\big\rangle + f_k\bigr)\cdot \beta_{k+1}^\vee.
\end{equation}
Applying $\langle \beta_{k+1},-\rangle$ to this equation, and rearranging, we
obtain
\begin{equation}\label{eq:algo2}
\varphi_{\beta_{k+1}}\bigl(u^{(k)}_{\beta_{k+1}}\bigr) = f_k = -\frac12\cdot
\Bigl( \big\langle\beta_{k+1}, \nu\bigl(z^{(k)}\bigr)\big\rangle +
\big\langle \beta_{k+1}, \nu\bigl(z^{(k+1)}\bigr)\big\rangle \Bigr).
\end{equation}
By \ref{V5} and \eqref{eq:nu} we have
\begin{align*}
\varphi_{-\beta_{k+1}}(u') &= -f_k > 0 \qquad \text{and} \\
\varphi_{-\beta_{k+1}}\bigl(\bigl(z^{(k)}\bigr)^{-1}u'' z^{(k)}\bigr) &= -f_k -
\big\langle \beta_{k+1}, \nu\bigl(z^{(k)}\bigr)\big\rangle > 0.
\end{align*}
This entails that $u'$ and $\bigl(z^{(k)}\bigr)^{-1} u'' z^{(k)}$ lie in $K$. We
now define
\[
u^{(k+1)} \coloneqq u'\cdot u^{(k)}\cdot \bigl(u^{(k)}_{\beta_{k+1}}\bigr)^{-1}
\cdot (u')^{-1} \in U_{\Delta^{(k+1)}} \cap U_{\Delta^{(k)}}.
\]
We remark that $u^{(k+1)}_{\beta_{k+2}} = u^{(k)}_{\beta_{k+2}}$ as can be seen
from \ref{V3} and the fact that $\beta_{k+2}$ is extremal in $\Sigma^{(k+1)}$.
Finally, we compute
\begin{align*}
u^{(k+1)} z^{(k+1)} &= u'\cdot u^{(k)}\cdot
\bigl(u^{(k)}_{\beta_{k+1}}\bigr)^{-1} \cdot (u')^{-1}\cdot m^{(k)} \cdot
z^{(k)} \cdot n_{\beta_{k+1}}\\
&= u'\cdot u^{(k)}z^{(k)}\cdot \bigl(z^{(k)}\bigr)^{-1}u'' z^{(k)}\cdot
n_{\beta_{k+1}} \in KzK.
\end{align*}
\end{enumerate}
\end{algo} 

Note that, as a byproduct, at each step the algorithm provides a lower bound for
$\varphi_{\beta_{k+1}}\bigl(u^{(k)}_{\beta_{k+1}}\bigr)$. It is because of this
property that Algorithm~\ref{algo} will be useful for us later. 

The next result will not be used in the sequel.

\begin{prop}\label{prop:algo} 
Algorithm~\ref{algo} terminates, that is, there exists $l\ge0$ such that
$u^{(l)} = 1$. Moreover, $\nu\bigl(z^{(l)}\bigr)$ lies in the $W_0$-orbit of
$\nu(z)$.
\end{prop} 
\begin{proof} 
Let $u\in U$, $z',z\in Z$ with $uz'\in KzK$. Suppose we are in case~\ref{algo-3}
at the $k$-th step. 
Then $\nu\bigl(z^{(k+1)}\bigr)$ is obtained by
reflecting $\nu\bigl(z^{(k)}\bigr)$ along the hyperplane $H_{\beta_{k+1},f_k}$.
But since $\langle \beta_{k+1},0\rangle + f_k = f_k <0$ and $\big\langle
\beta_{k+1}, \nu\bigl(z^{(k)}\bigr)\big\rangle + f_k < 0$, it follows that $0$
and $\nu\bigl(z^{(k)}\bigr)$ are on the same side of $H_{\beta_{k+1}, f_k}$,
whereas $\nu\bigl(z^{(k+1)}\bigr)$ lies on the other. An elementary argument
in Euclidean geometry now shows $\big\lVert \nu\bigl(z^{(k)}\bigr)\big\rVert <
\big\lVert \nu\bigl(z^{(k+1)}\bigr)\big\rVert$. 

As $\nu(Z)$ is a lattice in $V$, its intersection with the convex polytope $C$
spanned by the $W_0$-orbit of $\nu(z)$ is finite. Since $u^{(k)}z^{(k)} \in
KzK$, we have $\nu\bigl(z^{(k)}\bigr) \in C$, for all $k\ge0$, see
\S\ref{subsec:Cartan}. 
As the $z^{(k)}$ remain unchanged in the cases \ref{algo-1} and \ref{algo-2},
the above discussion shows that there are only finitely many
instances of case~\ref{algo-3}.

Let $k\ge0$ such that $\big\lVert \nu\bigl(z^{(k)}\bigr)\big\rVert$ is maximal.
As only the cases~\ref{algo-1} and \ref{algo-2} occur, it follows that for all
$j\ge0$ we have $z^{(k+j)} = z^{(k)}$ and $u^{(k+j)} \in U_{\Delta^{(k+j)}}
\cap U_{\Delta^{(k)}}$. 
In particular, we have $u^{(k+r)} \in
U_{\Delta^{(k+r)}} \cap U_{\Delta^{(k)}} = U_{-\Delta^{(k)}} \cap
U_{\Delta^{(k)}} = \{1\}$. Hence, Algorithm~\ref{algo} terminates with $l =
k+r$. Moreover, we have $Kz^{(k+r)} K = KzK$ by the construction in
Algorithm~\ref{algo} and the fact that $u^{(k+r)}=1$, and hence the last
assertion follows from the Cartan decomposition~\ref{cartan}.
\end{proof} 

We are now ready to prove our main technical result, which may be of
independent interest.

\begin{thm}\label{thm:main} 
Let $\alg P = \alg U_{\alg P}\alg M$ be a non-obtuse parabolic. Let $a\in Z$ be
strictly $M$-positive. Let $u\in U_P$, $z\in Z^-$, and $z'\in Z$ such that
$\nu(z) \le \nu(a^{-1})$ and $uz'\in KzK$. Then the following assertions hold:
\begin{enumerate}[label=(\roman*)]
\item\label{thm:main-i} $az'\in M^+$;
\item\label{thm:main-ii} $aua^{-1}\in K_P = K\cap P$.
\end{enumerate}
\end{thm} 
\begin{proof} 
Note that $uz'\in KzK$ implies $w.\nu(z')\le \nu(z)$ for all $w\in W_0$, see
Remark~\ref{rmk:Iwasawa}.\ref{iwasawa-c}.
Let $\lambda$ (resp. $\mu$) be the image of $a$ (resp. $z'$) in $\Lambda$. Then
\ref{thm:main-i} is equivalent to
\begin{equation}
\big\langle \alpha, \nu(\lambda + \mu)\big\rangle \le 0,\qquad \text{for all
$\alpha\in \Sigma^+\setminus\Sigma_M$}
\end{equation}
(cf. \eqref{eq:LambdaM+}).
But this follows from Lemma~\ref{lem:no-cone}, since by assumption $\lambda$ is
strictly $M$-positive and $\nu\bigl(w(\mu)\bigr) \le \nu(z) \le \nu(-\lambda)$
for all $w\in W_0$.
\bigskip

We now prove \ref{thm:main-ii}. As $\alg P$ is maximal parabolic, the roots
appearing in $\alg U_{\alg P}$ are contained in a single irreducible component
$\Phi_1$ of $\Phi$. Since all computations will be done in the subgroup of $G$
generated by $Z$ and $U_\alpha$, for $\alpha\in\Phi_1$, we may assume for
notational convenience that $\Phi$ (and hence $\Sigma$) is irreducible.
As in \S\ref{sec:non-obtuse}
we write $\alpha_1,\dotsc, \alpha_n$ for the simple roots in $\Sigma$ and put
$\Sigma_{U_P} = \Sigma^+ \setminus \Sigma_M$. Let $w_0$ (resp. $w_{0,M}$)
be the longest element in $W_0$ (resp. $W_{0,M}$). Denote $\alpha_0$ the
highest root of $\Sigma$ and write $\alpha_0 = \sum_{i=1}^n c_i(\alpha_0)
\alpha_i$.

Write $u = \prod_{\alpha\in\Sigma_{U_P}}u_\alpha$ for some ordering of the
factors (to be specified later). Since we have $K_P\cap U_P =
\prod_{\alpha\in\Sigma_{U_P}}U_{(\alpha,0)}$, and because
$U_{\beta,0} \cap U_{2\beta} = U_{2\beta,0}$ whenever $\beta,2\beta\in \Phi$ (by
\ref{V4}), it suffices to prove $\varphi_{\alpha}(au_{\alpha}a^{-1}) =
\varphi_\alpha(u_\alpha) - \langle \alpha, \nu(a)\rangle \ge 0$, that is,
\begin{equation}\label{eq:inequ}
\varphi_\alpha(u_\alpha) \ge \langle \alpha, \nu(a)\rangle,\qquad \text{for all
$\alpha\in\Sigma_{U_P}$.}
\end{equation}

The general procedure is as follows: We fix an ordering $o$ of $\Sigma_{U_P}$
with respect to which we write $u =
\prod_{\alpha\in\Sigma_{U_P}}u_{\alpha}$. For each $\alpha
\in \Sigma_{U_P}$ we apply Algorithm~\ref{algo} in order to estimate
$\varphi_\alpha(u_\alpha)$. This necessitates to
temporarily consider a different ordering, and we need to ensure
that in the notation of Algorithm~\ref{algo} we have $\varphi_\alpha(u_\alpha)
= \varphi_\alpha(u_{\beta_{k+1}}^{(k)})$, for the minimal $k\ge0$ for which
$\beta_{k+1} = \alpha$. (In many cases we will even have $u_\alpha =
u_{\beta_{k+1}}^{(k)}$.)
The next step in the algorithm then provides the desired estimate for
$\varphi_{\alpha}(u_\alpha)$. Finally, we go back to the initial ordering $o$
and repeat this procedure with another root of $\Sigma_{U_P}$.

\begin{enumerate}[label=(\alph*)] 
\item\label{thm:main-a} Let $w_0 = s_{i_1}\dotsm s_{i_r}$ be a reduced
decomposition and apply Algorithm~\ref{algo}. At the $k$-th step we have
\[
u^{(k)} = u^{(k)}_{\beta_{k+r}} u^{(k)}_{\beta_{k+r-1}} \dotsm
u^{(k)}_{\beta_{k+1}}.
\]
Note that, by \ref{thm:main-i}, we have $az^{(k)} \in
M^+$, for all $k\ge0$. Assume $\beta_{k+1}\in \Sigma_{U_P}$, so that
$\langle \beta_{k+1}, \nu(az^{(k)})\rangle \le 0$. 
In case~\ref{algo-1} this implies
\[
\varphi_{\beta_{k+1}}\bigl(u^{(k)}_{\beta_{k+1}}\bigr) \ge -\big\langle
\beta_{k+1}, \nu\bigl(z^{(k)}\bigr)\big\rangle \ge \langle \beta_{k+1},
\nu(a)\rangle.
\]
In case~\ref{algo-2} we estimate
$\varphi_{\beta_{k+1}}\bigl(u^{(k)}_{\beta_{k+1}}\bigr) \ge 0 \ge \langle
\beta_{k+1}, \nu(a)\rangle$.
If, however, we are in case~\ref{algo-3}, then \eqref{eq:algo2} implies
\[
\varphi_{\beta_{k+1}}\bigl(u^{(k)}_{\beta_{k+1}}\bigr) = -\frac12\cdot
\Bigl(\big\langle \beta_{k+1}, \nu\bigl(z^{(k)}\bigr)\big\rangle + \big\langle
\beta_{k+1}, \nu\bigl(z^{(k+1)}\bigr)\big\rangle \Bigr) \ge \langle \beta_{k+1},
\nu(a)\rangle.
\]
Thus, whenever $\beta_{k+1}\in \Sigma_{U_P}$, we have
\begin{equation}\label{eq:u(k)-estimate}
\varphi_{\beta_{k+1}}\bigl(u^{(k)}_{\beta_{k+1}}\bigr) \ge \langle \beta_{k+1},
\nu(a)\rangle.
\end{equation}

\item\label{thm:main-b} 

Let $\alpha_{i_0}$ be the unique simple root in
$\Sigma_{U_P}$. Let $\alpha\in \Sigma_{U_P}$ with the same length as
$\alpha_{i_0}$. By Corollary~\ref{cor:rootorder} we find a reduced decomposition
$w_0 = s_{i_1}\dotsm s_{i_r}$ such that for some $0\le l < r$ we have
$\beta_1,\dotsc,\beta_l\in \Sigma_M$ and $\beta_{l+1} = \alpha$. We apply
Algorithm~\ref{algo} to this reduced decomposition. Note that $u^{(0)}_{\beta_1}
= \dotsb = u^{(0)}_{\beta_l} = 1$. As $\alpha = \beta_{l+1}$ is a simple root in
$\Sigma^+_{s_{i_1}\dotsm s_{i_l}(\Delta)}$ (which contains $\Sigma_{U_P}$) it
follows that $\alpha$ cannot be expressed as the sum of two or more roots in
$\Sigma_{U_P}$. Hence, \ref{DR2} implies $u_\alpha = u^{(0)}_{\beta_{l+1}}$. 
Now, case \ref{algo-1} applies for the first
$l$ steps. Consequently, we have $u_{\beta_{l+1}}^{(l)} = u_\alpha$. Hence,
\eqref{eq:u(k)-estimate} shows $\varphi_{\alpha}(u_\alpha)\ge \langle
\alpha,\nu(a)\rangle$.

This proves \eqref{eq:inequ} in the case where $\alpha$ and $\alpha_{i_0}$ have
the same length. When $\Sigma$ is simply-laced, that is, of type ADE, then all
roots have the same length. This proves \ref{thm:main-ii} in this case.
\end{enumerate} 

It remains to study the cases where $\Sigma$ is of type $B_n$ or $C_n$, and where
$\alpha\in \Sigma_{U_P}$ and $\alpha_{i_0}$ have different lengths.

\begin{enumerate}[label=(\alph*),resume] 
\item\label{thm:main-c} Suppose that $\Sigma$ of type $B_n$ and that $\alg P$
corresponds to $\alpha_n = e_n$ in the notation of \hyperref[class-Bn]{$(B_n)$}
in the proof of Proposition~\ref{prop:classification}. (Note that $\Phi$ is not
necessarily reduced.)

We have
\[
\Sigma_{U_P} = \set{e_i}{1\le i\le n} \cup \set{e_i+e_j}{1\le i<j\le n}.
\]
We now choose a specific ordering of the factors as follows: let $o\colon
\Sigma_{U_P}\xrightarrow\cong \{1,2,\dotsc,\lvert\Sigma_{U_P}\rvert\}$ be a
bijection such that, writing $u = \prod_{i=1}^{\lvert
\Sigma_{U_P}\rvert} u_{o^{-1}(i)}$ with $u_{o^{-1}(i)} \in U_{o^{-1}(i)}$, we
have: $\varphi_{e_i}(u_{e_i}) <\varphi_{e_j}(u_{e_j})$ implies $o(e_i) >
o(e_j)$. 
With this choice of ordering we will prove~\eqref{eq:inequ}. For each
$\alpha\in\Sigma_{U_P}$ we will apply Algorithm~\ref{algo} to estimate
$\varphi_{\alpha}(u_{\alpha})$. As the algorithm changes the ordering, we
have to ensure that $\varphi_{\alpha}(u_\alpha) =
\varphi_{\alpha}(u^{(0)}_{\alpha})$. 

Observe that $e_i$ cannot be written as
the sum of two or more roots in $\Sigma_{U_P}$. An application of \ref{DR2}
shows that for any ordering the $e_i$-component of $u$ coincides with
$u_{e_i}$.
Therefore, the estimate for $\varphi_{e_i}(u_{e_i})$ is provided
by~\ref{thm:main-b}.

Note that, given $\gamma_1,\gamma_2\in \Sigma_{U_P}$, we have $e_i+e_j =
\gamma_1 + \gamma_2$ only if $\{e_i, e_j\} = \{\gamma_1, \gamma_2\}$. 
An application of \ref{DR2} shows that the $(e_i+e_j)$-component of $u$ in a
reordering depends only on the relative position of $u_{e_i}$ and $u_{e_j}$.
In order to estimate $\varphi_{e_i+e_j}(u_{e_i+e_j})$, we
thus have to ensure that the reordering needed for applying Algorithm~\ref{algo}
does not change the relative position of $u_{e_i}$ and $u_{e_j}$. 

Observe that every reduced decomposition of $w_{0,M}w_0$ necessarily starts with
$s_n s_{n-1}\dotsm$: indeed, this follows, since $s_1,\dotsc, s_{n-1}\in
W_{0,M}$ and $\ell(ww_{0,M}w_0) = \ell(w) + \ell(w_{0,M}w_0)$, for all $w\in
W_{0,M}$, and $s_ns_i = s_is_n$, for all $1\le i\le n-2$. 
Fix $1\le i,j\le n$ with $o(e_i)>o(e_j)$ and choose $w\in W_{0,M}\cong \frakS_n$
such that $w(e_n)= e_i$ and $w(e_{n-1}) = e_j$. 
As in the proof of Corollary~\ref{cor:rootorder} we find a reduced decomposition
$w_0 = s_{i_1}\dotsm s_{i_r}$ such that $s_{i_1}\dotsm s_{i_l}$ is a reduced
decomposition of $w$ (for some $0\le l\le r-2$) and $s_{i_{l+1}} = s_n$ and
$s_{i_{l+2}} = s_{n-1}$. 
In particular, we have
$\beta_1,\dotsc,\beta_l \in \Sigma_M$ and $\beta_{l+1} = e_i$. Since
$s_n(e_{n-1}-e_n) = e_{n-1}+e_n$, we also deduce $\beta_{l+2} =
e_i + e_j$. Note that $e_j = \beta_{l'}$ for some $l' > l+2$.

We apply Algorithm~\ref{algo} to this reduced decomposition and observe that, by
construction, the relative position of $u_{e_i}$, $u_{e_j}$ and $u^{(0)}_{e_i}$,
$u^{(0)}_{e_j}$
is the same; therefore, we have $u_{e_i+e_j} = u^{(0)}_{e_i+e_j}$. Note that
$u^{(l)} = u^{(0)}$ in $U_{P}$, and hence $u^{(l)}_{\beta_{l+1}} =
u^{(0)}_{e_i} = u_{e_i}$
and $u^{(l)}_{\beta_{l+2}} = u^{(0)}_{\beta_{l+2}} = u_{e_i+e_j}$. 
We now prove
\begin{equation}\label{eq:val_ei+ej}
\varphi_{e_i+e_j}(u_{e_i+e_j}) \ge \langle e_i+e_j, \nu(a)\rangle.
\end{equation}
In cases~\ref{algo-1} and \ref{algo-3} we have $u^{(l+1)}_{\beta_{l+2}} =
u^{(l)}_{\beta_{l+2}} = u_{e_i+e_j}$. 
Therefore, \eqref{eq:val_ei+ej} follows from~\eqref{eq:u(k)-estimate}.
Assume that we are in case~\ref{algo-2}, so that
$\varphi_{e_j}(u_{e_j}) \ge \varphi_{e_i}(u_{e_i}) \ge 0$. Then
\[
u_{e_i}^{-1}u_{e_j} = u_{e_j}u_{e_i}^{-1}\cdot [u_{e_i},u_{e_j}^{-1}]
\]
with $[u_{e_i},u_{e_j}^{-1}] \in U_{(e_i+e_j,0)}$, by \ref{V3}. But this means
$u^{(l+1)}_{\beta_{l+2}} = [u_{e_i},u_{e_j}^{-1}]\cdot u_{e_i+e_j}$. 
Therefore, we have either $\varphi_{e_i+e_j}(u_{e_i+e_j}) \ge 0 \ge \langle
e_i+e_j, \nu(a)\rangle$ or we have
$\varphi_{\beta_{l+2}}\bigl(u^{(l+1)}_{\beta_{l+2}}\bigr) =
\varphi_{e_i+e_j}(u_{e_i+e_j})$. In the latter case, \eqref{eq:val_ei+ej}
follows, again, from \eqref{eq:u(k)-estimate}. This proves \ref{thm:main-ii} in
case $\Sigma$ is of type $B_n$ and $\alg P$ corresponds to $\alpha_n$.

\item\label{thm:main-d} If $\Sigma$ is of type $C_n$ and $\alg P$ corresponds to
$\alpha_n = 2e_n$, then a similar argument as in \ref{thm:main-c} applies. The
argument becomes easier, though, since $U_P$ is commutative (use \ref{DR2} and
the fact that $c_n(\alpha_0)= 1$).

\item\label{thm:main-e} Assume that $\Phi$ is of type $BC_n$ and $\alg P$
corresponds to $\alpha_1 = e_1 - e_2$ in the notation of
\hyperref[class-Bn]{$(B_n)$} in the proof of
Proposition~\ref{prop:classification}. The other cases, where $\Phi$ is of type
$B_n$ or $C_n$ (and where $\alg P$ corresponds to $\alpha_1$) are proved in
essentially the same way. 

Note that $\Sigma$ is of type $B_n$ and we have
\[
\Sigma_{U_P} = \{e_1\}\cup \set{e_1\pm e_i}{2\le i\le n}.
\]

We remark that, again by \ref{DR2}, the $u_{e_1\pm e_i}$ do not depend on the
ordering of the factors. By \ref{thm:main-b} we have $\varphi_{e_1\pm
e_i}(u_{e_1\pm e_i}) \ge \langle e_1\pm e_i, \nu(a)\rangle$. It remains to prove
\begin{equation}\label{eq:val u_e1}
\varphi_{e_1}(u_{e_1}) \ge \langle e_1,\nu(a)\rangle.
\end{equation}
Note that if $2\varphi_{e_1}(u_{e_1}) \ge \varphi_{e_1-e_i}(u_{e_1-e_i}) +
\varphi_{e_1+e_i}(u_{e_1+e_i})$, for some $2\le i\le n$ and some ordering of the
factors, then we easily obtain \eqref{eq:val u_e1}. Therefore, we assume from
now on
\begin{equation}\label{eq:val u_e1-2}
2\varphi_{e_1}(u_{e_1}) <\varphi_{e_1-e_i}(u_{e_1-e_i}) +
\varphi_{e_1+e_i}(u_{e_1+e_i})
\end{equation}
for all $2\le i\le n$ and all orderings of the factors.

Given $v\in V$, we denote $s_{v}$ the orthogonal reflection in the hyperplane
orthogonal to $v$.

\begin{claim}\label{claim:main} 
The decomposition $s_{e_1} = (s_1s_2\dotsm s_{n-1})s_n(s_{n-1}s_{n-2}\dotsm
s_1)$ is reduced.
\end{claim} 
\begin{proof} 
We write this decomposition as $s_{i_1}\dotsm s_{i_{2n-1}}$ and put
$\beta_j \coloneqq s_{i_1}\dotsm s_{i_{j-1}}(\alpha_{i_j})$, for all $1\le j\le
2n-1$. Then we have
\[
\beta_j = \begin{cases}
s_1\dotsm s_{j-1}(e_j-e_{j+1}) = e_1 - e_{j+1}, & \text{for $1\le j\le n-1$;}\\
s_1\dotsm s_{n-1}(e_n) = e_1, & \text{for $j=n$.}
\end{cases}
\]
For $1\le j\le n-1$ we compute
\begin{align*}
\beta_{2n-j} &= s_1s_2\dotsm s_js_{j+1}\dotsm s_n s_{n-1}\dotsm
s_{j+1}(e_j-e_{j+1})\\
&= s_1\dotsm s_j s_{e_{j+1}}(e_j-e_{j+1})\\
&= s_1\dotsm s_j(e_j + e_{j+1})\\
&= e_1 + e_{j+1}.
\end{align*}
Therefore, the elements $\beta_1,\dotsc,\beta_{2n-1}$ are pairwise distinct and
\cite[Ch.\,IV, \S1, no.\,4, Lemma\,2]{Bourbaki.1981} shows that $\ell(s_{e_1})
= 2n-1$.
\end{proof} 

We fix a reduced decomposition $s_{i_1}\dotsm s_{i_r}$ of $w_0$ whose initial
piece is $s_1s_2\dotsm s_n s_{n-1}\dotsm s_1$.
Since the $u_\gamma$, for $\gamma\in \Sigma_{U_P}\setminus\{e_1\}$, are
independent of the chosen ordering, we are free to choose a convenient ordering
in order to estimate $\varphi_{e_1}(u_{e_1})$. We take the ordering given by the
fixed reduced decomposition of $w_0$, so that $u_{e_1}= u^{(0)}_{e_1}$, and
apply Algorithm~\ref{algo}.
We need to study the \emph{support} of $u^{(k)}$, that is, the set
$\bigl\{\gamma\in \Sigma\,\mid\,u^{(k)}_\gamma \neq 1\bigr\}$. 
We define recursively
$\Psi^{(0)}\coloneqq \Sigma_{U_P}$ and then $\Psi^{(k)}$ as the
closed\footnote{A subset $X\subseteq \Sigma^{(k)}$ is called \emph{closed} if
$\gamma,\delta\in X$ with $\gamma + \delta \in \Sigma^{(k)}$ implies $\gamma +
\delta \in X$.} subset of $\Sigma^{(k)}$ generated by $\Psi^{(k-1)}
\setminus\{e_1-e_{k+1}\}$ and $e_{k+1} - e_1$, for $1\le k\le n-1$. Concretely,
we have for all $0\le k\le n-1$:
\begin{align*}
\Psi^{(k)} &= \set{e_1\pm e_i}{k+2\le i\le n} \cup \set{e_i\pm e_1}{2\le
i\le k+1}\\
&\quad \cup \set{e_i\pm e_j}{2\le i\le k+1,\; i<j\le n} 
\cup \set{e_i}{1\le i\le k+1}.
\end{align*}

By construction, the support of $u^{(k)}$ is contained in $\Psi^{(k)}$. 

Under the addition map $\Psi^{(k)} \times \Psi^{(k)} \to \R^n$ the preimage of
$\{e_1,2e_1\}$ is the set of pairs $(e_1\pm e_i, e_1 \mp e_i)$, for $k+2\le i\le
n$. The preimage of $e_1\pm e_i$ is empty for $k+2\le i\le n$. Together with our
assumption~\eqref{eq:val u_e1-2} we show that this implies the following claim:

\begin{claim}\label{claim:main-2} 
For all $0\le k\le n-1$ one has:
\begin{enumerate}[label=\arabic*.]
\item $\varphi_{e_1}\bigl(u^{(k)}_{e_1}\bigr) =
\varphi_{e_1}\bigl(u^{(k-1)}_{e_1}\bigr)$;
\item $u^{(k)}_{e_1\pm e_i} = u^{(k-1)}_{e_1\pm e_i}$, for all $k+2\le i\le n$.
\end{enumerate}
(We put $u^{(-1)}_\gamma \coloneqq u_\gamma$, for $\gamma\in \Sigma$.)
\end{claim} 
\begin{proof} 
We prove the claim by induction on $k$, the case $k=0$ being trivial. Assume the
claim holds for all $0\le j\le k$, for some $0\le k \le n-2$. Recall that
$\beta_{k+1} = e_1 -e_{k+2}$.

Assume we are in case~\ref{algo-1}, so that
$\varphi_{e_1-e_{k+2}}\bigl(u^{(k)}_{e_1-e_{k+2}}\bigr) \ge -\big\langle
\beta_{k+1}, \nu\bigl(z^{(k)}\bigr)\big\rangle$. In this case, we clearly have
$u^{(k+1)}_{\gamma} = u^{(k)}_{\gamma}$, for all $\gamma\in \Psi^{(k)}
\setminus\{e_1 - e_{k+2}\}$, which proves the induction step in this case.

Suppose we are in case~\ref{algo-3}. Then $u^{(k+1)}_{e_1\pm e_i} =
u^{(k)}_{e_1\pm e_i}$, for all $k+3\le i\le n$, and $u^{(k+1)}_{e_1} =
u^{(k)}_{e_1}$. This shows the induction step in this case.

Finally, assume we are in case~\ref{algo-2} so that $\varphi_{e_1-e_{k+2}}
\bigl(u^{(k)}_{e_1-e_{k+2}}\bigr) \ge 0$. We then have $u^{(k+1)}_{e_1\pm e_i} =
u^{(k)}_{e_1\pm e_i}$, for all $k+3\le i\le n$. Moreover, we have
\[
u^{(k+1)}_{e_1} = u^{(k)}_{e_1} \cdot \bigl[u^{(k)}_{e_1-e_{k+2}},
\bigl(u^{(k)}_{e_1+e_{k+2}}\bigr)^{-1}\bigr] \in U_{e_1}.
\]
The induction hypothesis implies $\varphi_{e_1}\bigl(u^{(k)}_{e_1}\bigr) =
\varphi_{e_1}(u_{e_1})$ and $u^{(k)}_{e_1\pm e_{k+2}} = u_{e_1\pm e_{k+2}}$.
Using \ref{V4}, \ref{V3}, and~\eqref{eq:val u_e1-2}, we compute
\begin{align*}
2\varphi_{e_1}\bigl(\bigl[u^{(k)}_{e_1-e_{k+2}}, \bigl(u^{(k)}_{e_1 +
e_{k+2}}\bigr)^{-1}\bigr]\bigr) &= \varphi_{2e_1}\bigl(\bigl[
u^{(k)}_{e_1-e_{k+2}}, \bigl(u^{(k)}_{e_1+e_{k+2}}\bigr)^{-1}\bigr]\bigr)\\
&\ge \varphi_{e_1-e_{k+2}}\bigl(u^{(k)}_{e_1-e_{k+2}}\bigr) +
\varphi_{e_1+e_{k+2}} \bigl(u^{(k)}_{e_1+e_{k+2}}\bigr)\\
&> 2 \varphi_{e_1}\bigl(u^{(k)}_{e_1}\bigr).
\end{align*}
Therefore, we conclude $\varphi_{e_1}\bigl(u^{(k+1)}_{e_1}\bigr) =
\varphi_{e_1}\bigl(u^{(k)}_{e_1}\bigr)$. This proves the induction step in this
case and finishes the proof.
\end{proof} 

Now, Claim~\ref{claim:main-2} and \eqref{eq:u(k)-estimate} show
$\varphi_{e_1}(u_{e_1}) = \varphi_{\beta_n}\bigl(u^{(n-1)}_{\beta_n}\bigr) \ge
\langle e_1, \nu(a)\rangle$. This shows \eqref{eq:val u_e1} and finishes the
proof.
\end{enumerate} 
\end{proof} 
\section{Decomposition of Hecke polynomials}\label{sec:decomp} 
We fix a commutative ring $R$ with $1$. In \S\ref{subsec:twisted}
and \S\ref{subsec:decomp} we will assume that $p$ be invertible in $R$.
\subsection{Parabolic Hecke algebras}\label{subsec:parHecke} 
Parabolic Hecke algebras for the general linear and the symplectic group were
introduced and studied by Andrianov, see \cite{Andrianov.1977},
\cite{Andrianov.1979}, and the book \cite{Andrianov.1995}.
\begin{defn} 
Let $\alg P$ be a parabolic subgroup of $\alg G$. Then
\[
\Hecke_R(K_P, P)
\]
is called a \emph{parabolic Hecke algebra}.
\end{defn} 

\begin{lem}\label{lem:parabolic-embedding} 
Let $\alg P$ and $\alg Q$ be (not necessarily proper) parabolic subgroups of
$\alg G$ with $\alg P\subseteq \alg Q$. Then the map
\begin{align*}
\varepsilon_{P,Q}\colon \Hecke_R(K_Q,Q) &\longhookrightarrow \Hecke_R(K_P,P),\\
\sum_i r_i\cdot (K_Qg_i) &\longmapsto \sum_i r_i\cdot (K_Pg_i),
\end{align*}
where one may choose $g_i\in P$, is a well-defined injective $R$-algebra
homomorphism. Moreover, the following diagram is commutative:
\[
\begin{tikzcd}
\Hecke_R(K,G) \ar[r,hook,"\varepsilon_{Q,G}"] \ar[dr,hook,"\varepsilon_{P,G}"']
& \Hecke_R(K_Q,Q) \ar[d,hook,"\varepsilon_{P,Q}"]\\
& \Hecke_R(K_P,P).
\end{tikzcd}
\]
\end{lem} 
\begin{proof} 
Clearly, we have $K_P \subseteq K_Q$ and $K_Q\cap P = K_P$. The Iwasawa
decomposition~\ref{iwasawa} implies $Q = K_QP$. Therefore, the conditions
\eqref{eq:Hecke-embedding} for $(\Gamma, S) = (K_Q,Q)$ and
$(\Gamma_0,S_0) = (K_P,P)$ are satisfied and the first statement follows from
Proposition~\ref{prop:Hecke-embedding}. The commutativity of the diagram is
obvious.
\end{proof} 

Let $\alg P = \alg U_{\alg P}\alg M$ be a parabolic subgroup of $\alg G$. Let
$\pr_{\alg M}\colon \alg P\to \alg M$ be the canonical projection. Note that
$K_M = K\cap M$ is a special parahoric subgroup of $M$ (see
\S\ref{subsec:Iwahori-Weyl}). The map
\begin{align*}
\Theta^P_M\colon \Hecke_R(K_P,P) &\longrightarrow \Hecke_R(K_M,M),\\
\sum_i r_i\cdot (K_Pg_i) &\longmapsto \sum_i r_i\cdot \bigl(K_M \pr_{\alg
M}(g_i)\bigr)
\end{align*}
is a homomorphism of $R$-algebras.

\begin{defn} 
The composition
\[
\Satake^G_M\colon \Hecke_R(K,G) \xrightarrow{\varepsilon_{P,G}} \Hecke_R(K_P,P)
\xrightarrow{\Theta^P_M} \Hecke_R(K_M,M)
\]
is called the \emph{(partial) Satake homomorphism}.
\end{defn} 

If $\alg P = \alg B$ and $\alg M = \alg Z$, then the subgroup $K_Z$ is normal in
$Z$ and hence $\Hecke_R(K_Z,Z)$ identifies with the commutative group algebra
$R[K_Z\backslash Z] = R[\Lambda]$. In this case, the Satake homomorphism takes
the form
\[
\Satake^G \coloneqq \Satake^G_Z\colon \Hecke_R(K,G)\longrightarrow R[\Lambda].
\]

\begin{lem}\label{lem:partial Satake} 
Let $\alg Q = \alg U_{\alg Q}\alg L$ and $\alg P = \alg U_{\alg P}\alg M$ be
parabolic subgroups of $\alg G$ and assume that $\alg Q\subseteq \alg P$. The
diagram
\[
\begin{tikzcd}
\Hecke_R(K_P,P)\ar[r,"\Theta^P_M"] \ar[d,hook,"\varepsilon_{Q,P}"'] &
\Hecke_R(K_M,M) \ar[d,"\Satake^M_L"]\\
\Hecke_R(K_Q,Q) \ar[r,"\Theta^Q_L"'] & \Hecke_R(K_L,L)
\end{tikzcd}
\]
is commutative. In particular, one has $\Satake^G_L =
\Satake^M_L\circ\Satake^G_M$.
\end{lem} 
\begin{proof} 
Note that $\alg Q\cap \alg M$ is a parabolic subgroup of $\alg M$ with
Levi $\alg L$. Given
$b\in Q$, we have $\pr_{\alg M}(b) \in Q\cap M$ and $\pr_{\alg L}(\pr_{\alg
M}(b)) = \pr_{\alg L}(b)$. Hence, for all $\sum_i r_i\cdot (K_Pb_i)\in
\Hecke_R(K_P,P)$, where, by the Iwasawa decomposition~\ref{iwasawa}, we may choose
$b_i\in Q$, we compute
\begin{align*}
\Satake^M_L\bigl( \Theta^P_M\Bigl(\sum_i r_i\cdot (K_Pb_i)\Bigr)\bigr) &=
\Satake^M_L\Bigl(\sum_i r_i\cdot \bigl(K_M\pr_{\alg M}(b_i)\bigr)\Bigr)
= \sum_i r_i\cdot \bigl(K_L \pr_{\alg L}\bigl(\pr_{\alg M}(b_i)\bigr)\bigr)\\
&= \sum_i r_i\cdot \bigl(K_L \pr_{\alg L}(b_i)\bigr)
= \Theta^Q_L\Bigl(\sum_i r_i\cdot (K_Qb_i)\Bigr)\\ 
&= \Theta^Q_L\bigl(\varepsilon_{Q,P}\Bigl(\sum_i r_i\cdot (K_Pb_i)\Bigr)\bigr).
\end{align*}
In particular, in view of Lemma~\ref{lem:parabolic-embedding}, we have
\[
\Satake^M_L\circ\Satake^G_M = \Satake^M_L\circ\Theta^P_M\circ\varepsilon_{P,G}
= \Theta^Q_L\circ\varepsilon_{Q,P}\circ\varepsilon_{P,G} = \Theta^Q_L\circ
\varepsilon_{Q,G} = \Satake^G_L. \qedhere
\]
\end{proof} 
\subsection{The twisted action}\label{subsec:twisted} 
Assume that $R$ is a $\Z[1/p]$-algebra. The twisted action of $W_0$ on
$R[\Lambda]$ was defined by Henniart--Vign\'eras,
\cite[7.11, 7.12]{Henniart-Vigneras.2015}, in order to describe the image of
the integral Satake homomorphism. We give a slightly different presentation.

Given $b\in B$, we consider the integers (see \eqref{eq:mu_UP} in
\S\ref{subsec:positive})
\[
\mu_U(b) \coloneqq [K_U : K_U\cap b^{-1}K_Ub].
\]
Observe that $\mu_U$ is constant on $K_Z$-cosets, since $K_Z$ normalizes $K_U$.
Therefore, we obtain an induced map
\[
\mu_U\colon \Lambda \longrightarrow q^{\Z_{\ge0}}.
\]
Note that $\mu_U(\lambda) = 1$ if and only if $\lambda\in \Lambda^+$.

We employ the exponential notation $e^\lambda$ when we view $\lambda\in \Lambda$
as an element of $R[\Lambda]$.

\begin{defn} 
The \emph{twisted action} of $W_0$ on $R[\Lambda]$ is defined by
\[
w\star e^\lambda \coloneqq \frac{\mu_U(w(\lambda))}{\mu_U(\lambda)}\cdot
e^{w(\lambda)}, \qquad \text{for $\lambda\in \Lambda$, $w\in W_0$.}
\]
\end{defn} 

In order to describe the relation with the twisted action in
\cite[7.11]{Henniart-Vigneras.2015}, we recall the modulus character 
\[
\delta\colon B \longrightarrow q^\Z,\qquad \delta(b) \coloneqq [bK_Ub^{-1} : K_U]
= \mu_U(b)/\mu_U(b^{-1}),
\]
where $[bK_Ub^{-1} : K_U] \coloneqq \frac{[bK_Ub^{-1} : bK_Ub^{-1}\cap K_U]}
{[K_U : bK_Ub^{-1}\cap K_U]}$ denotes the generalized index.
Similar to the above, $\delta$ induces a character 
\[
\delta\colon \Lambda \longrightarrow q^\Z.
\]

\begin{lem}\label{lem:delta-mu} 
For all $w\in W_0$ and $\lambda\in \Lambda$, one has
\[
\frac{\delta(w(\lambda))}{\delta(\lambda)} =
\left(\frac{\mu_U(w(\lambda))}{\mu_U(\lambda)}\right)^2 =
\left(\frac{\mu_U(-\lambda)}{\mu_U(-w(\lambda))}\right)^2.
\]
\end{lem} 
\begin{proof} 
Note that $\frac{\delta(w(\lambda))}{\delta(\lambda)} =
\frac{\mu_U(w(\lambda))\cdot \mu_U(-\lambda)}{\mu_U(\lambda)\cdot
\mu_U(-w(\lambda))}$. Therefore, for both equalities it suffices to show
\[
\mu_U(\lambda)\cdot \mu_U(-\lambda) = \mu_U\bigl(w(\lambda)\bigr) \cdot
\mu_U\bigl(-w(\lambda)\bigr).
\]
But this follows from $\mu_U(\lambda)\mu_U(-\lambda) = q_\lambda$ (cf.
\cite[Proposition~3.14.(a)]{Heyer.2020}) and $q_\lambda = q_{w(\lambda)}$ (cf.
\cite[Proposition~5.13]{Vigneras.2016}).
\end{proof} 
\subsection{The Satake isomorphism}\label{subsec:Satake-iso} 
Given $\lambda\in \Lambda$, we denote $W_{0,\lambda}$ the stabilizer of
$\lambda$ under the (usual) $W_0$-action on $\Lambda$. Then $W_{0,\lambda}$ is
also the stabilizer of $e^\lambda$ under the twisted action of $W_0$ on
$R[\Lambda]$. 

Note that, if $R = \Z[1/p]$ and $\lambda\in \Lambda^+$, one has
\begin{equation}\label{eq:Slambda}
S_\lambda\coloneqq \sum_{w\in W_0/W_{0,\lambda}} w\star e^{\lambda} \in
\Z[\Lambda].
\end{equation}

With our notations, the main result of \cite{Henniart-Vigneras.2015} is the
following:

\begin{thm}\label{thm:Satake} 
Let $R$ be a commutative ring with $1$ and consider the Satake homomorphism
$\Satake^G\colon \Hecke_R(K,G)\to R[\Lambda]$.
\begin{enumerate}[label=(\roman*)]
\item\label{thm:Satake-i} $\Satake^G$ is injective.
\item\label{thm:Satake-ii} The image of $\Satake^G$ is a free $R$-module with
basis $\set{1\otimes S_\lambda}{\lambda\in \Lambda^+}$.

If $p\in R^\times$, then the image coincides with $R[\Lambda]^{W_0,\star}$, the
algebra of $W_0$-invariants under the twisted action.

\item\label{thm:Satake-iii} Both $R[\Lambda]$ and $\Hecke_R(K,G)$ are
commutative algebras of finite type over $R$.
\end{enumerate}
\end{thm} 
\begin{proof} 
We briefly explain how our notations relate to the notations in
\cite{Henniart-Vigneras.2015}. Consider the space $C^\infty_c(K\backslash G/K,
R)$ of compactly supported $K$-biinvariant functions $G\to R$ with product given
by convolution:
\[
(f_1*f_2)(g) = \sum_{h\in G/K}f_1(h)\cdot f_2(h^{-1}g),\qquad \text{for
$f_1,f_2\in C^\infty_c(K\backslash G/K,R)$ and $g\in G$,}
\]
where ``$h\in G/K$'' means that $h$ runs through a set of representatives for
the left cosets in $G/K$.
The map
\begin{align*}
\rho_G\colon C^\infty_c(K\backslash G/K,R) &\longrightarrow \Hecke_R(K,G),\\
f &\longmapsto \sum_{g\in K\backslash G} f(g^{-1})\cdot (Kg)
\end{align*}
is an anti-isomorphism of $R$-algebras.\footnote{As $C^\infty_c(K\backslash
G/K,R)$ turns out to be commutative, $\rho_G$ is in fact a homomorphism.}
Following \cite{Herzig.2011a}, the Satake homomorphism
in~\cite{Henniart-Vigneras.2015} is defined as
\begin{align*}
\Satake'\colon C^\infty_c(K\backslash G/K,R) &\longrightarrow C^\infty_c(Z/K_Z,
R) \cong R[\Lambda],\\
f &\longmapsto \Bigl[z\mapsto \sum_{u\in U/K_U} f(zu)\Bigr],
\end{align*}
where $K_U = K\cap U$. Now, the diagram
\[
\begin{tikzcd}
C^\infty_c(K\backslash G/K,R) \ar[d,"\rho_G"] \ar[r,"\Satake'"] &
C^\infty_c(Z/K_Z,R) \ar[d,"\rho_Z"]\\
\Hecke_R(K,G) \ar[r,"\Satake^G"'] & \Hecke_R(K_Z,Z)
\end{tikzcd}
\]
commutes: fix a representing system $\Gamma\subseteq Z$ for the coset
space $K_Z\backslash Z$, so that $K_U\backslash U\times \Gamma \cong K\backslash
G$ via $(K_Uu,z)\mapsto Kuz$. Then for each $f\in C^\infty_c(K\backslash G/K,R)$
we compute
\begin{align*}
\rho_Z\bigl(\Satake'(f)\bigr) &= \sum_{z\in \Gamma} \Satake'(f)(z^{-1})\cdot
(K_Zz)
= \sum_{z\in \Gamma} \sum_{u\in U/K_U} f(z^{-1}u)\cdot (K_Zz)\\
&= \sum_{z\in \Gamma}\sum_{u\in K_U\backslash U} f\bigl((uz)^{-1}\bigr) \cdot
(K_Zz)
= \sum_{uz\in K\backslash G} f\bigl((uz)^{-1}\bigr)\cdot (K_Zz)\\
&= \Satake^G\Bigl(\sum_{uz\in K\backslash G} f\bigl((uz)^{-1}\bigr)\cdot (K
uz)\Bigr)
= \Satake^G\bigl(\rho_G(f)\bigr).
\end{align*}
If $p\in R^\times$, then the twisted action of $W_0$ on $C^\infty_c(Z/K_Z,R)$
is defined by
\[
w\circ e^\lambda \coloneqq \delta^{1/2}(\lambda - w(\lambda)\bigr) \cdot
e^{w(\lambda)},
\]
where $\delta^{1/2}$ is a square root of $\delta$. This is indeed defined over
$R$, since Lemma~\ref{lem:delta-mu} shows that
$\delta^{1/2}(\lambda - w(\lambda))$ actually lies in $q^{\Z}$. The same lemma
also shows that $\rho_Z(w\circ e^\lambda) = w\star e^{-\lambda}$.

Therefore, we have $\rho_Z\bigl(\sum_{w\in W_0/W_{0,\lambda}} w\circ
e^\lambda\bigr) = S_{-\lambda}$ in $R[\Lambda]$, for all $\lambda\in
\Lambda^-$.

Now, \ref{thm:Satake-i} and \ref{thm:Satake-ii} are \cite[7.15 Thm. and 
7.13 Cor.]{Henniart-Vigneras.2015}, and \ref{thm:Satake-iii} is
\cite[7.16]{Henniart-Vigneras.2015}.
\end{proof} 

\begin{rmk}\label{rmk:Satake} 
Let $z\in Z^-$ and put $\lambda = zK_Z \in \Lambda$. It follows from
Remark~\ref{rmk:Iwasawa}.\ref{iwasawa-a} and \ref{iwasawa-b} that 
\begin{equation}\label{eq:Satake-explicit}
\Satake^G\bigl((z)_K\bigr) = e^\lambda + \sum_{\substack{\mu\in
\Lambda\text{ s.t.}\\\nu(\mu)< \nu(\lambda)}} a_{\mu}\cdot e^\mu \in
\Z[1/p][\Lambda].
\end{equation}
By Theorem~\ref{thm:Satake}, $\Satake^G\bigl((z)_K\bigr)$ is invariant under the
twisted action of $W_0$. Therefore, $\frac{a_\mu}{\mu_U(\mu)} =
\frac{a_{w(\mu)}}{\mu_U(w(\mu))}$, for all $w\in W_0$ and all $\mu\in \Lambda$.
In particular, $a_\mu\neq 0$ if and only if $a_{w(\mu)}\neq 0$ for all $w\in
W_0$. Now, \eqref{eq:Satake-explicit} shows $w(\mu) \le \lambda$, for all $w\in
W_0$ and all $\mu\in \Lambda$ with $\mu\le \lambda$. This explains
Remark~\ref{rmk:Iwasawa}.\ref{iwasawa-c}.
\end{rmk} 
\subsection{Centralizers in parabolic Hecke algebras} 
\label{subsec:centralizer}
Let $\alg P = \alg U_{\alg P}\alg M$ be a parabolic subgroup of $\alg G$. We
choose a strictly $M$-positive element $a_P \in Z$, see \S\ref{subsec:positive}.
This means that $a_P$ lies in the center of $M$ and satisfies
\[
\langle \alpha, \nu(a_P)\rangle <0,\qquad \text{for all $\alpha\in
\Sigma^+\setminus \Sigma_M$.}
\]
Note that $K_Pa_PK_P = K_Pa_P$ and hence $(a_P)_{K_P} = (K_Pa_P)$ in
$\Hecke_R(K_P,P)$. We consider the centralizer algebra
\[
C^+_P \coloneqq \set{X\in \Hecke_R(K_P,P)}{X\cdot (a_P)_{K_P} = (a_P)_{K_P}\cdot
X}.
\]
The algebra $C^+_P$ was originally studied by Andrianov when $\alg P$ is the
``Siegel parabolic'' of a symplectic group, see
\cite{Andrianov.1977, Andrianov.1979}.

\begin{lem}\label{lem:C^+_P} 
The following statements hold true:
\begin{enumerate}[label=(\roman*)]
\item\label{lem:C^+_P-i} $C^+_P = \set{X\in \Hecke_R(K_P,P)}{\text{$X = \sum_i
r_i\cdot (K_Pm_i)$ with $m_i\in M$ and $r_i\in R$}}$.

\item\label{lem:C^+_P-ii} For all $X\in \Hecke_R(K_P,P)$, there exists $n>0$ such
that $(a_P)_{K_P}^nX \in C^+_P$.

\item\label{lem:C^+_P-iii} The map $\Theta^P_M$ induces by restriction an
isomorphism $C^+_P\cong \Hecke_R(K_M,M^+)$. In particular, $C^+_P$ is
commutative.
\end{enumerate}
\end{lem} 
\begin{proof} 
See \cite[Lemma~4 and Corollary~5]{Heyer.2021}. The last assertion in
\ref{lem:C^+_P-iii} follows from the fact that $\Hecke_R(K_M,M)$ is commutative
by Theorem~\ref{thm:Satake} applied to $M$.
\end{proof} 

Note that Lemma~\ref{lem:C^+_P}.\ref{lem:C^+_P-i} shows that $C^+_P$ is
independent of the choice of $a_P$.

Recall the anti-involution $\zeta_P$ on $\Hecke_R(K_P,P)$,
cf.~\eqref{eq:Hecke-involution}, which is given by $\zeta_P\bigl((g)_{K_P}\bigr)
= (g^{-1})_{K_P}$.

\begin{rmk*}\label{rmk*:C^-_P} 
Let $a_P\in Z$ be a strictly $M$-positive element. Then $a_P^{-1}$ is strictly
$M$-negative and
\[
\zeta_P(C^+_P) = C^-_P \coloneqq \set{X\in \Hecke_R(K_P,P)}{X\cdot
(a_P^{-1})_{K_P} = (a_P^{-1})_{K_P}\cdot X}.
\]
The analog of Lemma~\ref{lem:C^+_P} for $C^-_P$ also holds.\footnote{However,
note that it is not $\Theta^P_M$ which induces an isomorphism $C^-_P \cong
\Hecke_R(K_M,M^-)$, but rather $\zeta_M\circ\Theta^P_M\circ\zeta_P$.}
\end{rmk*} 

\begin{lem}\label{lem:zeta-Ker} 
One has
\[
\zeta_P\bigl(\Ker \Theta^P_M\bigr) = \Ker \Theta^P_M.
\]
\end{lem} 
\begin{proof} 
Note that, for each $g\in P$, the index $\mu(g)\coloneqq [K_P : K_P\cap
g^{-1}K_Pg]$ counts the number of right
$K_P$-cosets in $K_PgK_P$, by \eqref{eq:mu}. Similarly, $\mu_M(m) \coloneqq [K_M
: K_M\cap m^{-1}K_Mm]$ counts the number of right $K_M$-cosets in $K_MmK_M$.
Moreover, 
\[
\delta\colon P\to q^\Z,\qquad \delta(g)\coloneqq [gK_Pg^{-1} : K_P] =
\frac{\mu(g)}{\mu(g^{-1})}
\]
is the modulus character of $P$. 
As every element in $U_P$ is contained in a compact group, we have
$\delta\big|_{U_P} = 1$, by~\cite[Ch.\,I, 2.7]{Vigneras.1996}. Therefore,
\begin{equation}\label{eq:delta} 
\delta(g) = \delta\bigl(\pr_{\alg M}(g)\bigr),\qquad \text{for all $g\in P$.}
\end{equation}

Note that $\Theta^P_M\bigl((g)_{K_P}\bigr) = \frac{\mu(g)}
{\mu_M(\pr_{\alg M}(g))}\cdot \bigl(\pr_{\alg M}(g)\bigr)_{K_M}$ by
\cite[Proposition~4.3]{Heyer.2020}, and hence
$\Ker \Theta^P_M$ is generated by elements of the form $(g)_{K_P}
-\frac{\mu(g)}{\mu(\pr_{\alg M}(g))}\cdot (\pr_{\alg M}(g))_{K_P}$. We compute
\begin{align*}
\Theta^P_M\bigl(\zeta_P\Bigl((g)_{K_P} &- \frac{\mu(g)}{\mu(\pr_{\alg
M}(g))}\cdot \bigl(\pr_{\alg M}(g)\bigr)_{K_P}\Bigr)\bigr)\\
&= \frac{\mu(g^{-1})}{\mu_M(\pr_{\alg M}(g^{-1}))}\cdot \bigl(\pr_{\alg
M}(g^{-1})\bigr)_{K_M} - \frac{\mu(g)\cdot \mu(\pr_{\alg
M}(g^{-1}))}{\mu(\pr_{\alg M}(g))\cdot \mu_M(\pr_{\alg M}(g^{-1}))}\cdot
\bigl(\pr_{\alg M}(g^{-1})\bigr)_{K_M}\\
&= 0, \qquad \text{by \eqref{eq:delta}.}
\end{align*}
This shows $\zeta_P(\Ker\Theta^P_M) \subseteq \Ker\Theta^P_M$. The assertion
follows from $\zeta_P^2 = \id$.
\end{proof} 

\begin{rmk*} 
One has $\Theta^P_M\circ \zeta_P \neq \zeta_M\circ \Theta^P_M$ (apply both maps
to $(a_P)_{K_P}$ for a strictly $M$-positive $a_P$). 
\end{rmk*} 
\subsection{Example of a parabolic Hecke algebra}\label{subsec:example} 
The purpose of this section is to work out an example of the setup so far.

Let $\O_{\field}$ be the valuation ring of $\field$ and fix a uniformizer
$\pi\in \O_\field$. Consider the group $G = \GL_2(\field)$ and the maximal
compact subgroup $K = \GL_2(\O_{\field})$. Let $B \subseteq G$ be the subgroup
of upper triangular matrices and $Z \subseteq B$ the subgroup of diagonal
matrices. Fix a coefficient ring $R$.
A variant of the parabolic Hecke algebra
$\Hecke_R(K_B,B)$ is briefly discussed in Vienney's thesis
\cite[p.~102]{Vienney.2012}. We prove the following structure result of
$\Hecke_R(K_B,B)$ in Appendix~\ref{appendix:H(B)}.
\begin{thm}\label{thm:example} 
The $R$-algebra $\Hecke_R(K_B,B)$ is generated by the elements
\begin{align*}
X_+&\coloneqq (\begin{pmatrix}\pi&0\\0&1\end{pmatrix})_{K_B}, & X_- &\coloneqq
(\begin{pmatrix}\pi^{-1} & 0\\0 & 1\end{pmatrix})_{K_B}, & Y &\coloneqq (\pi
E_2)_{K_B}, & Y^{-1} &\coloneqq (\pi^{-1}E_2)_{K_B},
\end{align*}
where $E_2$ is the $2\times 2$ identity matrix, subject only to the following
relations:
\begin{align}\label{eq:ex-H(B)}
\begin{split}
YY^{-1} &= Y^{-1}Y = 1,\\
YX_+ &= X_+Y,\\
YX_- &= X_-Y,\\
X_+X_- &= q\cdot 1.
\end{split}
\end{align}
In particular, $\Hecke_R(K_B,B)$ is non-commutative.
\end{thm} 

\begin{rmk*} 
It follows from Theorem~\ref{thm:example} that $X_+$ is a left zero-divisor
(resp. $X_-$ is a right zero-divisor), because
\[
X_+\cdot (X_-X_+ - q\cdot 1) = (X_-X_+ - q\cdot 1)\cdot X_- = 0.
\]
If $q$ is invertible in $R$, then $X_+$ is right invertible (resp. $X_-$ is
left invertible), since
\[
X_+\cdot q^{-1}X_- = q^{-1}X_+\cdot X_- = 1.
\]
\end{rmk*} 

Since $\Lambda \cong \Z^2$, the Hecke algebra $\Hecke_R(K_Z,Z)$ identifies with
$R[x^{\pm1},y^{\pm}]$ via $(\begin{psmallmatrix}\pi & 0\\0 &
1\end{psmallmatrix})_{K_Z} \mapsto x$ and $(\begin{psmallmatrix}1 & 0\\0 &
\pi\end{psmallmatrix})_{K_Z} \mapsto y$. Then the map $\Theta^B_Z$ is given by
\begin{align*}
\Theta^B_Z\colon \Hecke_R(K_B,B) &\longrightarrow \Hecke_R(K_Z,Z),\\
X_+ &\longmapsto x,\\
X_- &\longmapsto qx^{-1},\\
Y &\longmapsto xy.
\end{align*}
The kernel of $\Theta^B_Z$ is the two-sided ideal generated by $X_-X_+ - q\cdot
1$.

Assume $q\in R^\times$. The twisted action of $W_0 = \{1,w_0=
\begin{psmallmatrix}0 & 1\\1 &
0\end{psmallmatrix}\}$ on $R[x^{\pm 1},y^{\pm 1}]$ is given by 
\[
w_0\star x = qy\qquad \text{and}\qquad w_0\star y = q^{-1}x.
\]
For all $a,b,c\in \Z$ with $a>b$, the elements
\begin{alignat}{2}
S_{a,b} &\coloneqq x^ay^b + q^{a-b}x^by^a, &\hspace{3em}
S_{c,c} &\coloneqq (xy)^c
\end{alignat}
are $W_0$-invariant with respect to the twisted action; in fact they constitute
an $R$-basis of $R[\Lambda]^{W_0,\star}$. Moreover, the relations
\begin{alignat}{4}
S_{c,c} &= S_{1,1}^c, && \text{for $c\in \Z$,}\\
S_{a,b} &= S_{b,b}\cdot S_{a-b,0} = S_{a-b,0}\cdot S_{b,b}, &\hspace{2em}&
\text{for $a,b\in \Z$ with $a>b$,}\\
S_{n,0}\cdot S_{1,0} &= S_{n+1,0} + q\cdot S_{1,1}\cdot S_{n-1,0}, && \text{for
$n\in\Z_{\ge1}$}
\end{alignat}
are immediate. This easily implies $R[\Lambda]^{W_0,\star} = R[S_{1,0},
S_{1,1}^{\pm 1}] = R[x+qy, (xy)^{\pm1}]$.

We identify $\Lambda^+$ with $\set{(a,b)\in \Z^2}{a\ge b}$. By the Cartan
decomposition~\ref{cartan}, every double coset in $G$ is uniquely of the form
$K\begin{pmatrix}\pi^a&0\\0 & \pi^b\end{pmatrix}K$ with $(a,b)\in \Lambda^+$.
Let us compute the Satake homomorphism $\Satake^G\colon \Hecke_R(K,G) \to
R[\Lambda]$. Let $A\subseteq \O_{\field}$ be a complete system of
representatives for the residue class field of $\field$ such that $0\in A$. Note
that $\lvert A\rvert = q$. For each $(a,b)\in \Lambda^+$ we have
\[
\begin{split}
K\triang{\pi^a}{0}{\pi^b}K = K\triang{\pi^a}{0}{\pi^b} &\sqcup
\bigsqcup_{c=1}^{a-b-1} \bigsqcup_{\substack{\beta_b,\dotsc,\beta_{b+c-1}\in A\\
\beta_b\neq 0}} K\triang{\pi^{a-c}}{\sum_{i=b}^{b+c-1}\beta_i\pi^i}
{\pi^{b+c}}\\
&\sqcup \bigsqcup_{\beta_b,\dotsc,\beta_{a-1}\in A} K\triang{\pi^b}
{\sum_{i=b}^{a-1}\beta_i\pi^i}{\pi^a}.
\end{split}
\]
Let $\gamma\coloneqq \lfloor\frac{a-b}{2}\rfloor$ be the largest
integer $\le \frac{a-b}{2}$. We deduce
\begin{align*}
\Satake^G\bigl((\triang{\pi^a}{0}{\pi^b})_K\bigr) &= x^ay^b + \sum_{c=1}^{a-b-1}
(q-1)q^{c-1} x^{a-c}y^{b+c} + q^{a-b} x^by^a\\
&= S_{a,b} + (q-1)\cdot \sum_{c=1}^{\gamma} q^{c-1}
S_{a-c,b+c} +\epsilon\cdot (q-1)q^{\gamma-1}S_{\gamma,\gamma},
\end{align*}
and where $\epsilon=1$ if $a-b$ is even and non-zero, and $\epsilon=0$
otherwise. Consider on $\Lambda$ the partial ordering defined by $(c,d) \le
(a,b)$ if $c\le a$ and $c+d = a+b$. As usual we write $(c,d) < (a,b)$ if
$(c,d)\le (a,b)$ and $(c,d)\neq (a,b)$. Note that for each $(a,b)\in \Lambda^+$
there are only finitely many elements $(c,d)$ in $\Lambda^+$ satisfying $(c,d)
< (a,b)$. Then we have shown
\[
\Satake^G\bigl((\triang{\pi^a}{0}{\pi^b})_K\bigr) \in  S_{a,b} +
\sum_{\substack{(c,d)\in \Lambda^+,\\ (c,d) < (a,b)}} \Z.S_{c,d}.
\]
By a ``triangular argument'' it follows that $\Satake^G$ is injective with image
$R[\Lambda]^{W_0,\star}$. In particular, $\Hecke_R(K,G)$ is commutative.
Moreover,
\begin{align*}
\Satake^G\bigl((\triang{\pi}{0}{1})_K\bigr) &= S_{1,0} = x + qy,\\
\Satake^G\bigl((\pi E_2)_K\bigr) &= S_{1,1} = xy,
\end{align*}
which shows that $\Hecke_R(K,G)$ identifies with the polynomial ring generated
by $(\triang{\pi}{0}{1})_K$ and $(\pi E_2)_K$, with $(\pi E_2)_K$ invertible. We
have verified Theorem~\ref{thm:Satake} in this specific example.

We can also view $\Hecke_R(K,G)$ as a subalgebra of $\Hecke_R(K_B,B)$ via the
embedding
\begin{align*}
\varepsilon_{B,G}\colon \Hecke_R(K,G) &\longhookrightarrow \Hecke_R(K_B,B),\\
(\begin{pmatrix}\pi &0\\0&1\end{pmatrix})_K &\longmapsto X_+ + X_-Y,\\
(\pi E_2)_K &\longmapsto Y.
\end{align*}

Note that $C^+_B$ is the centralizer of $X_+$. Explicitly, $C^+_B$ is the
polynomial algebra
\[
C^+_B = R[X_+,Y^{\pm1}] \subseteq  \Hecke_R(K_B,B).
\]
The anti-involution $\zeta_B$ on $\Hecke_R(K_B,B)$ is determined by
$\zeta_B(X_+) = X_-$ and $\zeta_B(Y) = Y^{-1}$.

Consider the polynomial
\[
Q(t) = 1 - (\begin{pmatrix}\pi&0\\0&1\end{pmatrix})_K\cdot t + q\cdot
(\pi\cdot E_2)_K\cdot t^2 \in \Hecke_R(K,G)[t].
\]
Applying $\Satake^G$ to the coefficients of $Q$, the resulting polynomial
$Q^{\Satake^G}(t)$ decomposes as follows:
\[
Q^{\Satake^G}(t) = 1 - (x+qy)\cdot t + qxy\cdot t^2 = (1 - xt)\cdot (1-qyt) \in
R[x^{\pm1},y^{\pm1}][t].
\]
One may ask whether this decomposition can be lifted to a decomposition of
$Q(t)$ in $\Hecke_R(K,G)[t]$. Unfortunately, this is false. But it turns out
that one can find a decomposition of $Q(t)$ over the parabolic Hecke algebra
$\Hecke_R(K_B,B)$: applying $\varepsilon_{B,G}$ to the coefficients of $Q(t)$,
we obtain
\[
Q^{\varepsilon_{B,G}}(t) = 1 - (X_++X_-Y)\cdot t + qYt^2 = (1-X_+t)\cdot
(1-X_-Yt) \in \Hecke_R(K_B,B)[t].
\]
Here, the free variable $t$ in $\Hecke_R(K_B,B)[t]$ commutes with the elements
in $\Hecke_R(K_B,B)$.
Note that the order of the factors is important, because $\Hecke_R(K_B,B)$ is
non-com\-mu\-ta\-tive.

In the next subsection we prove a general decomposition theorem following the
ideas of Andrianov~\cite{Andrianov.1977,Andrianov.1995}.
\subsection{The decomposition theorem}\label{subsec:decomp} 
From now on we assume that $R$ is a $\Z[1/p]$-algebra. Let $\alg P = \alg
U_{\alg P}\alg M$ be a parabolic subgroup of $\alg G$. We view the embeddings
$\varepsilon_{P,G}\colon \Hecke_R(K,G)\hookrightarrow \Hecke_R(K_P,P)$ and
$\varepsilon_{B,P}\colon \Hecke_R(K_P,P)\hookrightarrow \Hecke_R(K_B,B)$ as
inclusions.

Let $a_P\in Z$ be a strictly $M$-positive element and denote its image in
$\Lambda$ by $\lambda_P$. Choose a representing system $W^M_0$ of
$W_0/W_{0,M}$ in $W_0$. Consider the polynomial
\[
\wt \chi_{a_P}(t)\coloneqq \prod_{w\in W^M_0}\bigl(1 -w\star
e^{-\lambda_P}\cdot t\bigr) \in 1 + tR[\Lambda][t].
\]
This definition does not depend on the choice of $W^M_0$, because $W_{0,M}$
fixes $e^{-\lambda_P} \in R[\Lambda]$ with respect to the twisted action. Note
that, by construction, we have $\wt \chi_{a_P}(e^{\lambda_P}) = 0$, and the
coefficients of $\wt \chi_{a_P}(t)$ are $W_0$-invariant for the twisted action.
By Theorem~\ref{thm:Satake}.\ref{thm:Satake-ii} the coefficients of $\wt
\chi_{a_P}(t)$ lie in the image of the Satake map $\Satake^G$. Since, by
Theorem~\ref{thm:Satake}.\ref{thm:Satake-i}, $\Satake^G$ is injective, there
exists a unique polynomial
\begin{equation}\label{eq:chi}
\chi_{a_P}(t) = \sum_{i=0}^{\lvert W^M_0\rvert} X_i\cdot t^i \in 1 +
t\Hecke_R(K,G)[t]
\end{equation}
with $\sum_{i} \Satake^G(X_i)\cdot t^i = \wt\chi_{a_P}(t)$. (Note that $X_0 =
1$.) Explicitly, we have, for all $0\le i\le \lvert W^M_0\rvert$,
\begin{equation}\label{eq:S(Xi)}
\Satake^G(X_i) = (-1)^i \mu_{U}(-\lambda_P)^{-i} \cdot \sum_{\substack{J
\subseteq W^M_0\\ \lvert J\rvert = i}} \prod_{w\in J}
\mu_U\bigl(w(-\lambda_P)\bigr) \cdot e^{\sum_{w\in J}w(-\lambda_P)} \in
R[\Lambda].
\end{equation}

\begin{lem}\label{lem:chi-ker} 
One has
\[
\chi_{a_P}\bigl((a_P)_{K_P}\bigr) \coloneqq \sum_{i=0}^{\lvert W^M_0\rvert}
(a_P)^i_{K_P}\cdot X_i \in \Ker \Theta^P_M.
\]
\end{lem} 
\begin{proof} 
By Lemma~\ref{lem:partial Satake} we have
$\Satake^M\bigl(\Theta^P_M\bigl((a_P)_{K_P}\bigr)\bigr) =
\Theta^B_Z\bigl((a_P)_{K_B}\bigr) = e^{\lambda_P}$, and the restriction
of $\Satake^M\circ \Theta^P_M$ to $\Hecke_R(K,G)$ coincides with $\Satake^G$. We
compute
\[
(\Satake^M\circ \Theta^P_M)\bigl(\chi_{a_P}\bigl((a_P)_{K_P}\bigr)\bigr) =
\sum_i e^{i\lambda_P}\cdot \Satake^G(X_i) = \wt\chi_{a_P}(e^{\lambda_P}) = 0.
\]
Since by Theorem~\ref{thm:Satake}.\ref{thm:Satake-i} the map $\Satake^M$ is
injective, the assertion follows.
\end{proof} 

In order for the theory to work, one needs to assume the following strengthening
of Lemma~\ref{lem:chi-ker}:

\begin{hyp}\label{hyp} 
The element $(a_P)_{K_P}$ is a left root of $\chi_{a_P}(t)$, meaning that
\[
\chi_{a_P}\bigl((a_P)_{K_P}\bigr) = 0\qquad \text{in $\Hecke_R(K_P,P)$.}
\]
\end{hyp} 

This hypothesis has been verified in many cases: Andrianov \cite{Andrianov.1977}
essentially proved it for $G = \Sp_{2n}(\field)$ with $P$ being the ``Siegel
parabolic'', \ie the subgroup of matrices whose lower left quadrant is zero.
Gritsenko then adopted the methods of Andrianov to prove it for $G =
\GL_n(\field)$ and all parabolics \cite{Gritsenko.1988,Gritsenko.1992}. Finally,
Gritsenko verified Hypothesis~\ref{hyp} for the classical groups
$\Sp_{2n}(\field)$, $\SU_n(\field)$, and $\SO_n(\field)$, for the parabolics 
fixing a line in the standard representation, see \cite{Gritsenko.1990}.

The main contribution of this article is the verification of
Hypothesis~\ref{hyp} for general connected reductive groups and
\emph{non-obtuse} parabolics, cf.~\S\ref{sec:non-obtuse}. This covers, 
in particular, all the cases mentioned above.

\begin{thm}\label{thm:hyp} 
Assume that $\alg P$ is non-obtuse. Then Hypothesis~\ref{hyp} holds true.
\end{thm} 
\begin{proof} 
Lemma~\ref{lem:C^+_P}.\ref{lem:C^+_P-iii} shows $C^+_P\cap \Ker\Theta^P_M =
\{0\}$. Hence, in view of Lemma~\ref{lem:chi-ker}, it suffices to prove
$(a_P)_{K_P}^i X_i \in C^+_P$, for all $0\le i\le \lvert W^M_0\rvert$.

Let $i$ be arbitrary but fixed.
Recall the explicit description~\eqref{eq:S(Xi)}. Note that $\sum_{w\in
J}\nu\bigl(w(-\lambda_P)\bigr) \le \nu(-i\lambda_P)$, where $J\subseteq W^M_0$
is such that $\lvert J\rvert = i$, see \cite[Ch.\,VI, \S1, no.\,6,
Prop.\,18]{Bourbaki.1981}. If we write
\[
X_i = \sum_{j=1}^{n} c_{j}\cdot (z_{j})_K\qquad \text{in $\Hecke_R(K,G)$,}
\]
then this and~\eqref{eq:Satake-explicit} imply $\nu(z_{j}) \le \nu(a_P^{-i})$,
for all $1\le j\le n$. In view
of the Cartan decomposition~\ref{cartan} we may choose $z_j\in Z^-$.

By the Iwasawa decomposition~\ref{iwasawa} for $G$ and the Cartan
decomposition~\ref{cartan} for $M$, we have
\[
G = KP = KU_PM = KU_PK_MZ^{+,M}K_M = KU_PZ^{+,M}K_M,
\]
because $K_M$ normalizes $U_P$. Therefore, we can write, for every $j$,
\[
(z_{j})_K = \sum_{l=1}^{n_{j}} c_{jl}\cdot (Ku_{jl}z'_{jl}k_{jl}),
\]
with $u_{jl}\in U_P$, $z'_{jl}\in Z^{+,M}$, and $k_{jl}\in K_M$. Note that
$u_{jl}z'_{jl} \in Kz_jK$, for all $j$, $l$.

Since $\alg P$ is non-obtuse, Theorem~\ref{thm:main} implies $a_P^i
u_{jl}a_P^{-i} \in K_P$, for all $j$, $l$. Now, since $a_P^iz'_{jl}k_{jl} \in
M$, Lemma~\ref{lem:C^+_P}.\ref{lem:C^+_P-i} implies
\[
(a_P)_{K_P}^i\cdot X_i = \sum_{j=1}^n\sum_{l=1}^{n_j}c_jc_{jl}\cdot (K_P
a_P^iz'_{jl}k_{jl}) \in C^+_P. \qedhere
\]
\end{proof} 

Consider the submodules
\begin{align*}
\OO^+_P &\coloneqq C^+_P. \Hecke_R(K,G)\coloneqq \set{\sum_i Y_iZ_i \in
\Hecke_R(K_P,P)}{Y_i\in C^+_P,\, Z_i\in \Hecke_R(K,G)},\\
\OO^-_P &\coloneqq \zeta_P(\OO^+_P) = \Hecke_R(K,G).C^-_P
\end{align*}
of $\Hecke_R(K_P,P)$. (Note that $\zeta_P$ preserves $\Hecke_R(K,G)$ by
Lemma~\ref{lem:emb-anti}.)

\begin{defn} 
Assume that Hypothesis~\ref{hyp} is satisfied. For every $n\in\Z_{\ge1}$ we
define recursively the ``negative powers'' of $(a_P)_{K_P}$ as
\[
(a_P)_{K_P}^{-n}\coloneqq - \sum_{i=1}^{\lvert W^M_0\rvert}
(a_P)_{K_P}^{i-n}\cdot X_i \in \OO^+_P.
\]
It should be noted that $(a_P)_{K_P}$ is not invertible in $\Hecke_R(K_P,P)$. In
fact, for $n>1$, one even has $\bigl((a_P)_{K_P}^{-1}\bigr)^n \neq
(a_P)_{K_P}^{-n}$. However, by a simple induction on $d$ we have
\[
(a_P)_{K_P}^n\cdot (a_P)_{K_P}^d = (a_P)_{K_P}^{n+d},\qquad \text{for all $n\in
\Z_{\ge0}$ and $d\in \Z$.}
\]
\end{defn} 

\begin{ex}\label{ex:neg-pow} 
Let us compute ``negative powers'' for $G = \GL_2(\field)$. The notations are
the same as in \S\ref{subsec:example}. Choose the strictly positive element $a_B
= \begin{psmallmatrix}\pi & 0\\0 & 1\end{psmallmatrix}$, so that $(a_B)_{K_B} =
X_+$ in $\Hecke_R(K_B,B)$. The polynomial 
\[
\wt\chi_{a_B}(t) = (1-x^{-1}t)\cdot \bigl(1-(qy)^{-1}t\bigr) = 1 -\bigl(x^{-1} +
(qy)^{-1}\bigr)\cdot t + (qxy)^{-1}\cdot t^2 \in R[x^{\pm1},y^{\pm1}][t]
\]
annihilates $\Theta^B_Z\bigl((a_B)_{K_B}\bigr) = x$. We compute
\begin{align*}
\Satake^G\Bigl(q^{-1}(\pi E_2)_K^{-1}\cdot (\begin{pmatrix}\pi & 0\\0 &
1\end{pmatrix})_K\Bigr) &= q^{-1}\cdot (xy)^{-1}\cdot (x+qy) = x^{-1} +
(qy)^{-1}\\
\Satake^G\bigl(q^{-1}(\pi E_2)_K^{-1}\bigr) &= (qxy)^{-1},\\
\intertext{and}
\varepsilon_{B,G}\Bigl(q^{-1}\cdot (\pi E_2)_K^{-1}\cdot (\begin{pmatrix}\pi &
0\\0 &1 \end{pmatrix})_K\Bigr) &= q^{-1}Y^{-1}\cdot (X_++X_-Y) = q^{-1}\cdot
(X_+Y^{-1} + X_-),\\
\varepsilon_{B,G}\bigl(q^{-1}\cdot  (\pi E_2)_K^{-1}\bigr) &= q^{-1}Y^{-1}.
\end{align*}
Therefore,
\[
\chi_{a_B}(t) = 1 - q^{-1}\cdot (X_+Y^{-1} + X_-)\cdot t + q^{-1}Y^{-1}\cdot
t^2.
\]
It is clear that $X_+$ is a left root of $\chi_{a_B}(t)$. We compute
\begin{align*}
X_+^{-1} &= q^{-1}\cdot (X_+Y^{-1} + X_-) - q^{-1}X_+Y^{-1}\\ 
&= q^{-1}X_-,\\
X_+^{-2} &= q^{-1} X_+^{-1}\cdot (X_+Y^{-1}+X_-) - q^{-1}Y^{-1}\\
&= q^{-2}X_-^{2} + q^{-2}Y^{-1} (X_-X_+ - q\cdot 1)\\
X_+^{-3} &= q^{-1} X_+^{-2}(X_+Y^{-1}+X_-) - q^{-1}X_+^{-1}Y^{-1}\\
&= q^{-3}X_-^3 + q^{-3}\cdot (X_-X_+-q\cdot 1)X_+Y^{-2} +
q^{-3}X_-(X_- X_+-q\cdot 1)Y^{-1}.
\end{align*}
\end{ex} 

Recall from Lemma~\ref{lem:C^+_P}.\ref{lem:C^+_P-ii} that for every $X\in
\OO^+_P$ there exists $n>0$ such that $(a_P)_{K_P}^nX \in C^+_P$. Using
``negative powers'' it is possible to reconstruct $X$ from $(a_P)_{K_P}^nX$.

\begin{lem}\label{lem:negpow} 
Assume that Hypothesis~\ref{hyp} is satisfied (for $\alg P$).
\begin{enumerate}[label=(\roman*)]
\item\label{lem:negpow-i} For every $X\in\Hecke_R(K,G)$ and every $n\ge0$ with
$(a_P)_{K_P}^n X \in C^+_P$ we have
\[
(a_P)_{K_P}^n \cdot X\cdot (a_P)_{K_P}^d = (a_P)_{K_P}^{n+d}\cdot X,\qquad
\text{for all $d\in\Z$.}
\]
\item\label{lem:negpow-ii} Let $\alg Q$ be a parabolic contained in $\alg P$.
For every $X\in \OO^+_Q$ we have, inside $\Hecke_R(K_Q,Q)$,
\[
(a_P)_{K_P}^n\cdot X\cdot (a_P)_{K_P}^{-n} = X,\qquad \text{for all $n\gg0$.}
\]
\end{enumerate}
\end{lem} 
\begin{proof} 
\begin{itemize}
\item[\ref{lem:negpow-i}] Let $X\in \Hecke_R(K,G)$ and let $n\ge 0$ such that
$(a_P)_{K_P}^nX \in C^+_P$. We prove the assertion by descending induction on
$d$. If $d\ge0$, this follows from the assumption that $(a_P)_{K_P}^nX$
centralizes $(a_P)_{K_P}$. Now assume $d<0$. Since $\Hecke_R(K,G)$ is
commutative by Theorem~\ref{thm:Satake}, we compute
\begin{align*}
(a_P)_{K_P}^n\cdot X\cdot (a_P)_{K_P}^d &= - \sum_{i=1}^{\lvert W^M_0\rvert}
(a_P)_{K_P}^n\cdot X\cdot (a_P)^{i+d}_{K_P}\cdot X_i\\
&= -\sum_{i=1}^{\lvert W_0^M\rvert} (a_P)^{n+i+d}_{K_P}\cdot XX_i 
&& \text{(induction hypothesis)}\\
&= \biggl(-\sum_{i=1}^{\lvert W^M_0\rvert} (a_P)_{K_P}^{n+i+d}X_i\biggr)\cdot X
&& \text{($\Hecke_R(K,G)$ is commutative)}\\
&= (a_P)_{K_P}^{n+d}\cdot X.
\end{align*}
This finishes the induction step.

\item[\ref{lem:negpow-ii}] Write $X = \sum_j Y_jZ_j$ with $Y_j\in C^+_Q$ and
$Z_j\in \Hecke_R(K,G)$. Choose $n\in \Z_{>0}$ such that $(a_P)_{K_P}^n Z_j \in
C^+_P$ for all $j$. From Lemma~\ref{lem:C^+_P}.\ref{lem:C^+_P-i} applied to $\alg
Q$, it follows easily that $(a_P)_{K_P}$ centralizes $C^+_Q$. Using
\ref{lem:negpow-i}, we compute
\[
(a_P)^n_{K_P}\cdot X\cdot (a_P)_{K_P}^{-n} = \sum_j Y_j\cdot (a_P)_{K_P}^n Z_j
(a_P)_{K_P}^{-n} = \sum_j Y_jZ_j = X.
\]
\end{itemize}
\end{proof} 

We are now able to prove the following characterization of
Hypothesis~\ref{hyp}:
\begin{prop}\label{prop:hyp-equiv} 
The following assertions are equivalent:
\begin{enumerate}[label=(\roman*)]
\item\label{prop:hyp-i} Hypothesis~\ref{hyp} is satisfied.
\item\label{prop:hyp-ii} If $X\in \OO^+_P$ and $n\ge0$ are such that
$(a_P)_{K_P}^nX = 0$, then $X = 0$.
\item\label{prop:hyp-iii} $\OO^+_P\cap \Ker\Theta^P_M = \{0\}$.
\item\label{prop:hyp-iv} If $X\in \OO^-_P$ and $n\ge0$ are such that
$X\cdot (a_P^{-1})_{K_P}^n = 0$, then $X = 0$.
\item\label{prop:hyp-v} $\OO^-_P\cap \Ker\Theta^P_M = \{0\}$.
\end{enumerate}
\end{prop} 
\begin{proof} 
The equivalences \ref{prop:hyp-ii}$\iff$\ref{prop:hyp-iv} and
\ref{prop:hyp-iii}$\iff$\ref{prop:hyp-v} follow from
Lemma~\ref{lem:zeta-Ker} and $\zeta_P(\OO^+_P) = \OO^-_P$.

Since $\chi_{a_P}\bigl((a_P)_{K_P}\bigr) \in \OO^+_P \cap \Ker\Theta^P_M$ by
Lemma~\ref{lem:chi-ker}, it follows that \ref{prop:hyp-iii} implies
\ref{prop:hyp-i}.

Assume that \ref{prop:hyp-i} holds. Let $X\in \OO^+_P$ and $n\ge0$ such that
$(a_P)_{K_P}^nX = 0$. After possibly enlarging $n$, it follows from
Lemma~\ref{lem:negpow}.\ref{lem:negpow-ii} that $X =
(a_P)_{K_P}^nX(a_P)_{K_P}^{-n} = 0$, whence \ref{prop:hyp-ii}.

Assume that \ref{prop:hyp-ii} holds. Let $X\in \OO^+_P$ such that $\Theta^P_M(X)
= 0$. By Lemma~\ref{lem:C^+_P}.\ref{lem:C^+_P-ii} there exists $n\ge0$ such that
$(a_P)_{K_P}^nX \in C^+_P$. But we also have $(a_P)_{K_P}^nX \in
\Ker\Theta^P_M$, hence Lemma~\ref{lem:C^+_P}.\ref{lem:C^+_P-iii} implies
$(a_P)_{K_P}^nX = 0$. By the assumption we deduce $X = 0$. This proves
\ref{prop:hyp-iii}.
\end{proof} 

\begin{cor}\label{cor:hyp} 
Assume that Hypothesis~\ref{hyp} is satisfied. Then
\[
\Hecke_R(K_P,P) = \OO^+_P\oplus \Ker\Theta^P_M = \OO^-_P\oplus \Ker\Theta^P_M.
\]
\end{cor} 
\begin{proof} 
In view of Lemma~\ref{lem:zeta-Ker} it suffices to prove the first equality. By
Proposition~\ref{prop:hyp-equiv} it remains to prove $\Hecke_R(K_P,P) \subseteq
\OO^+_P + \Ker\Theta^P_M$. So let $X\in \Hecke_R(K_P,P)$ and choose $n\in
\Z_{>0}$ such that $(a_P)_{K_P}^nX \in C^+_P$. Then
$(a_P)_{K_P}^nX(a_P)_{K_P}^{-n} \in \OO^+_P$ and $X -
(a_P)_{K_P}^nX(a_P)_{K_P}^{-n} \in \Ker\Theta^P_M$. 
\end{proof} 

The next result implies that it suffices to verify Hypothesis~\ref{hyp} for
\emph{maximal} parabolics only.

\begin{prop}\label{prop:hyp-cap} 
Let $\alg Q$ be another parabolic subgroup of $\alg G$ with Levi $\alg L$.
Assume that Hypothesis~\ref{hyp} is satisfied for $\alg P$ and $\alg Q$. Then
Hypothesis~\ref{hyp} is satisfied for $\alg P\cap \alg Q$.
\end{prop} 
\begin{proof} 
Let $a_Q \in Z$ be a strictly $L$-positive element. Then $a_{P\cap Q}\coloneqq
a_Pa_Q$ is strictly $L\cap M$-positive. The elements $(a_P)_{K_P}$ and
$(a_Q)_{K_Q}$ commute inside $\Hecke_R(K_{P\cap Q},P\cap Q)$ and we have
$(a_{P\cap Q})_{K_{P\cap Q}} = (a_P)_{K_P}(a_Q)_{K_Q}$ in $\Hecke_R(K_{P\cap
Q}, P\cap Q)$.
Let $X\in \OO^+_{P\cap Q}$ and let $n\ge0$ such that
$(a_{P\cap Q})_{K_{P\cap Q}}^n X = 0$. 

After enlarging $n$ if necessary, Lemma~\ref{lem:negpow}.\ref{lem:negpow-ii}
shows
\begin{align*}
0 &= (a_{P\cap Q})^n_{K_{P\cap Q}} X\cdot (a_Q)_{K_Q}^{-n}(a_P)_{K_P}^{-n}\\
&= (a_P)_{K_P}^n(a_Q)_{K_Q}^nX\cdot (a_Q)_{K_Q}^{-n}(a_P)_{K_P}^{-n}\\
&= (a_P)_{K_P}^nX\cdot (a_P)_{K_P}^{-n}\\
&= X.
\end{align*}
The assertion now follows from ``\ref{prop:hyp-ii}$\implies$\ref{prop:hyp-i}''
in Proposition~\ref{prop:hyp-equiv}.
\end{proof} 

\begin{notation*} 
If $\psi\colon A\to B$ is a homomorphism of $R$-algebras and $f(t) = \sum_i
a_it^i \in A[t]$ is a polynomial, we denote $f^\psi(t)\coloneqq \sum_i
\psi(a_i)t^i \in B[t]$ the polynomial obtained by applying $\psi$ to the
coefficients of $f(t)$.
\end{notation*} 

We now prove the decomposition theorem for Hecke polynomials. Recall
the following commutative diagram:
\[
\begin{tikzcd}
\Hecke_R(K_B,B) \ar[rrd,hook,bend left,"\Theta^B_Z"]\\
\Hecke_R(K_P,P) \ar[u,hook,"\varepsilon_{B,P}"] \ar[r,"\Theta^P_M"] &
\Hecke_R(K_M,M) \ar[r,hook,"\Satake^M"] & R[\Lambda]\\
\Hecke_R(K,G) \ar[u,hook,"\varepsilon_{P,G}"] \ar[rru,hook,bend right,"\Satake^G"]
\end{tikzcd}
\]

\begin{thm}[Decomposition Theorem]\label{thm:decomp} 
Assume that Hypothesis~\ref{hyp} holds true.
Let $d(t) \in \Hecke_R(K,G)[t]$ be a polynomial and assume that there is a
decomposition
\[
d^{\Satake^G}(t)= \wt f(t)\cdot \wt g(t)\qquad \text{in $R[\Lambda][t]$,}
\]
such that one of the following properties is satisfied:
\begin{enumerate}[label=(\alph*)]
\item\label{thm:decomp-a} $\wt f(t)$ has coefficients in
$(\Satake^M\circ\Theta^P_M)(C^+_P)$ with constant term $1$, or 
\item\label{thm:decomp-b} $\wt g(t)$ has coefficients in
$(\Satake^M\circ\Theta^P_M)(C^-_P)$ with
constant term $1$.
\end{enumerate}
Then there exist polynomials $f(t), g(t)\in \Hecke_R(K_P,P)[t]$ such that
\begin{alignat}{4}
\deg f(t) &= \deg \wt f(t), &\hspace{3em} f^{\Satake^M\circ\Theta^P_M}(t) &= \wt
f(t),\\
\deg g(t) &= \deg \wt g(t), & g^{\Satake^M\circ\Theta^P_M}(t) &= \wt g(t),
\end{alignat}
and
\[
d(t) = f(t)\cdot g(t)\quad \text{in $\Hecke_R(K_P,P)[t]$.}
\]
\end{thm} 
\begin{proof} 
Note that $\Theta^B_Z\big|_{\Hecke_R(K_P,P)} = \Satake^M\circ\Theta^P_M$ and
$\Theta^B_Z\big|_{\Hecke_R(K,G)} = \Satake^G$.

Assume that \ref{thm:decomp-a} is satisfied. The case \ref{thm:decomp-b} is
completely analogous. By Proposition~\ref{prop:hyp-equiv}
the restriction of $\Theta^P_M$ to $\OO^+_P$ is injective. As $\Satake^M$ is
injective by Theorem~\ref{thm:Satake}, it follows that the restriction of
$\Theta^B_Z$ to $\OO^+_P$ is injective. By the assumption there exists a unique
polynomial $f(t)\in C^+_P[t]$ satisfying $f^{\Theta^B_Z}(t) = \wt f(t)$. Its
constant term is necessarily $1$, and hence there exists a power series $h(t) =
\sum_{i=0}^\infty h_it^i \in C^+_P\llbracket t\rrbracket$ with $h(t)\cdot f(t) =
f(t)\cdot h(t) = 1$. Now, set
\[
g(t) \coloneqq h(t)\cdot d(t) \in \OO^+_P\llbracket t\rrbracket.
\]
Then
$g^{\Theta^B_Z}(t) = h^{\Theta^B_Z}(t)\cdot d^{\Theta^B_Z}(t) =
h^{\Theta^B_Z}(t)\cdot \wt f(t)\cdot \wt g(t) = \wt g(t)$.
As the restriction of $\Theta^B_Z$ to $\OO^+_P$ is injective, it follows that
$g(t)$ is a polynomial of degree $\deg \wt g(t)$. Moreover, we have
$f(t)\cdot g(t) = f(t)\cdot h(t)\cdot d(t) = d(t)$ in $\Hecke_R(K_P,P)[t]$.
\end{proof} 

\begin{rmk*} 
Suppose we are in the situation of Theorem~\ref{thm:decomp}. 
\begin{enumerate}[label=(\roman*)]
\item The polynomial $d^{\Satake^G_M}(t)$ decomposes in $\Hecke_R(K_M,M)[t]$.
\item In case~\ref{thm:decomp-a} the proof shows that one can choose $f(t)\in
C^+_P[t]$ and $g(t)\in \OO^+_P[t]$. Since, by Proposition~\ref{prop:hyp-equiv},
the restriction of $\Theta^P_M$ to $\OO^+_P$ is injective, it follows that
$f(t)$ and $g(t)$ are unique with these properties.
\item Similarly, in case~\ref{thm:decomp-b} one can choose $f(t)\in \OO^-_P[t]$
and $g(t)\in C^-_P[t]$, and both polynomials are unique with these properties.
\end{enumerate}
\end{rmk*} 

We draw some consequences of Theorem~\ref{thm:decomp}, cf. also \cite[Thm.~6.2
and Cor.]{Andrianov.1977}.

\begin{cor}\label{cor:decomp} 
Assume that Hypothesis~\ref{hyp} holds true.
\begin{enumerate}[label=(\roman*)]
\item\label{cor:decomp-i} Let $f(t)\in C^+_P[t]$ be a polynomial with constant
term $1$. Then there exists a polynomial $g(t)\in \Hecke_R(K_P,P)[t]$ with
constant term $1$ and $\deg g(t) \le \deg f(t)\cdot \bigl(\lvert W^M_0\rvert -
1\bigr)$ such that
\[
f(t)\cdot g(t) \in \Hecke_R(K,G)[t].
\]
\item\label{cor:decomp-ii} Let $X\in C^+_P$. Then there exists a monic
polynomial $d(t) = \sum_i d_it^i \in \Hecke_R(K,G)[t]$ such that $\deg d(t) \le
\lvert W^M_0\rvert$ and
\[
\sum_i X^id_i = 0.
\]
\end{enumerate}

\begin{enumerate}[label=(\roman*')]
\item\label{cor:decomp-i'} Let $g(t) \in C^-_P[t]$ be a polynomial with constant
term $1$. Then there exists a polynomial $f(t)\in \Hecke_R(K_P,P)[t]$ with
constant term $1$ and $\deg f(t) \le \deg g(t)\cdot \bigl(\lvert W^M_0\rvert -
1\bigr)$ such that
\[
f(t)\cdot g(t) \in \Hecke_R(K,G)[t].
\]
\item\label{cor:decomp-ii'} Let $X\in C^-_P$. Then there exists a monic
polynomial $d(t) = \sum_i d_it^i \in \Hecke_R(K,G)[t]$ such that $\deg d(t) \le
\lvert W^M_0\rvert$ and
\[
\sum_i d_iX^i = 0.
\]
\end{enumerate}
\end{cor} 
\begin{proof} 
Note that, using the anti-automorphism $\zeta_P$ on $\Hecke_R(K_P,P)$, which on
$\Hecke_R(K,G)$ restricts to $\zeta_G$ by Lemma~\ref{lem:emb-anti}, one can
easily deduce \ref{cor:decomp-i'} and \ref{cor:decomp-ii'} from
\ref{cor:decomp-i} and \ref{cor:decomp-ii}, respectively. \bigskip

Let us prove \ref{cor:decomp-i}. So let $f(t)$ be a polynomial with coefficients
in $C^+_P$ and constant term $1$. Write $f^{\Theta^B_Z}(t) \eqqcolon \wt f(t) =
\sum_i f_it^i \in R[\Lambda][t]$. Since $\Satake^M\circ\Theta^P_M =
\Theta^B_Z\circ \varepsilon_{B,P}$ by Lemma~\ref{lem:partial Satake}, the
coefficients of $\wt f(t)$ lie in $\Satake^M\bigl(\Hecke_R(K_M,M)\bigr)$. Hence,
the coefficients of $\wt f(t)$ are invariant under the twisted action of
$W_{0,M}$. Given $w\in W_0$, write $\wt f^w(t) = \sum_i (w\star f_i)\cdot t^i$.
The polynomial $\wt d(t) \coloneqq \prod_{w\in W^M_0} \wt f^w(t)$ is then
$W_0$-invariant with respect to the twisted action and hence has coefficients in
$\Satake^G\bigl(\Hecke_R(K,G)\bigr)$. Moreover, it factors as $\wt d(t) = \wt
f(t)\cdot \wt g(t)$ for some $\wt g(t) \in R[\Lambda][t]$ with constant term $1$
and $\wt g(t) \le \deg f(t)\cdot \bigl(\lvert W^M_0\rvert - 1\bigr)$.
The existence of the polynomial $g(t)$ with the desired properties now follows
from Theorem~\ref{thm:decomp}.

We now prove \ref{cor:decomp-ii}. Let $X\in C^+_P$. Applying \ref{cor:decomp-i}
to the polynomial $f(t) \coloneqq 1 - Xt$, we find a polynomial $g(t) =
\sum_{i=0}^{r-1} g_it^i \in \Hecke_R(K_P,P)[t]$ with $g_0 = 1$ and $r \le \lvert
W^M_0\rvert$ and such that
\[
f(t)\cdot g(t) \eqqcolon \sum_{i=0}^r d_{r-i}t^i \in \Hecke_R(K,G)[t].
\]
Since $f(t)\cdot g(t) = 1 + \sum_{i=1}^{r-1}(g_i-Xg_{i-1})t^i - Xg_{r-1}t^r$, a
comparison of coefficients shows that $d_0 = -Xg_{r-1}$, $d_i = g_{r-i} -
Xg_{r-(i+1)}$ for $1\le i\le r-1$, and $d_r = g_0 = 1$. Therefore, the
polynomial $\sum_{i=0}^r d_it^i \in \Hecke_R(K,G)[t]$ is monic, of degree $r\le
\lvert W^M_0\rvert$, and satisfies
\[
\sum_{i=0}^r X^id_i = -Xg_{r-1} + \sum_{i=1}^{r-1} \bigl(X^ig_{r-i} -
X^{i+1}g_{r-(i+1)}\bigr) + X^r = 0. \qedhere
\]
\end{proof} 

\begin{ex} 
We continue Example~\ref{ex:neg-pow} and apply Theorem~\ref{thm:decomp} to the
polynomial
\[
\chi_{a_B}(t) = 1 -q^{-1}\cdot (X_+Y^{-1} + X_-)\cdot t + q^{-1}Y^{-1}t^2 \in
\Hecke_R(K,G)[t].
\]
Then $\chi_{a_B}^{\Satake^G}(t) = 1 - \bigl((qy)^{-1} + x^{-1}\bigr)\cdot t +
(qxy)^{-1}t^2$ decomposes in $R[x^{\pm1},y^{\pm1}][t]$ as
$\wt f(t)\cdot \wt g(t)$ with $\wt f(t) = 1- (qy)^{-1}t$ and $\wt g(t) = 1 -
x^{-1}t$. Then $f(t) \coloneqq 1 - q^{-1}X_+Y^{-1}t \in C^+_P[t]$ is the unique
polynomial with $f^{\Theta^B_Z}(t) = \wt f(t)$. Let $h(t) \coloneqq
\sum_{i=0}^\infty (q^{-1}X_+Y^{-1})^it^i \in C^+_B\llbracket t\rrbracket$ be the
inverse power series. Then
\begin{align*}
g(t)&\coloneqq h(t)\cdot \chi_{a_B}(t)\\
&= \underbrace{h(t) - h(t)\cdot q^{-1}X_+Y^{-1}t}_{=1} - h(t) q^{-1}X_-t +
h(t)\cdot q^{-1}Y^{-1}t^2\\
&= 1 - q^{-1}X_-t - \sum_{i=1}^\infty (q^{-1}X_+Y^{-1})^{i-1}\cdot
q^{-1}Y^{-1}t^{i+1} +
\sum_{i=0}^\infty (q^{-1}X_+Y^{-1})^i\cdot q^{-1}Y^{-1} t^{i+2}\\
&= 1 - q^{-1}X_-t.
\end{align*}
Hence, $\chi_{a_B}(t)$ decomposes in $\Hecke_R(K_B,B)[t]$ as
\[
\chi_{a_B}(t) = (1- q^{-1}X_+Y^{-1}t)\cdot (1 - q^{-1}X_-t).
\]
\end{ex} 

\appendix
\counterwithin{prop}{section}
\section{Proof of Theorem~\ref{thm:example}}\label{appendix:H(B)} 
Let $B\subseteq \GL_2(\field)$ be the subgroup of upper triangular matrices and
let $K_B$ be the subgroup of $B$ with entries in $\O_{\field}$. Fix a
uniformizer $\pi\in \O_{\field}$ and recall that $q$ denotes the cardinality of
the residue field of $\field$. Let $R$ be a coefficient ring. 

We describe the $R$-algebra $\Hecke_R(K_B,B) = R\otimes_\Z \Hecke(K_B,B)$ in
terms of generators and relations. As $R\otimes_\Z-$ is right exact, we may
reduce to the case $R = \Z$.

First, we need to understand the double cosets of $B$ with respect to $K_B$. Let
$A$ be a complete system of representatives for $\O_{\field}/(\pi)$ with $0\in
A$. Then $A_B \coloneqq \set{\sum_{i=1}^n a_i\pi^{-i}}{n\in \Z_{>0}, a_i\in A}$
is a complete system of representatives for $\field/\O_{\field}$.

\begin{lem}\label{lem:B-double-coset} 
As a set $B$ decomposes as
\[
B = \bigsqcup_{\substack{a,b,c\in\Z,\\ b\le \min\{a,c\}} } K_B
\triang{\pi^a}{\pi^b}{\pi^c} K_B,
\]
and for all $a,b,c\in \Z$ with $b\le \min\{a,c\}$ one has the decomposition
\[
K_B \triang{\pi^a}{\pi^b}{\pi^c} K_B = \begin{cases}
\bigsqcup_{\substack{\beta\in A_B\pi^c,\\ \val_{\field}(\beta) = b}} 
K_B \triang{\pi^a}{\beta}{\pi^c}, & \text{if $b<\min\{a,c\}$;}\\
K_B \triang{\pi^a}{0}{\pi^c}K_B = \bigsqcup_{\substack{\beta\in A_B\pi^c,\\
\val_{\field}(\beta)\ge a}} K_B \triang{\pi^a}{\beta}{\pi^c}, & \text{if $b =
\min\{a,c\}$.}
\end{cases}
\]
\end{lem} 
\begin{proof} 
Let $\triang{\alpha}{\beta}{\gamma} \in B$. Write $\alpha = \alpha_0\pi^a$ and 
$\gamma = \gamma_0\pi^c$, where $\alpha_0,\gamma_0\in \O_{\field}^\times$,
$a,c\in \Z$, and write $\frac{\beta}{\alpha_0\pi^c} = \beta' + x$, with
$\beta'\in A_B$ and $x\in \O_{\field}$. Then $\beta' = 0$ or
$\val_{\field}(\beta) = \val_{\field}(\beta'\pi^c)$. Moreover,
\[
\triang{\alpha}{\beta}{\gamma} = \triang{\alpha_0}{\alpha_0x}{\gamma_0}\cdot
\triang{\pi^a}{\beta'\pi^c}{\pi^c} \in K_B \triang{\pi^a}{\beta'\pi^c}{\pi^c}.
\]
Let now $a,a',c,c'\in \Z$ and $\beta,\beta'\in A_B$ with
\[
\triang{\pi^a}{\beta\pi^c}{\pi^c}\cdot
\triang{\pi^{a'}}{\beta'\pi^{c'}}{\pi^{c'}}^{-1} = \triang{\pi^{a-a'}}{\beta
\pi^{c-c'} - \beta'\pi^{a-a'}}{\pi^{c-c'}} \in K_B.
\]
Then $a=a'$ and $c=c'$, and then $\beta -\beta'\in \O_{\field}$, that is, $\beta
= \beta'$. This shows that $B$ is the disjoint union of the right cosets
$K_B\triang{\pi^a}{\beta\pi^c}{\pi^c}$, where $a,c\in \Z$ and $\beta\in A_B$.

Let now $a,c\in \Z$. Take any $0\neq \beta = \beta_0 \pi^{\val_{\field}(\beta)}
\in A_B\pi^c$ with $\beta_0\in \O_{\field}^\times$. If $\val_{\field}(\beta)
<\min\{a,c\}$, then
\[
\triang{\pi^a}{\beta}{\pi^c} = \triang{1}{0}{\beta_0^{-1}} \cdot
\triang{\pi^a}{\pi^{\val_{\field}(\beta)}}{\pi^c} \cdot \triang{1}{0}{\beta_0}
\in K_B \triang{\pi^a}{\pi^{\val_{\field}(\beta)}}{\pi^c} K_B.
\]
If $\val_{\field}(\beta) \ge \min\{a,c\}$, then $\val_{\field}(\beta) \ge a$,
because $\val_{\field}(\beta) <c$ always holds true. Hence,
\[
\triang{\pi^a}{\beta}{\pi^c} = \triang{\pi^a}{0}{\pi^c}\cdot
\triang{1}{\beta\pi^{-a}}{1} \in K_B \triang{\pi^a}{0}{\pi^c}K_B.
\]
The lemma follows.
\end{proof} 

We now prove Theorem~\ref{thm:example}. Let $\Hecke$ be the $\Z$-algebra
generated by the variables $X_+$, $X_-$, $Y$, and $Y^{-1}$, subject to the
relations~\eqref{eq:ex-H(B)}. This means that $Y$ is central and invertible and
we have $X_+X_- = q\cdot 1$ in $\Hecke$. We show that
\begin{align*}
\rho\colon \Hecke &\longrightarrow \Hecke(K_B,B),\\
X_+ &\longmapsto (\triang{\pi}{0}{1})_{K_B},\\
X_- &\longmapsto (\triang{\pi^{-1}}{0}{1})_{K_B},\\
Y &\longmapsto (\pi E_2)_{K_B}
\end{align*}
gives a well-defined isomorphism of algebras. 

It is clear that $(K_B)$ is the unit in $\Hecke(K_B,B)$, and that
$(\pi E_2)_{K_B}$ is central and invertible with inverse $(\pi^{-1}E_2)_{K_B}$.
In addition, using Lemma~\ref{lem:B-double-coset}, we compute
\begin{align*}
(\triang{\pi}{0}{1})_{K_B}\cdot (\triang{\pi^{-1}}{0}{1})_{K_B} &= (K_B
\triang{\pi}{0}{1})\cdot \sum_{\beta\in A}
(K_B\triang{\pi^{-1}}{\beta\pi^{-1}}{1})\\
&= \sum_{\beta\in A} (K_B\triang{1}{\beta}{1}) = q\cdot (K_B).
\end{align*}
Therefore, $\rho$ is a well-defined algebra homomorphism.

Observe that $\rho(X_+^m) = (\triang{\pi^m}{0}{1})_{K_B}$ and $\rho(Y^k) =
(\pi^kE_2)_{K_B}$, for all $m\in\Z_{\ge0}$ and $k\in \Z$. For each
$n\in\Z_{\ge0}$ we compute
\begin{align*}
\rho(X_-^n) &= \rho(X_-)^n = \left(\sum_{\beta\in A}
(K_B\triang{\pi^{-1}}{\beta\pi^{-1}}{1})\right)^n\\
&= \sum_{\beta_1,\dotsc,\beta_n\in A} (K_B\triang{\pi^{-n}}{\sum_{i=1}^n
\beta_i\pi^{-i}}{1}) = (\triang{\pi^{-n}}{0}{1})_{K_B}.
\end{align*}

Given $m,n\in \Z_{\ge0}$ and $k\in \Z$, this shows
\begin{align*}
\rho(X_-^nX_+^mY^k) &= (\triang{\pi^{-n}}{0}{1})_{K_B}\cdot
(\triang{\pi^m}{0}{1})_{K_B}\cdot (\pi^kE_2)_{K_B}\\
&= \sum_{\beta_1,\dotsc,\beta_n\in A} (K_B\triang{\pi^{m+k-n}}
{\pi^k\sum_{i=1}^n \beta_i\pi^{-i}}{\pi^k})\\
&= \sum_{b=k-n}^{\min\{k,m+k-n\}} (\triang{\pi^{m+k-n}}{\pi^b}{\pi^k})_{K_B}.
\end{align*}
For $m,n\in\Z_{\ge1}$ and $k\in \Z$ this implies
\[
\rho\bigl(X_-^nX_+^mY^k - X_-^{n-1}X_+^{m-1}Y^k\bigr) =
(\triang{\pi^{m+k-n}}{\pi^{k-n}}{\pi^k})_{K_B}.
\]
Now notice that
\begin{equation}\label{eq:rho-basis}
\begin{split}
\set{X_+^mY^k}{m\in\Z_{\ge0},k\in\Z} &\cup \set{X_-^nY^k}{n\in\Z_{\ge1}, k\in
\Z}\\
&\cup \set{X_-^nX_+^mY^k - X_-^{n-1}X_+^{m-1}Y^k}{m,n\in\Z_{\ge1}, k\in\Z}
\end{split}
\end{equation}
generates $\Hecke$ as a $\Z$-module and identifies via $\rho$ with the
double coset basis of $\Hecke(K_B,B)$. It follows that~\eqref{eq:rho-basis}
is in fact $\Z$-linearly independent, hence a $\Z$-basis of $\Hecke$.
Consequently, $\rho$ is an isomorphism.

\bibliographystyle{alphaurl}
\bibliography{../references}{}
\end{document}